\documentclass{amsart}
\usepackage{amsmath,amsthm,amssymb,bm}
\usepackage[dvipsnames]{xcolor}
 \usepackage{indentfirst}
\usepackage{graphicx}
\usepackage{tikz}
\usetikzlibrary{decorations.pathreplacing, decorations.pathmorphing, decorations.shapes}
\usetikzlibrary{arrows,calc}
\usepackage{mathrsfs,esint}
\usepackage{comment}
\usepackage{enumerate}
\usepackage{enumitem}
\usepackage{hyperref}    
\hypersetup{
     colorlinks = true,
     linkcolor = red,
     anchorcolor =green!50!black,
     citecolor = green!50!black,
     filecolor =green!50!black,
     urlcolor = green!50!black
}
\usepackage{cleveref}

\usepackage{float}
\usepackage[ruled, vlined]{algorithm2e}

\usepackage{subcaption}
\usepackage{pgfplots}
\pgfplotsset{compat=newest, compat/show suggested version=false}
\usepackage{tikz}
\usetikzlibrary{intersections}
\usepackage{tikz-3dplot}
\tdplotsetmaincoords{60}{20}

\newtheorem{theorem}{Theorem}[section]
\newtheorem{proposition}[theorem]{Proposition}

\newtheorem{lemma}[theorem]{Lemma}
\newtheorem{remark}[theorem]{\textit{Remark}}
\newtheorem{definition}[theorem]{Definition}
\newtheorem{assumption}[theorem]{Assumption}

\crefname{assumption}{Assumption}{Assumptions}
\crefname{problem}{Problem}{Problems}

\newcommand{\beq}{\begin{equation}}
\newcommand{\eeq}{\end{equation}}
\newcommand{\sgn}{\operatorname{sgn}}

\newcommand{\supp}{\text{supp}\,}
\newcommand{\divv}{\mathrm{div}\,}

\def\cD{\mathcal{D}}

\def\cF{\mathcal{F}}
\def\cG{\mathcal{G}}
\def\cT{\mathcal{T}}
\def\cH{\mathcal{H}}

\def\cP{\mathcal{P}}

\def\cS{\mathcal{S}}
\def\cT{\mathcal{T}}

\def\C{\mathbb{C}}

\def\I{\mathbb{I}}

\def\N{\mathbb{N}}

\def\R{\mathbb{R}}
\def\S{\mathbb{S}}

\def\sS{\mathscr{S}}

\def\fS{\mathfrak{S}}

\def\eb{\bm e}

\def\pb{\bm p}

\def\vb{\bm v}
\def\wb{\bm w}
\def\xb{\bm x}
\def\yb{\bm y}
\def\zb{\bm z}

\def\Pb{\bm P}
\def\nub{\bm\nu}
\def\xib{\bm\xi}
\def\phib{\bm\phi}
\def\etab{\bm\eta}

\def\mub{\bm\mu}
\def\psib{\bm\psi}

\def\lab{\bm\lambda}

\newcommand{\al}{\alpha}

\newcommand{\del}{\delta}
\newcommand{\ep}{\epsilon}

\newcommand{\sig}{\sigma}
\newcommand{\om}{\omega}

\newcommand{\Ga}{{\Gamma}}

\newcommand{\Sig}{\Sigma}

\newcommand{\Om}{\Omega}

\title[Compactness for nonlocal Dirichlet energy with applications]{Compactness results for a Dirichlet energy of nonlocal gradient with applications}

\author{Zhaolong Han}
\address{Department of Mathematics, University of California, San Diego, CA 92093, United States} 
\email{zhhan@ucsd.edu}
\author{Tadele Mengesha}
\address{Department of Mathematics, University of Tennessee, Knoxville, TN 37996, United States} 
\email{mengesha@utk.edu}
\author{Xiaochuan Tian}
\address{Department of Mathematics, University of California, San Diego, CA 92093, United States} 
\email{xctian@ucsd.edu}
\thanks{ZH and XT's research is supported in part by NSF grants DMS-2111608 and DMS-2240180. TM's research is supported by NSF grant DMS-2206252. 
The early stages of this project were conducted during a SQuaRE at the American Institute for Mathematics (AIM) 
titled {\it Variational methods for multiscale and nonlinear nonlocal models with applications to peridynamics in San Jose, CA, in May 2023.}
The authors thank AIM for providing a supportive and mathematically rich environment. }
\date{}

\begin{document}

\begin{abstract}
    We prove two compactness results for function spaces with finite Dirichlet energy of half-space nonlocal gradients. In each of these results, we provide sufficient conditions on a sequence of kernel functions that guarantee the asymptotic compact embedding of the associated nonlocal function spaces into the class of square-integrable functions. Moreover, we will demonstrate that the sequence of nonlocal function spaces converges in an appropriate sense to a limiting function space. As an application, we prove uniform Poincar\'e-type inequalities for sequence of half-space gradient operators.  We also apply the compactness result to demonstrate the convergence of appropriately parameterized nonlocal heterogeneous anisotropic diffusion problems. We will construct asymptotically compatible schemes for these type of problems. Another application concerns the convergence and robust discretization of a nonlocal optimal control problem.
\end{abstract}


\maketitle

\section{Introduction}
In this paper, we revisit the half-space nonlocal gradient operator studied in \cite{han2023nonlocal} defined in the principal value sense as 
\[
\mathcal{G}^{\nub}_{w}u(\xb) = \lim_{\epsilon \to 0}\int_{\cH_{\nub} \setminus B_{\epsilon}(\bm 0)} \frac{\zb}{|\zb|} (u(\xb + \zb)-u(\xb))w(\zb)d\zb,\quad\xb\in \R^d
\]
where for $d\geq 1$,  $u:\mathbb{R}^d \to \mathbb{R}$ is a measurable function, $\nub \in \mathbb{R}^{d}$ is a fixed unit vector, $\cH_{\nub}$ is the half-space $\cH_{\nub} = \{\zb\in \mathbb{R}^{d}: \zb\cdot \nub \geq 0\}$, and $w$ is a nonnegative kernel function that satisfies appropriate conditions specified later.
Additionally, inspired by the integration by parts formula, a distributional nonlocal gradient $\mathfrak{G}_w^{\nub}$ is introduced (see \cref{eq:distributional_gradient}), and the related nonlocal Sobolev-type space $\mathcal{S}_{w}^{\nub}(\Omega)$, defined by 
\[
\mathcal{S}_{w}^{\nub}(\Omega) = \{u\in L^2(\mathbb{R}^{d}): \mathfrak{G}^{\nub}_w u\in L^2(\R^d;\R^d), u=0\text{ a.e. in }\mathbb{R}^{d}\setminus\Omega\}
\]
for an open domain $\Omega\subset \R^d$ is studied in \cite{han2023nonlocal}.

The study of nonlocal operators is largely motivated by peridynamics, introduced by Silling \cite{silling2000reformulation}, which has been effective in modeling materials undergoing large strains and fractures. Within the peridynamics theory, a specific category of material models known as correspondence models employs nonlocal gradients and classical stress-strain relationships to calculate the forces between particles \cite{foster2010viscoplasticity,SEWX07,tupek2014extended,warren2009nonordinary}. Despite their effectiveness, correspondence models are challenged by material instability issues \cite{silling2017stability}, which, however, can be addressed through the use of half-space gradient operators. The exploration of such operators began with the work of Lee and Du \cite{lee2020nonlocal} in the context of periodic domains. Subsequent developments, carried out by Han and Tian \cite{han2023nonlocal}, extended the study to general domains in \(\R^d\), where they showed a nonlocal Poincar\'e inequality on $\mathcal{S}_{w}^{\nub}(\Omega)$ with $\Omega$ being a bounded domain with continuous boundary. The well-posedness of nonlocal diffusion problems and peridynamics correspondence models using half-space nonlocal operators is therefore implied via the application Lax-Milgram theorem. For additional studies related to the nonlocal gradient, please refer to \cite{arroyo2022functional,bellido2023nonlocal,bellido2024nonlocal,cueto2024gamma,delia2022connections,mengesha2015localization,mengesha2016characterization,shieh2015new,shieh2018new} and the references therein.

This work primarily aims to analyze the convergence of variational problems associated with the nonlocal gradient $\mathfrak{G}_w^{\nub}$ and the energy space $\mathcal{S}_{w}^{\nub}(\Omega)$. In particular, the focus is on understanding the convergence of solutions when the variational problems are associated with a sequence of kernels $w_n$ that approaches a specific limit as $n$ approaches infinity. The key for showing such convergence relies on establishing compactness results in the limit as $n$ approaches infinity in the spirit of Bourgain, Brezis and Mironescu \cite{bourgain2001another}.  In this work, 
we prove two types of compactness results, each associated with a particular type of kernel sequence convergence. One type of kernel convergence, as detailed in \Cref{ass:kernel4cptthm}, is fundamentally characterized by the convergence of $\min(1, |\zb|)w_n(\zb)$ to a measure with atomic mass at the origin as $n\to\infty$ (or as $\delta\to 0$ in the context of \Cref{ass:kernel4cptthm}). Such a result is crucial for establishing the convergence of solutions from nonlocal variational problems to their local counterparts as nonlocality diminishes. A different kind of kernel convergence presupposes $\{ w_n \}_n$ is a sequence of nonnegative  radial kernels converges increasingly to a kernel $w$ almost everywhere, i.e., $0\le w_n\nearrow w$ a.e., where $w$ is nonintegrable near the origin (for detailed assumptions, refer to \Cref{assu:nonint_kernel}), resulting in nonlocal-to-nonlocal convergence. A key intermediate result for this form of compactness is the locally compact embedding of $\mathcal{S}_{w}^{\nub}(\R^d)$ into $L^2(\R^d)$, which extends the classical compact embedding of $H^s$ ($s\in (0,1)$) into $L^2$ by choosing $w(\zb)$ as the Riesz fractional kernel $C/|\zb|^{d+s}$ \cite{bellido2023nonlocal,shieh2015new,shieh2018new}. Furthermore, the compactness results yield Poincar\'e inequalities that remain uniform regardless of the size of $n$. The uniform Poincar\'e inequalities also serve as essential intermediary steps in showing the convergence of variational solutions.  

To establish the compactness results, we rely on a variant of the Riesz-Kolmogorov-Fr\'echet compactness criterion \cite{mengesha2012nonlocal}. Addressing nonlocal Dirichlet energies with nonlocal gradients presents more challenges compared to the ``double-integral'' type nonlocal energies studies in previous works, for example, \cite{bourgain2001another,mengesha2014nonlocal}. In recent studies \cite{bellido2023nonlocal,cueto2023variational}, compactness results for nonlocal energies with truncated Riesz fractional gradients are established where a nonlocal version of the fundamental theorem of calculus plays a pivotal role. However, none of the existing methodologies are suitable for addressing the challenges posed by nonlocal energies defined through a half-space gradient. In our study, to show the Riesz-Kolmogorov-Fr\'echet compactness criterion, we employ Fourier analysis along with fine estimates (see \Cref{sec:Op_revisited}) of the Fourier symbols for half-space nonlocal gradient operators.

The compactness results are applied to obtain convergence of parameterized nonlocal heterogeneous anisotropic diffusion problems, together with the asymptotic compatibility of their Galerkin approximations with appropriate finite dimensional subspaces. The concept of asymptotic compatibility is introduced by \cite{tian2014asymptotically} and is aimed for robust numerical discretization of parameterized nonlocal variational problems as the nonlocal modeling parameter and the discretization parameter both go to an asymptotic limit. We summarize the main convergence theorem in \cite{tian2014asymptotically} with a slight relaxation on the conditions, i.e., the assumption that strong convergence of operators acting on a dense subset is replaced by the convergence of bilinear forms in a weak sense. It is worth noting that in the case of $0\le w_n\nearrow w$ for a sufficiently singular kernel $w$,  the Galerkin approximation to the nonlocal problem associated with $w_n$ serves as a nonconforming Discontinuous Galerkin (DG) scheme for the nonlocal problem tied to $w$, as discussed in \cite{tian2015nonconforming}. Another application we discuss is an optimal control problem with nonlocal diffusion equation as the constraint. The approach we follow and the results we obtain for this part of the study parallel  \cite{mengesha2023optimal} an optimal control problem is analyzed where the constraint is the linear bond-based peridynamics model. 

The paper is organized as follows. In \Cref{sec:Op_revisited}, we recall the notion and some basic properties of nonlocal half-space operators studied in \cite{han2023nonlocal}. Further, the nonlocal function spaces are reformulated using the Fourier symbol of the nonlocal gradient, with a crucial lower bound estimate for the symbol presented. The major compactness results are established in \Cref{sec:cpt}, utilizing the symbol estimates. Applications are demonstrated through the proof of convergence for nonlocal diffusion problems and the establishment of asymptotic compatibility in Section \ref{sec:conv_and_AC}. Additionally, the study of the optimal control problem is undertaken in Section \ref{sec:optimal_control}. Finally, we conclude in \Cref{sec:conclusion}.

\section{Revisiting Nonlocal Half-space Operators and associated function spaces}\label{sec:Op_revisited}
\subsection{Nonlocal Half-space Operators}
In this section, we revisit definitions of nonlocal half-space vector operators as introduced in \cite{han2023nonlocal} and present some new properties.

Let $\nub\in\R^d$ be a fixed unit vector. Denote by $\chi_{\nub}(\zb)$ the characteristic function of the half-space $\mathcal{H}_{\nub}:=\{\zb\in \mathbb{R}^d:\zb\cdot\nub\geq 0\}$ parameterized by the unit vector $\nub$. Throughout this paper, we assume that $w$ satisfies the following conditions.
\beq
\label{eq:kernelassumption}
\left\{
\begin{aligned}
& w\in L^1_{\mathrm{loc}}(\R^d\backslash \{\bm 0\}),\ w\geq 0 \ \text{a.e.}, \ w\ \text{is radial}; \\
&   M_w^1:=\int_{|\zb|\le 1}|\zb|w(\zb)d\zb\in (0, \infty) \quad \textrm{and}\quad M_w^2:= \int_{|\zb|>1} w(\zb)d\zb<\infty.\\
\end{aligned}
\right.
\eeq
We notice that by assumption \eqref{eq:kernelassumption}  
there exists $\ep_0>0$ such that 
\begin{equation}\label{eq:epsilon_0}
0< \int_{\ep_0<|\zb|<1}w(\zb)d\zb <\infty. 
\end{equation}

We recall that the nonlocal half-space  gradient, divergence, and curl operators are defined in the principal value sense (see \cite[Definition~2.1]{han2023nonlocal}). For smooth functions $u\in C^1_c(\R^d)$ and $\vb\in C^1_c(\R^d;\R^d)$, the operators are meaningful pointwise  and are given as (\cite[Lemma~2.1]{han2023nonlocal})
\begin{equation}\label{hbgrad}
    \mathcal{G}^{\nub}_w u(\xb)=\int_{\R^d}\chi_{\nub}(\yb-\xb)\frac{\yb-\xb}{|\yb-\xb|}(u(\yb)-u(\xb))w(\yb-\xb)d\yb,\quad \xb\in \R^d,
\end{equation}
\begin{equation}\label{hbdiv}
    \mathcal{D}_w^{\nub} \vb(\xb)=\int_{\R^d}\chi_{\nub}(\yb-\xb)\frac{\yb-\xb}{|\yb-\xb|}\cdot (\vb(\yb)-\vb(\xb))w(\yb-\xb)d\yb,\quad\xb\in \R^d,
\end{equation}
and for $p\in [1,\infty]$, $\mathcal{G}_w^{\nub} u \in L^p(\R^d;\R^d)$, $\mathcal{D}_w^{\nub} \vb\in L^p(\R^d)$ with the estimate that for a constant $C$ depending on $d$ and $p$, 
\begin{equation}\label{estgradonwholespace}
    \|\mathcal{G}_w^{\nub} u\|_{L^p(\R^d;\R^d)}\le C \left(M_w^1\|\nabla u\|_{L^p(\R^d;\R^d)} + M_w^2 \|u\|_{L^p(\R^d)}\right), \,\textrm{and}
\end{equation}
\begin{equation}\label{estdivonwholespace}
    \|\mathcal{D}_w^{\nub} \vb\|_{L^p(\R^d)} \le C \left(M_w^1\|\nabla \vb\|_{L^p(\R^d;\R^{d\times d})} + M_w^2 \|\vb\|_{L^p(\R^d;\R^d)}\right).
\end{equation}
In addition, the nonlocal gradient and nonlocal divergence are related via a nonlocal integration by parts formula 
\[
\int_{\mathbb{R}^{d}} \mathcal{G}^{\nub}_{w} u(\xb)\cdot {\vb}(\xb) d \xb = -\int_{\mathbb{R}^{d}} u(\xb)\mathcal{D}^{-\nub}_{w} {\vb}(\xb) d\xb,
\]
for any $u\in C^1_c(\R^d)$ and $\vb\in C^1_c(\R^d;\R^d)$. In fact, as shown in \cite[Proposition~2.4]{han2023nonlocal}, the above formula is valid for $u\in L^{1}(\mathbb{R}^{d})$ such that $w(\xb-\yb)|u(\xb) -u(\yb)| \in L^{1}(\mathbb{R}^d\times \mathbb{R}^d)$. 
This formula motivates the introduction of the distributional nonlocal gradient operator $\mathfrak{G}_w^{\nub} u\in (C^\infty_c(\R^d;\R^{d}))'$ defined  for $u\in L^p(\R^d)$, $1\le p\le\infty$, as 
\begin{equation}
\label{eq:distributional_gradient}
    \langle \mathfrak{G}_w^{\nub} u,\bm\phi\rangle:=-\int_{\R^d} u(\xb)\cdot\mathcal{D}_w^{-\nub}\bm\phi(\xb)d\xb ,\quad\forall \bm\phi\in C_c^\infty(\R^d;\R^{d}), 
\end{equation}
(see \cite[Definition~2.2]{han2023nonlocal}).

It has also been demonstrated in \cite{lee2020nonlocal} that for $u\in C^\infty_c(\mathbb{R}^d)$, $\vb\in C^\infty_c(\mathbb{R}^d;\R^{d})$, and $\xib\in\R^d$
\begin{equation}\label{eq:fourierhbgrad}
    \mathcal{F}(\mathcal{G}_w^{\nub} u)(\bm\xi)=\bm\lambda_w^{\nub}(\bm\xi) \hat{u}(\bm\xi), \text{ and }
\end{equation}
\begin{equation}\label{eq:fourierhbdiv}
    \mathcal{F}(\mathcal{D}^{\nub}_w  \vb)(\bm\xi)=\bm\lambda_w^{\nub}(\bm\xi)^T\hat{\vb}(\bm\xi)
\end{equation} where the Fourier symbol $\bm\lambda_w^{\nub}(\bm\xi)$ is given by 
\begin{equation}\label{deffouriersymboldiv}
    \bm\lambda_w^{\nub}(\bm\xi):=\int_{\mathbb{R}^d}\chi_{\nub}(\zb)\frac{\zb}{|\zb|}w(\zb)(e^{2\pi i\bm\xi\cdot \zb}-1)d\zb.
\end{equation}
In the above, and hereafter, $\mathcal{F}$ and $\mathcal{F}^{-1}$ represent the Fourier transform and the inverse Fourier transform, respectively. We will also use  the $\hat{\cdot}$-notation for Fourier transform interchangeably. That is, 
$\mathcal{F}(u) = \hat{u}$.  

By \cite[Lemma~2.4]{han2023nonlocal}, the Fourier symbol $\bm\lambda_w^{\nub}(\bm\xi)$ is bounded by a linear function, i.e.,
\begin{equation}\label{Fsym0thorderest}
    |\bm\lambda_w^{\nub}(\xib)|\le 2\sqrt{2}\pi M_w^1|\xib|+\sqrt{2}M_w^2,\quad\xib\in\R^d.
\end{equation}
In fact, in  \cite[Lemma~5.1]{han2023nonlocal}, the following estimates on the behavior of $\bm\lambda_w^{\nub}(\xib)$ near the origin and at infinity are proved. It is worth noting that in \cite{han2023nonlocal}, it is assumed that the kernel does not vanish in a neighborhood of the origin. Nevertheless, the arguments presented in that work remain valid with this assumption replaced by $w\neq 0\in  L^1_{\mathrm{loc}}(\R^d\backslash \{\bm 0\})$, in particular $w$ satisfying \eqref{eq:kernelassumption}. 

\begin{lemma}\label{lem:labbdd_original}
    Let \(w\) satisfy \eqref{eq:kernelassumption}. 
    \begin{enumerate}
        \item For any \(N>0\) there exists a constant \(C=C(N,w,d)>0\) such that 
        \[|\bm\lambda^{\nub}_{w}(\xib)|\ge C\int_{\R^d}\min(1,w(\xb))d\xb,\quad\forall |\xib|\ge N .\]
        \item Suppose that in addition \(w\) has finite first moment, i.e., \(\int_{\R^d} |\zb|w(\zb)d\zb <\infty\).
        Then there exist constants $N_1=N_1(w,d)\in (0,1)$ and $C_1=C_1(w,d)>0$ such that \[|\bm\lambda^{\nub}_{w}(\xib)|\ge C_1|\xib|,\quad\forall |\xib|\le N_1. \]
    \end{enumerate}
\end{lemma}
For the compactness result we obtain below, we will need an improved lower bound of \(\lab^{\nub}_w (\xib)\) for large $\xib$ than provided in the first part of the above lemma.
The following lemma has this improved lower bound whose proof will be postponed to \Cref{subsec:Appendix1} since it is long and technical.
\begin{lemma}\label{lem:lablowerbdd_xilarge}
    Let \(w\) satisfy \eqref{eq:kernelassumption}. Assume in addition that \(w\) is nonincreasing if \(d=1\). Then for any \(N>0\) and $\ep>0$, there exists a constant \(C=C(N\ep,d)>0\) such that 
    \[|\lab^{\nub}_w(\xib)|\ge|\Re(\lab^{\nub}_{w})(\xib)|\ge C\int_{|\zb|>\frac{N\ep}{|\xib|}} w(\zb)d\zb,\quad\forall |\xib|>N.\]
    Here $\Re(\lab^{\nub}_{w})$ stands for the real part of $\lab^{\nub}_{w}$.
\end{lemma}
\subsection{Nonlocal Function Spaces via Fourier symbols}\label{sec:nlcl_fct_sp}
In this section, we introduce a function space via the Fourier transform formula of the nonlocal gradient operator defined in the previous section. We will show that this function space is a Hilbert space and that it is the same space as the space of function with square integrable distributional nonlocal gradient that was defined and thoroughly studied in \cite{han2023nonlocal}.

Define the function space $\cH^{\nub}_w(\R^d)$ as
\begin{equation}
    \cH^{\nub}_w(\R^d):=\left\{u\in L^2(\R^d):|u|_{\cH^{\nub}_w(\R^d)}^{2}:=\int_{\R^d}|\bm\lambda_w^{\nub}(\xib)\hat{u}(\xib)|^2d\xib<\infty\right\}
\end{equation}
equipped with norm \[\|u\|_{\cH^{\nub}_w(\R^d)}:=\left( \|u\|_{L^2(\R^d)}^2+|u|_{\cH^{\nub}_w(\R^d)}^2  \right)^{\frac{1}{2}}\]
It is direct to check that $\cH^{\nub}_w(\R^d)$ is a normed vector space. Moreover, with the inner product given by $( u, v)_{\cH^{\nub}_w(\R^d)}  = (u, v)_{L^{2}(\mathbb{R}^d)} + (\lab^{\nub}_w \hat{u}, \lab^{\nub}_w \hat{v})_{L^{2}(\mathbb{R}^d; \mathbb{C}^{d})}$ for any $u, v \in \cH^{\nub}_w(\R^d)$,  $\cH^{\nub}_w(\R^d)$ is, in fact, a Hilbert space. Before we prove this statement we make the following remark. 

Let $\sS(\R^d;\C)$ be the space of Schwartz functions and $\sS'(\R^d;\C^d)$ be the space of tempered distribution. Since $u\in L^2(\R^d)$, $\hat{u}\in L^2(\R^d;\C)$. Moreover, we may define the product $\bm\lambda^{\nub}_w\hat{u}$ as a tempered distribution in the following way: 
    \[\langle \lab^{\nub}_w \hat{u}, \phib\rangle := \left(\lab^{\nub}_w\hat{u},\phib\right)_{L^2(\R^d;\C^d)}= \left(\hat{u},\left(\lab^{\nub}_w \right)^T \phib\right)_{L^2(\R^d;\C)},\quad \forall \phib\in\sS(\R^d;\C^d).\] 
    and use  \cref{Fsym0thorderest} to show that $\bm\lambda^{\nub}_w\hat{u}\in \sS'(\R^d;\C^d)$. We will use this observation to prove the completeness of $\cH^{\nub}_w(\R^d)$. 
\begin{proposition}\label{prop:hilbert}
    $\cH^{\nub}_w(\R^d)$ is a Hilbert space. 
\end{proposition}
\begin{proof}
Let $\{u_n\}_{n=1}^\infty$ be a Cauchy sequence in $\cH^{\nub}_w(\R^d)$. Then there exist $u\in L^2(\R^d)$ and $\vb\in L^2(\R^d;\mathbb{C}^d)$ such that $u_n\to u$ and $\bm\lambda^{\nub}_w\hat{u}_n\to \vb$ in $L^2(\R^d;\C^d)$. 
Then for any $\phib\in\sS(\R^d;\C^d)$, \[\begin{split}
        \langle \bm\lambda^{\nub}_w\hat{u},\phib\rangle&=\left(\hat{u},\left( \bm\lambda^{\nub}_w \right)^T\phib\right)_{L^2(\R^d;\C)}=\left( u,\cF^{-1}\left( \left( \bm\lambda^{\nub}_w \right)^T\phib \right)\right)_{L^2(\R^d;\C)}\\
        &=\lim_{n\to\infty
        }\left( u_n,\cF^{-1}\left( \left( \bm\lambda^{\nub}_w \right)^T\phib \right)\right)_{L^2(\R^d;\C)}=\lim_{n\to\infty}\langle \bm\lambda^{\nub}_w\hat{u}_n,\phib\rangle.
    \end{split}\] 
    
    Note that the above equalities hold true since $\left( \bm\lambda^{\nub}_w \right)^T\phib \in L^2(\R^d;\C)$ as a result of \cref{Fsym0thorderest} and $\phib\in \sS(\R^d;\C^d)$.
     Thus,  $\bm\lambda^{\nub}_w\hat{u}_n\to\bm\lambda^{\nub}_w\hat{u}$ in $\sS'(\R^d;\C^d)$. Uniqueness of limits in the sense of tempered distributions, we have $\bm\lambda^{\nub}_w\hat{u}=\vb\in L^2(\R^d)$ That is, $u\in \cH^{\nub}_w(\R^d)$  with $u_n\to u$ in $\cH^{\nub}_w(\R^d)$. Thus $\cH^{\nub}_w(\R^d)$ is complete.
\end{proof}

In \cite{han2023nonlocal}, the nonlocal function space  $\cS^{\nub}_w(\R^d)$ is defined as follows: $$\cS^{\nub}_w(\R^d):=\{u\in L^2(\R^d):\mathfrak{G}^{\nub}_w u\in L^2(\R^d;\R^d)\}$$ with norm \[\|u\|_{\cS^{\nub}_w(\R^d)}:=\left( \|u\|_{L^2(\R^d)}^2+|u|_{\cS^{\nub}_w(\R^d)}^2 \right)^{\frac{1}{2}},\]
where $|u|_{\cS^{\nub}_w(\R^d)}:=\|\mathfrak{G}^{\nub}_w u\|_{L^2(\R^d;\R^d)}$.
For a given open domain $\Omega\subset \mathbb{R}^{d}$, we define the closed subspace 
\[
\mathcal{S}_{w}^{\nub}(\Omega) = \{u\in \mathcal{S}_{w}^{\nub}(\mathbb{R}^{d}): u=0\quad \text{a.e. in }\mathbb{R}^{d}\setminus\Omega\}
\]
that collects functions in $\mathcal{S}_{w}^{\nub}(\mathbb{R}^{d})$ that vanish outside of $\Omega$.  
It is shown in \cite{han2023nonlocal}  that $\mathcal{S}_{w}^{\nub}(\Omega)$ is  a Hilbert space and that if $\Omega$ has a continuous boundary  $C^\infty_c(\Omega)$ is dense in $\cS^{\nub}_w(\Omega)$. We now establish the equivalence of the space $\cS^{\nub}_w(\R^d)$ with $\cH^{\nub}_w(\R^d)$.
\begin{proposition}\label{prop:SequalH}
    $\cS^{\nub}_w(\R^d)=\cH^{\nub}_w(\R^d)$ with equal norms. More precisely, the identity map is an isometrical isomorphism between $\cS^{\nub}_w(\R^d)$ and $\cH^{\nub}_w(\R^d)$.
\end{proposition}
\begin{proof}
    First, we show $\cS^{\nub}_w(\R^d)\subseteq \cH^{\nub}_w(\R^d)$. It suffices to show that for $u\in \cS^{\nub}_w(\R^d)$, $\|\bm\lambda^{\nub}_w\hat{u}\|_{L^2(\R^d;\C^d)}=\|\mathfrak{G}^{\nub}_w u\|_{L^2(\R^d;\R^d)}<\infty$. Since $C^\infty_c(\R^d)$ is dense in $\cS^{\nub}_w(\R^d)$, there exists $\{u_n\}_{n=1}^\infty\subset C^\infty_c(\R^d)$ such that $u_n\to u$ in $L^2(\R^d)$ and $\mathcal{G}^{\nub}_w u_n\to \mathfrak{G}^{\nub}_w u$ in $L^2(\R^d)$. Since Fourier transform is an isomorphism on $L^2(\R^d)$, $\bm\lambda^{\nub}_w\hat{u}_n=\cF\left( \mathcal{G}^{\nub}_w u_n \right)$ converges to $\widehat{\mathfrak{G}^{\nub}_w u}$ in $L^2(\R^d;\C^d)$. Similar to the proof of \Cref{prop:hilbert}, one has $\bm\lambda^{\nub}_w\hat{u}_n\to \bm\lambda^{\nub}_w\hat{u}$ in $\sS'(\R^d;\C^d)$. Then we conclude that $\bm\lambda^{\nub}_w\hat{u}=\widehat{\mathfrak{G}^{\nub}_w u}\in L^2(\R^d;\C^d)$ and  $\|\bm\lambda^{\nub}_w\hat{u}\|_{L^2(\R^d;\C^d)}=\|\mathfrak{G}^{\nub}_w u\|_{L^2(\R^d;\R^d)}<\infty$.

    Second, we show that $\cH^{\nub}_w(\R^d)\subseteq \cS^{\nub}_w(\R^d)$. It suffices to show that for any $u\in\cH^{\nub}_w(\R^d)$, $\mathfrak{G}^{\nub}_w u\in L^2(\R^d;\R^d)$. By definition, $\mathfrak{G}^{\nub}_w u\in (C^\infty_c(\R^d;\R^d))'$ and for any $\phib\in C^\infty_c(\R^d;\R^d)$, \[\langle\mathfrak{G}^{\nub}_w u,\phib\rangle=-\left( u,\mathcal{D}^{-\nub}_w\phib \right)_{L^2(\R^d)}=\left( \hat{u},\overline{\bm\lambda^{\nub}_w}^T\hat{\phib} \right)_{L^2(\R^d;\C)},\] where we used $\cF\left( \mathcal{D}^{-\nub}_w\phib \right)=-\overline{\bm\lambda^{\nub}_w}^T\hat{\phib}$ by \cref{deffouriersymboldiv,eq:fourierhbdiv}. On the other hand, since $u\in\cH^{\nub}_w(\R^d)$, $\bm\lambda^{\nub}_w\hat{u}\in L^2(\R^d;\C^d)$ and thus $\cF^{-1}\left( \bm\lambda^{\nub}_w\hat{u} \right)\in L^2(\R^d;\C^d)$. For any $\phib\in C^\infty_c(\R^d;\R^d)$, \[\left( \cF^{-1}\left( \bm\lambda^{\nub}_w\hat{u} \right),\phib \right)_{L^2(\R^d;\C^d)}=\left( \bm\lambda^{\nub}_w\hat{u},\hat{\phib} \right)_{L^2(\R^d;\C^d)}=\left( \hat{u},\overline{\bm\lambda^{\nub}_w}^T\hat{\phib} \right)_{L^2(\R^d;\C)}.\]
    Therefore, $\mathfrak{G}^{\nub}_w u=\cF^{-1}\left( \bm\lambda^{\nub}_w\hat{u} \right)$ in the sense of distribution. 
    We note that $\cF^{-1}\left( \bm\lambda^{\nub}_w\hat{u} \right)$ is real-valued. This follows from the argument that for any $\phib\in C^\infty_c(\R^d;\R^d)$, \[\left( \cF^{-1}\left( \bm\lambda^{\nub}_w\hat{u} \right),\phib \right)_{L^2(\R^d;\C^d)}=\left( \hat{u},\overline{\bm\lambda^{\nub}_w}^T\hat{\phib} \right)_{L^2(\R^d;\C)}=-\left( u,\mathcal{D}^{-\nub}_w\phib \right)_{L^2(\R^d)}\in\R,\]and so  $\cF^{-1}\left( \bm\lambda^{\nub}_w\hat{u} \right)\in L^2(\R^d;\R^d)$, thus $\mathfrak{G}^{\nub}_w u\in L^2(\R^d;\R^d)$ and the proof is complete.
\end{proof}

We remark that, on the one hand, in the event that the kernel $w$ satisfying \eqref{eq:kernelassumption} also belongs to $L^{1}(\mathbb{R}^{d})$, we have that $\|\lab^{\nub}_w\|_{L^{\infty}(\mathbb{R}^{d};\R^d)} \leq 2 \|w\|_{L^{1}(\mathbb{R}^{d})}$. In this case, the function space $\mathcal{S}^{\nub}_w(\mathbb{R}^{d})$ coincides with $L^{2}(\mathbb{R}^d)$ with norm estimate 
\[
\|u\|_{\cS^{\nub}_{w}(\R^d)}^2=\int_{\R^d}(1+|\lab^{\nub}_w(\xib)|^2)|\hat{u}(\xib)|^2d\xib \leq \left(1 + 4\|w\|^2_{L^{1}(\mathbb{R}^{d})}\right) \|u\|^2_{L^{2}(\mathbb{R}^d)},\quad\forall u\in \cS^{\nub}_{w}(\R^d).
\]
On the other hand, when the kernel $w$ satisfying \eqref{eq:kernelassumption} is not integrable, i.e. \(\int_{\R^d} w(\zb)d\zb=+\infty\), or equivalently $\int_{|\zb|<1}w(\zb)d\zb=+\infty$, the space $\mathcal{S}^{\nub}_w(\mathbb{R}^{d})$ is properly contained in $L^{2}(\mathbb{R}^d)$. In fact, as one of our main results of the next section,  we will show that $\mathcal{S}^{\nub}_w(\mathbb{R}^{d})$  is locally compactly contained in $L^{2}(\mathbb{R}^d)$. An example of such type of kernel is $w_{s}(\zb) = |\zb|^{-d-s}$ for $s\in (0, 1)$. In this case, $C_1(s) |\xib|^{s}\le |\lab^{\nub}_{w_{s}}(\xib)|\le C_2(s) |\xib|^{s}$, where $C_1(s)(1-s)\to 1$ and $C_2(s)(1-s)\to 2\pi$ as $s\nearrow 1$, leading to the conclusion that $\cS^{\nub}_{w_s}(\R^d) = \cH^{\nub}_{w_{s}}(\R^d)  = H^{s}(\R^d)$, the fractional Sobolev space. 

As we will see in the sections below, working on function spaces with compactly supported kernels leads to simplifications. To that end, given a nonintegrable kernel $w$, we will show that we can have an equivalent characterization of the function space $\mathcal{S}^{\nub}_w(\mathbb{R}^{d})$ using only the cut-off $w^c(\zb):=w(\zb)\chi_{B_1(\bm 0)}(\zb)$ of the kernel $w$.

\begin{lemma}\label{lem:equiv_norm_trunc_w}
    Let $w(\zb)$ be a kernel satisfying \eqref{eq:kernelassumption}.
    Denote $w^c(\zb):=w(\zb)\chi_{B_1(\bm 0)}(\zb)$ as the compactly supported kernel obtained cutting of the tail of $w$. Then there exist positive constants $c_1$ and $c_2$ depending on $M_w^2=\int_{|\zb|>1}w(\zb)d\zb$ such that 
    \begin{equation}
        c_1\|u\|_{\cS^{\nub}_w(\R^d)}\le \|u\|_{\cS^{\nub}_{w^c}(\R^d)}\le c_2\|u\|_{\cS^{\nub}_{w}(\R^d)},\quad\forall u\in \cS^{\nub}_{w}(\R^d).
    \end{equation}
\end{lemma}
\begin{proof}
We begin noting that $w^{c}$ satisfies \eqref{eq:kernelassumption}, and so from \Cref{prop:SequalH} we have that 
    \[
    \|u\|_{\cS^{\nub}_{w}(\R^d)}^2=\int_{\R^d}(1+|\lab^{\nub}_w(\xib)|^2)|\hat{u}(\xib)|^2d\xib,\quad\forall u\in \cS^{\nub}_{w}(\R^d).
    \]
    A similar equality also holds for functions in $\cS^{\nub}_{w^c}(\R^d)$ with $w$ replaced by $w^c$.
  Now to compare $\|u\|_{\cS^{\nub}_w(\R^d)}$ and $\|u\|_{\cS^{\nub}_{w^c}(\R^d)}$, it suffices to compare $\lab^{\nub}_w(\xib)$ and $\lab^{\nub}_{w^c}(\xib)$. By \eqref{Fsym0thorderest}, 
    \begin{equation}
        |\lab^{\nub}_w(\xib)-\lab^{\nub}_{w^c}(\xib)|=|\bm\lambda_{w-w^c}^{\nub}(\xib)|\le 2 M_{w}^2, \quad\forall \xib\in\R^d.
    \end{equation} 
    Then \[
    |\lab^{\nub}_{w^c}(\xib)|^2\le (|\lab^{\nub}_w(\xib)|+2M_{w}^2)^2\le 2(|\lab^{\nub}_w(\xib)|^2+4(M_w^2)^2)
    \]
    and \[
    \|u\|_{\cS^{\nub}_{w^c}(\R^d)}^2\le \max\{2,1+8(M_w^2)^2\}\|u\|_{\cS^{\nub}_w(\R^d)}^2,\quad\forall u\in \cS^{\nub}_{w}(\R^d).
    \]
    Therefore, $c_2=(\max\{2,1+4(M_w^2)^2\})^{\frac{1}{2}}$. Changing the role of $w$ and $w^c$, we get the other inequality with $c_1=c_2^{-1}$.
\end{proof}

\section{Compactness Results}\label{sec:cpt}
In this section, we prove two compactness results. The first one states that for a fixed nonintegrable kernel $w$ satisfying \eqref{eq:kernelassumption}, the function space $\mathcal{S}_{w}^{\nub}(\mathbb{R}^{d})$ is  compactly contained in $L^{2}(\mathbb{R}^{d})$ with respect to the $L^{2}_{\mathrm{loc}}$-topology. This, in particular, implies that for such $w$, and any bounded domain with continuous boundary $\mathcal{S}_{w}^{\nub}(\Omega)$ is compactly contained in $L^{2}(\Omega)$. Other variants of this compactness results will also be proved. The second result is related to the sequence of parameterized function spaces $\{\mathcal{S}_{w_n}^{\nub}(\mathbb{R}^{d})\}_n$ associated with a sequence of kernel $\{w_n\}_n$ where $\{\min(1, |\zb|) w_n(\zb)\}_n$ is concentrating around $0$ and behaving like a Dirac delta-sequence in the appropriate sense. We will give the precise statement later, but  the result essentially states that a sequence  $\{u_n\}$ of $L^{2}$ functions with the property that    $\sup_n\|u_n\|_{\mathcal{S}_{w_n}^{\nub}(\mathbb{R}^{d})} < \infty$, has a compact closure in $L^{2}(\Omega)$, for any bounded domain $\Omega$ of $\mathbb{R}^{d}$.  
These compactness results are proved with the aid of a variant of Riesz-Kolmogorov-Fr\'echet compactness criterion proved in \cite{bourgain2001another}, see also \cite[Lemma 5.4]{mengesha2012nonlocal}.  

To state the criterion, we observe that for $\Pb\in L^1(\R^d;\R^d)$ and $f\in L^p(\R^d)$, $p\ge 1$, we understand the convolution $\Pb*f:\R^d\to\R^d$ as $(\Pb*f)_j(\xb):=(\Pb_j*f)(\xb)$. Then $\Pb*f\in L^p(\R^d;\R^d)$.

\begin{lemma}\label{lem:RKFcpt}
    Suppose that $\Pb\in L^1(\R^d;\R^d)$ and $\pb:=\int_{\R^d}\Pb(\xb)d\xb\in\R^d\backslash\{\bm 0\}$. Define $\Pb_\tau(\xb):=\tau^{-d}\Pb\left(\xb/\tau\right)$ for $\tau>0$. For any $1\leq p< \infty$, if $\{f_n\}_{n=1}^\infty$ is a bounded sequence of functions in $L^p(\R^d)$ and \begin{equation}
        \lim_{\tau\to 0^+}\limsup_{n\to\infty}\|\Pb_\tau*f_n-\pb f_n\|_{L^p(\R^d;\R^d)}=0,
    \end{equation}
    then $\{f_n\}$ has a compact closure in $L^p(\Om)$ for any bounded domain $\Om\subset \R^d$.
\end{lemma}

\subsection{Compact embedding of function spaces associated with nonintegrable kernel}\label{sec:cpt_non_int}

In this subsection we show a local compact embedding result of $\cS^{\nub}_w(\R^d)$ into $L^{2}(\R^d)$ when $w$ is a nonintegrable kernel. To be precise, we state the assumption on $w$ as follows which will remain in force through out this subsection. 
\begin{assumption}\label{assu:nonint_kernel}
    Assume that \(w\) satisfies \eqref{eq:kernelassumption} and the following conditions:
    \begin{enumerate}
        \item \(\int_{\R^d} w(\zb)d\zb=+\infty\), or equivalently $\int_{|\zb|<1}w(\zb)d\zb=+\infty$.
        \item If \(d=1\), then we assume in addition that \(w\) is nonincreasing.
    \end{enumerate}
\end{assumption}
We use \(\overline{w}\) as the radial representation of \(w\), i.e, \(\overline{w}:[0,\infty)\to[0,\infty)\) such that \(\overline{w}(|\xb|)=w(\xb)\) for all \(\xb\in\R^d\).
Note that the above assumptions hold for \(w(\zb)=|\zb|^{-d-s}\chi_{B_1(\bm 0)}(\zb)\) with \(s\in (0,1)\), as well as for $w(\zb) = -\ln(|\zb|)|\zb|^{-d}\chi_{B_1(\bm 0)}(\zb)$. 
The main result we prove in this subsection is the following. 
\begin{theorem}\label{thm:loc_cpt_embedding}
    Let $w$ satisfy \Cref{assu:nonint_kernel}. Let \(\{u_n\}_{n=1}^\infty\subset \cS^{\nub}_w(\R^d)\) be a bounded sequence in $L^{2}(\mathbb{R}^{d})$ with bounded nonlocal gradient seminorm 
    \begin{equation}\label{bdd-seminorm}
    \sup_n \|\mathfrak{G}^{\nub}_w u_n\|_{L^2(\R^d;\R^d)}=B<\infty,\end{equation}
    then for any bounded domain \(\Om\subset\R^d\), \(\{u_n|_\Om\}_{n=1}^\infty\) is precompact in \(L^2(\Om)\). In other words, \(\cS^{\nub}_w(\R^d)\) is locally compactly embedded in \(L^2(\R^d)\). 
\end{theorem}
As we mentioned earlier, we prove the theorem using the criterion given in  \Cref{lem:RKFcpt}. To that end, we set
  $\Pb(\xb):=\chi_{\nub}(\xb)\chi_{B_1(\bm 0)}(\xb)\xb/|\xb|$. 
By taking the Fourier transform, one can verify that for $u\in L^2(\R^d)$ and $\xib\in\R^d$, \[\cF\left(\Pb_\tau*u-\pb u\right)(\xib)=\bm\eta_\tau(\xib)\hat{u}(\xib),\]where 
\begin{equation}\label{defeta}
    \bm\eta_\tau(\xib):=\int_{\R^d}\frac{1}{\tau^d}\chi_{\nub}(\zb)\chi_{B_\tau(\bm 0)}(\zb)\frac{\zb}{|\zb|}(e^{-2\pi i\zb\cdot\xib}-1)d\zb,\quad\xib\in\R^d.
\end{equation}
By a change of variable one implies that 
\[\bm\eta_\tau(\xib)=\int_{\R^d}\chi_{\nub}(\zb)\chi_{B_1(\bm 0)}(\zb)\frac{\zb}{|\zb|}(e^{-2\pi i\tau\zb\cdot\xib}-1)d\zb,\quad\xib\in\R^d.\]
It then follows from a calculation similar to \cite[Theorem 2.4]{lee2020nonlocal} that 
\begin{equation}\label{etaest1}
    |\bm\eta_\tau(\xib)|\le V_d\min\left\{\sqrt{2}\pi \tau|\xib|, 1\right\},
\end{equation}
where $V_d$ is the volume of the \(d\)-dimensional unit ball \(B^d_1(\bm 0)\).

After noticing that the boundedness condition \eqref{bdd-seminorm} of Theorem \ref{thm:loc_cpt_embedding} can be expressed using the Fourier symbols as $\lab^{\nub}_{w^c} (\xib)$, our next goal is to estimate $|\etab_\tau(\xib)|$ in terms of $|\lab^{\nub}_{w^c} (\xib)|$. To that end, we have the following crucial estimating device.  
\begin{remark}\label{cut-off-lowerbound}
We reiterate that for any $w$ satisfying \eqref{eq:kernelassumption}, its cut-off  $w^c=w\chi_{B_1(\bm 0)}$ also satisfies \eqref{eq:kernelassumption} and as a consequence of \Cref{lem:labbdd_original}(2), there exist constants $N_1=N_1(w,d)>0$ and $C_1=C_1(w,d)>0$ such that \[|\bm\lambda^{\nub}_{w^c}(\xib)|\ge C_1|\xib|,\quad\forall |\xib|\le N_1.\]
We recall from \eqref{eq:epsilon_0} that there exists $\ep_0>0$ such that 
\[
0<\int_{\ep_0<|\zb|<1}w(\zb)d\zb<\infty,
\]
and thus, taking \(N=N_1\) and $\ep=\ep_0$ in \Cref{lem:lablowerbdd_xilarge}, it follows that there exists a constant \(C_2=C_2(N_1\ep_0,d)>0\) such that 
\[|\bm\lambda^{\nub}_{w^c}(\xib)|\ge C_2\int_{\frac{N_1\ep_0}{|\xib|}<|\zb|<1} w(\zb)d\zb>0,\quad\forall |\xib|>N_1.\]
\end{remark}

\begin{lemma}\label{lem:compare_eta_lab_fix}
Suppose that $\epsilon_0>0$ as in \eqref{eq:epsilon_0}.    There exists \(C=C(w,d)\) such that for \(0<\tau<\ep_0\), \[|\etab_\tau(\xib)|\le Cg(\tau) |\lab^{\nub}_{w^c} (\xib)|,\quad\forall\xib\in\R^d,\]
    where \(g:(0,\ep_0)\to\R_+\) satisfies
    \[\lim_{\tau\to 0^+ }g(\tau)=0.\]
\end{lemma}
\begin{proof}
    Fix \(0<\tau<\ep_0\), we prove the inequality by discussing \(0\le |\xib|\le N_1\), \(N_1\le |\xib|\le N_1\ep_0/\tau\) and \(|\xib|>N_1\ep_0/\tau\), where $\epsilon_0$ and $N_1$ are as in \Cref{cut-off-lowerbound}. 
        
    We begin from the estimate given in \eqref{etaest1} where for some \(c_\eta=c_\eta(d)>0\) we have  \[|\etab_\tau(\xib)|\le c_\eta(d) \min\{\tau |\xib|, 1\},\quad \text{for all}\ \xib\in\R^d.
    \]
    Then on the one hand for \(|\xib|\le  N_1\), one has \[|\etab_\tau(\xib)|\le c_\eta \tau |\xib|\le \frac{c_\eta}{C_1}\tau |\lab^{\nub}_{w^c} (\xib)|.\]
    On the other hand, for \(|\xib|>N_1\ep_0/\tau>N_1\) by \Cref{cut-off-lowerbound} one has 
    \[\begin{split}
        |\etab_\tau(\xib)|&\le c_\eta\le \frac{c_\eta}{C_2}\left(\int_{\frac{N_1\ep_0}{|\xib|}<|\zb|<1}w(\zb)d\zb\right)^{-1}|\lab^{\nub}_{w^c} (\xib)|\\
        &\le \frac{c_\eta}{C_2}\left(\int_{\tau<|\zb|<1}w(\zb)d\zb\right)^{-1}|\lab^{\nub}_{w^c} (\xib)|,
    \end{split}\]
    where we used nonnegativity of \(w\) in the last step. 
    Finally, for \(N_1\le |\xib|\le N_1\ep_0/\tau\), define 
    \[f(\tau):=\sup_{|\xib|\in\left[N_1,\frac{N_1\ep_0}{\tau}\right]} \frac{\tau|\xib|}{\int_{\frac{N_1\ep_0}{|\xib|}<|\zb|<1} w(\zb)d\zb}=\sup_{s\in\left[N_1\tau,N_1\ep_0\right]}\frac{s}{\int_{\frac{N_1\ep_0 \tau}{s}<|\zb|<1} w(\zb)d\zb},\]
    then again by \Cref{cut-off-lowerbound} we know for \(N_1\le |\xib|\le N_1\ep_0/\tau\),
    \[|\etab_\tau(\xib)|\le c_\eta \tau |\xib|\le \frac{c_\eta}{C_2}f(\tau)|\lab^{\nub}_{w^c}(\xib)|.\]
    It remains to show \(f(\tau)\to 0\) as \(\tau\to 0^+\). Suppose that there exist a constant \(\alpha>0\) and a sequence \(\{\tau_n\}_{n=1}^\infty\) such that \(\tau_n\to 0\) as \(n\to \infty\) and \(f(\tau_n)>\alpha\) for any \(n\ge 1\). Then for each \(n\ge 1\) there exists \(s_n\in [N_1\tau_n,N_1\ep_0]\) such that 
    \[\frac{s_n}{\int_{\frac{N_1\ep_0\tau_n}{s_n}<|\zb|<1} w(\zb)d\zb}>\alpha.\]
    We claim that \(\tau_n/s_n\to 0\) as \(n\to\infty\). Suppose this does not hold. Then there exist \(0< B\le 1/N_1\) and subsequences \(\{\tau_{n_k}\}_{k=1}^\infty\) and \(\{s_{n_k}\}_{k=1}^{\infty}\) such that \(B\le \tau_{n_k}/s_{n_k}\le 1/N_1\). Then since $N_1 \tau_{n_k} / s_{n_k} \leq 1$, we have 
     \[\alpha< \frac{s_{n_k}}{\int_{\frac{N_1\ep_0\tau_{n_k}}{s_{n_k}}<|\zb|<1} w(\zb)d\zb}\le \frac{\tau_{n_k}}{B\int_{\ep_0<|\zb|<1} w(\zb)d\zb}\to 0,\quad k\to +\infty.\]
    This is a contradiction. Therefore, the claim holds. Pick a decreasing subsequence of \(\{\tau_n/s_n\}_{n=1}^\infty\) and still denote it by \(\{\tau_n/s_n\}_{n=1}^\infty\). Then since $\int_{|\zb|<1}w(\zb) d\zb =\infty$, we have that as $n\to \infty,$
    \[\alpha< \frac{s_n}{\int_{\frac{N_1\ep_0\tau_n}{s_n}<|\zb|<1} w(\zb)d\zb}\le \frac{N_1\ep_0}{\int_{\frac{N_1\ep_0\tau_n}{s_n}<|\zb|<1} w(\zb)d\zb}\to\frac{N_1\ep_0}{\int_{|\zb|<1} w(\zb)d\zb}=0,\]
    where we used monotone convergence theorem.  This is a contradiction. Therefore, \(f(\tau)\to 0\) as \(\tau\to 0^+\) and 
    \[g(\tau):=\max\left\{\tau,f(\tau),\left(\int_{\tau<|\zb|<1}w(\zb)d\zb\right)^{-1}\right\}\to 0,\quad \text{as } \tau\to 0^+.\]
    Hence the proof is complete with \(C=C(w,d)=c_\eta/\min\{C_1, C_2\}\).
\end{proof}
We are now ready to give the proof of the main result. 

\begin{proof}[The proof of  \Cref{thm:loc_cpt_embedding}]
    Note that by \Cref{prop:SequalH}, \[\sup_n \|\bm\lambda_{w}^{\nub}\hat{u}_n\|_{L^2(\R^d;\C^d)}=\sup_n \|\mathfrak{G}^{\nub}_w u_n\|_{L^2(\R^d;\R^d)}=B<\infty.\]
    By \Cref{lem:equiv_norm_trunc_w}, we may use the cut-off kernel $w^{c}$ and there exists a constant $B'$ depending on $B$, $M_w^2$ and $\sup_n\|u_n\|_{L^2(\R^d)}$ such that
    \[
    \sup_n \|\bm\lambda_{w^c}^{\nub}\hat{u}_n\|_{L^2(\R^d;\C^d)}\le B'<\infty.
    \]
    We use the compactness criterion \Cref{lem:RKFcpt} to show that $\{u_n\}_{n=1}^\infty$ has a compact closure in $L^2(\Om)$ for bounded domain $\Om\subset \R^d$.
    Choose $P(\xb):=\chi_{\nub}(\xb)\chi_{B_1(\bm 0)}(\xb)\xb/|\xb|$. To that end, it suffices to show that \begin{equation}\label{limlimsup_fix}
        \lim_{\tau\to 0^+}\limsup_{n\to\infty}\|\Pb_\tau*u_n-\pb u_n\|_{L^2(\R^d;\R^d)}= 0
    \end{equation} 
    But recall that 
    \[
    \|\Pb_\tau*u_n-\pb u_n\|_{L^2(\R^d;\R^d)} = \|\mathcal{F}(\Pb_\tau*u_n-\pb u_n)\|_{L^2(\R^d;\R^d)} = \|\etab_\tau\hat{u}_n\|_{L^2(\R^d;\R^d)}
    \]
    where \(\etab_\tau\) is given by \eqref{defeta}.
    By \Cref{lem:compare_eta_lab_fix}, there exists $C=C(w,d)$ such that for \(0<\tau<\ep_0\), \[|\etab_\tau(\xib)|\le Cg(\tau) |\lab^{\nub}_{w^c} (\xib)|,\quad\forall\xib\in\R^d,\]
    where \(g:(0,\ep_0)\to\R_+\) satisfies
    $\lim_{\tau\to 0^+ }g(\tau)=0$. 
    Then for any \(n\ge 1\),
    \[\|\etab_\tau\hat{u}_n\|_{L^2(\R^d;\C^d)}\le  Cg(\tau)\|\bm\lambda_{w^c}^{\nub}\hat{u}_n\|_{L^2(\R^d;\C^d)}\le B'Cg(\tau)\to 0,\quad\tau\to 0^+.\]
    Hence, \eqref{limlimsup_fix} holds and that conpletes the  proof of the theorem.
\end{proof}

An important variant of the above compactness result holds true when applied to a sequence of function spaces $\cS^{\nub}_{w_n}(\R^d)$ corresponding to the radial kernels $w_n(\zb)$ satisfying $0\leq w_{n}(\zb) \nearrow w(\zb)$ a.e. $\zb$. 
We will present two classes of examples that satisfy this convergence property. First, take $w_n(\zb):=\min\{n, w(\zb)\}$ for $n\in\N_+$. For each $n$, $w_n \in L^{1}(\mathbb{R}^d)\cap L^{\infty}(\mathbb{R}^d)$  and  $w_{n}(\zb) \nearrow w(\zb)$ for all $\zb\in \mathbb{R}^{d}$. Note that $\cS^{\nub}_{w_n}(\R^d)$ coincides with $L^{2}(\R^d)$ for each $n$, while $\cS^{\nub}_{w}(\R^d)$ is locally compactly contained in $L^{2}(\R^d)$. For another example, we take $s\in (0, 1)$ and consider $w_s(\zb) = |\zb|^{-d-s}\chi_{B_{1}(\bm 0)}(\zb)$. For a sequence of numbers $0\leq s_n\nearrow s$ we know $w_{s_n}(\zb) := |\zb|^{-d-s_n}\chi_{B_{1}(\bm 0)}(\zb) \nearrow w_{s}(\zb)$ for all $\zb\in \mathbb{R}^{d}$, and that $\cS^{\nub}_{w_{s_n}}(\R^{d}) = H^{s_{n}}(\R^{d})$, and $\cS^{\nub}_{w_{s}}(\R^{d}) = H^{s}(\R^{d}),$ which is compactly contained in $L^{2}(\R^d)$. 

\begin{theorem}\label{thm:cpt_seq_nonint_kernel}
   Let $w$ satisfy \Cref{assu:nonint_kernel}. Suppose that $\{w_n\}_{n=1}^\infty$ is a sequence of radial kernels satisfying 
   $
   0\leq w_{n}(\zb) \nearrow w(\zb)$  for almost every $\zb\in \R^d$ and \(\{u_n\}_{n=1}^\infty\subset L^2(\R^d)\) is  a bounded sequence with 
   \[\sup_n \|\mathfrak{G}^{\nub}_{w_n} u_n\|_{L^2(\R^d;\R^d)}=B<\infty,\]
    then for any bounded domain \(\Om\subset\R^d\), \(\{u_n|_\Om\}_{n=1}^\infty\) is precompact in \(L^2(\Om)\). Moreover, if in addition $u_n\in\cS^{\nub}_{w_n}(\Om)$ for every $n$, that is, $u_n=0$ in $\Om^c$, then for any limit point \(u\), the zero-extension \(\tilde{u}\) of \(u\) outside \(\Om\) satisfies $\tilde{u}\in \cS^{\nub}_w(\Om)$ with \[\|\mathfrak{G}^{\nub}_{w} \tilde{u}\|_{L^2(\R^d;\R^d)}\le B.\]
\end{theorem}
\begin{proof}
    We begin by noting that for sufficiently large $n$, $w_n$ satisfies \eqref{eq:kernelassumption}. Indeed, by monotone convergence theorem, $M^1_{w_n}\to M^1_w$ and $M^2_{w_n}\to M^2_w$ as $n\to\infty$ which yields $M^1_{w_n}\in (0,\infty)$ and $M^2_{w_n}<\infty$ for sufficiently large $n$. Without loss of generality, we may assume $w_n$ satisfies \eqref{eq:kernelassumption} for all $n$. As before, we denote the cut-off by \(w_n^c(\zb):=w_n(\zb)\chi_{B_1(\bm 0)}(\zb)\) for \(n\in\N_+\), and again satisfies \eqref{eq:kernelassumption}. 
    Similar to the proof of \Cref{lem:equiv_norm_trunc_w}, applying \eqref{Fsym0thorderest} it follows that \begin{equation}
        |\lab^{\nub}_{w_n^c}(\xib)-\lab^{\nub}_{w_n}(\xib)|=|\bm\lambda_{w_n-w_n^c}^{\nub}(\xib)|\le 2\int_{|\zb|>1}w_n(\zb)d\zb\le 2M_w^2+1,
    \end{equation}  
    for $n$ large enough and thus there exists a constant $\tilde{B}$ depending only on $B$, $M_w^2$ and $\sup_n\|u_n\|_{L^2(\R^2)}$ such that
    \[
    \sup_n \int_{\R^d}|\lab^{\nub}_{w_n^c}(\xib)\hat{u}_n(\xib)|^2d\xib\le \tilde{B}^2.
    \]
    For the first part of the proof, similar to the proof of \Cref{thm:loc_cpt_embedding}, it suffices to show that 
    \begin{equation}
        \lim_{\tau\to 0^+}\limsup_{n\to\infty}\|\etab_\tau\hat{u}_n\|_{L^2(\R^d;\C^d)}=0,
    \end{equation} 
    where \(\etab_\tau\) is given by \eqref{defeta}. To that end, after estimating as 
    \[\begin{split}    \|\etab_\tau\hat{u}_n\|_{L^2(\R^d;\C^d)}^2&=\int_{\R^d}|\etab_\tau(\xib)\hat{u}_n(\xib)|^2d\xib\\
        &\le \left(\sup_{\xib\neq \bm 0}\frac{|\etab_\tau(\xib)|}{|\lab^{\nub}_{w_n^c}(\xib)|}\right)^2\int_{\R^d}|\lab^{\nub}_{w_n^c}(\xib)\hat{u}_n(\xib)|^2d\xib\\
        &\le \Tilde{B}^2\left(\sup_{\xib\neq \bm 0}\frac{|\etab_\tau(\xib)|}{|\lab^{\nub}_{w_n^c}(\xib)|}\right)^2,
    \end{split}\]
    it suffices to show that 
    \begin{equation}\label{eq:lim_ratio_eta_lamb}
        \lim_{\tau\to 0^{+}}\limsup_{n\to\infty} \sup_{\xib\neq \bm 0}\frac{|\etab_\tau(\xib)|}{|\lab^{\nub}_{w_n^c}(\xib)|}=0.
    \end{equation}
    Let $\epsilon_0$ and $N_1=N_1(w,d)>0$ be as in \Cref{cut-off-lowerbound} such that for some $C_1=C_1(w,d)>0$ 
    \[|\bm\lambda^{\nub}_{w^c}(\xib)|\ge C_1|\xib|,\quad\forall |\xib|\le N_1.\]
    Applying \eqref{Fsym0thorderest} to $w^c-w_n^c$, we know that for \(n\) sufficiently large,
    \begin{equation*}
        |\lab^{\nub}_{w^c}(\xib)-\lab^{\nub}_{w_n^c}(\xib)|=|\bm\lambda_{w^c-w_n^c}^{\nub}(\xib)|\le 2\sqrt{2}\pi |\xib| M_{w^c-w_n^c}^1\le \frac{1}{2}C_1|\xib|,\quad\forall\xib\in\R^d,
    \end{equation*} 
    since $M_{w^c-w_n^c}^1=M^1_w-M^1_{w_n}\to 0$ as $n\to\infty$.
    Therefore, for \(n\) sufficiently large, one has 
    \[|\bm\lambda^{\nub}_{w_n^c}(\xib)|\ge \frac{1}{2}C_1|\xib|,\quad\forall |\xib|\le N_1.\]
    On the other hand, since \(w_n^c\) satisfies \eqref{eq:kernelassumption}, applying \Cref{lem:lablowerbdd_xilarge} for \(w_n^c\) and \(N=N_1\) we know that there exists a constant \(C_2=C_2(N_1\ep_0,d)>0\) independent of \(n\) such that 
    \[|\bm\lambda^{\nub}_{w_n^c}(\xib)|\ge C_2\int_{\frac{N_1\ep_0}{|\xib|}<|\zb|<1} w_n(\zb)d\zb>0,\quad\forall |\xib|>N_1.\]
    The rest of the proof builds upon the proof of \Cref{lem:compare_eta_lab_fix}. We sketch the proof using the same notations from the proof of \Cref{lem:compare_eta_lab_fix}. Fix \(0<\tau<\ep_0\). To show \eqref{eq:lim_ratio_eta_lamb}, we discuss three cases: \(0< |\xib|\le N_1\), \(N_1\le |\xib|\le N_1\ep_0/\tau\) and \(|\xib|>N_1\ep_0/\tau\). By checking the proof of \Cref{lem:compare_eta_lab_fix} and replacing \(w^c\) by \(w_n^c\), one may show that for \(n\) sufficiently large,
    \[\sup_{0<|\xib|\le N_1} \frac{|\etab_\tau(\xib)|}{|\lab^{\nub}_{w_n^c} (\xib)|}\le 2\frac{c_\eta}{C_1}\tau \]
    and 
    \[\sup_{|\xib|>\frac{N_1\ep_0}{\tau}} \frac{|\etab_\tau(\xib)|}{|\lab^{\nub}_{w_n^c} (\xib)|}\le \frac{c_\eta}{C_2}\left(\int_{\tau<|\zb|<1}w_n(\zb)d\zb\right)^{-1}.\]
    Finally, notice that 
    \[\begin{split}
        \sup_{|\xib|\in \left[N_1,\frac{N_1\ep_0}{\tau}\right]} \frac{|\etab_\tau(\xib)|}{|\lab^{\nub}_{w_n^c} (\xib)|}&\le \frac{c_\eta}{C_2}\sup_{|\xib|\in \left[N_1,\frac{N_1\ep_0}{\tau}\right]} \frac{\tau|\xib|}{\int_{\frac{N_1\ep_0}{|\xib|}<|\zb|<1} w_n(\zb)d\zb}\\
        &=\frac{c_\eta}{C_2} 
        \sup_{s\in \left[N_1\tau,N_1\ep_0\right]}\frac{s}{\int_{\frac{N_1\ep_0 \tau}{s}<|\zb|<1} w_n(\zb)d\zb},
    \end{split}\]
    where we did a change of variables in the last equality. Introduce the notation
    \[f(n,\tau):=\sup_{s\in \left[N_1\tau,N_1\ep_0\right]}\frac{s}{\int_{\frac{N_1\ep_0 \tau}{s}<|\zb|<1} w_n(\zb)d\zb},\]
    we show next that 
    \begin{equation}\label{eq:lim_fntau}
        \lim_{\tau\to 0^+}\limsup_{n\to \infty} f(n,\tau)=0,
    \end{equation}
    
    Suppose that \eqref{eq:lim_fntau} is false. Then there exists \(\al_0>0\) and sequences \(\{\tau_j\}_{j=1}^\infty\), \(\{n_j\}_{j=1}^\infty\) and \(\{s_j\}_{j=1}^\infty\) with \(\tau_j\to 0\), \(n_j\to +\infty\) and \(s_j\in [N_1\tau_j,N_1\ep_0]\) such that
    \[\frac{s_j}{\int_{\frac{N_1\ep_0 \tau_j}{s_j}<|\zb|<1} w_{n_j}(\zb)d\zb}>\al_0.\]
    Arguing as in the proof of \Cref{lem:compare_eta_lab_fix}, one may show that \(\tau_j/s_j\to 0\) as \(j\to \infty\), and without loss of generality, one may assume that \(\{\tau_j/s_j\}_{j=1}^\infty\) is decreasing and \(\{n_j\}_{j=1}^\infty\) is increasing. Since \(\{\chi_{B^c_{N_1\ep_0 \tau_j/s_j}(\bm 0)}(\zb)w_{n_j}(\zb)\}_{j=1}^\infty\) is an increasing sequence of nonnegative functions, by monotone convergence theorem,
    \[\al_0< \frac{s_j}{\int_{\frac{N_1\ep_0\tau_j}{s_j}<|\zb|<1} w_{n_j}(\zb)d\zb}\le \frac{N_1\ep_0}{\int_{\frac{N_1\ep_0\tau_j}{s_j}<|\zb|<1} w_{n_j}(\zb)d\zb}\to\frac{N_1\ep_0}{\int_{|\zb|<1} w(\zb)d\zb}=0\]
    as \(n\to\infty\), a contradiction. Therefore \eqref{eq:lim_fntau} holds, and consequently \eqref{eq:lim_ratio_eta_lamb} holds by combing the three cases. The precompactness follows from \Cref{lem:RKFcpt}. Finally, assuming that $u_n\in\cS^{\nub}_{w_n}(\Om)$ for every $n$, we show that for any limit point \(u\), the zero-extension \(\tilde{u}\) of \(u\) outside \(\Om\) satisfies $\tilde{u}\in \cS^{\nub}_w(\Om)$ with \[\|\mathfrak{G}^{\nub}_{w} \tilde{u}\|_{L^2(\R^d;\R^d)}\le B.\]
    Without loss of generality, assume that \(u_n\to u\) in \(L^2(\Om)\). Then \(u_n\to \tilde{u}\) in \(L^2(\R^d)\). Thus \(\hat{u}_n\to \hat{\tilde{u}}\) in \(L^2(\R^d)\) and up to a subsequence we know \(\hat{u}_n\to \hat{\tilde{u}}\) a.e. in \(\R^d\). On the other hand, by \eqref{Fsym0thorderest} one may show that \(\lab^{\nub}_{w_n}(\xib)\to \lab^{\nub}_w(\xib)\) as \(n\to\infty\) for any \(\xib\in\R^d\) as $M^1_{w_n}\to M^1_w$ and $M^2_{w_n}\to M^2_w$. Then by Fatou's lemma, one obtains that
    \[\int_{\R^d}|\lab^{\nub}_w(\xib) \hat{\tilde{u}}(\xib)|^2 d\xib\le \liminf_{n\to\infty} \int_{\R^d}|\lab^{\nub}_{w_n}(\xib) \hat{u}_n(\xib)|^2 d\xib\le B^2.\]
    This completes the proof.
\end{proof}

As a consequence of the above compactness result \Cref{thm:cpt_seq_nonint_kernel}, we  establish the following  uniform Poincar\'e inequality .
\begin{theorem}\label{thm:unifPoincare_nonint_kernel}
    Let \(\{w_n:n\in\N_+\}\) be a family of kernels as in \Cref{thm:cpt_seq_nonint_kernel}. Then there exist $N_0\in\N_+$ and  $C(N_0)>0$ such that for any $n\in \N_+$ with $n>N_0$, \begin{equation}
        \|u\|_{L^2(\Om)}\le C(N_0)\|\mathfrak{G}^{\nub}_{w_n} u\|_{L^2(\R^d;\R^d)},\quad\forall u\in \cS^{\nub}_{w_n}(\Om).
    \end{equation}
\end{theorem}
\begin{proof}
    Let \[\frac{1}{A}:=\inf \left\{ \|\mathfrak{G}^{\nub}_{w} u\|_{L^2(\R^d;\R^d)}:\, u\in \cS^{\nub}_w(\Om),\|u\|_{L^2(\Om)}=1 \right\}.\]By the nonlocal Poincar\'e inequality for the kernel $w$, i.e., \cite[Theorem~5.1]{han2023nonlocal}, one has $0<A<\infty$. To prove the theorem, we will prove that for any $\ep>0$, there exist $N_0(\ep)>0$ and $A+\ep>0$ such that for any $n>N_0$, \[\|u\|_{L^2(\Om)}\le (A+\epsilon)\|\mathfrak{G}^{\nub}_{w_n}u\|_{L^2(\R^d;\R^d)},\quad\forall u\in \cS^{\nub}_{w_n}(\Om).\]
   We prove this statement by contradiction. Suppose that there exist $C>A$,  sequences $\{n_k\}_{k=1}^{\infty}$ and $\{u_{n_k}\}_{k=1}^\infty$ where $n_k>k$ for all $k\in\N$ such that $u_{n_k}\in \cS^{\nub}_{w_{n_k}}(\Om)$, $\|u_{n_k}\|_{L^2(\Om)}=1$ and \[\|\mathfrak{G}^{\nub}_{w_{n_k}} u_{n_k}\|_{L^2(\R^d;\R^d)}<\frac{1}{C}.\]
    Then by \Cref{thm:cpt_seq_nonint_kernel}, there exists \(u\in \cS^{\nub}_w(\Om)\) such that \(u_{n_k}\to u\) in \(L^2(\Om)\) up to a subsequence and  \[\|\mathfrak{G}^{\nub}_{w} u\|_{L^2(\R^d;\R^d)}\le \frac{1}{C}<\frac{1}{A}.\] But since as a strong limit \(\|u\|_{L^2(\Om)}=1\), this gives a contradiction to the fact that $A$ is chosen the best constant.
\end{proof}

\subsection{Compactness with vanishing nonlocality}

In this subsection, we prove the second compactness criterion which is based on boundedness of the sequence of nonlocal gradient energies  associated with kernels with vanishing nonlocality. To that end, in this section we will be working with the set of kernels \(\{w_\del:\del\in (0,1)\}\) satisfying \eqref{eq:kernelassumption} and the additional assumptions that for any \(R>0\),
\begin{equation}\label{eq:kernelassumption_2}
    \lim_{\del\to 0} \int_{B_R^c(\bm 0)} w_\del(\zb)d\zb = 0,
\end{equation}
\begin{equation}\label{eq:kernelassumption_3}
    \lim_{\del\to 0} \int_{B_R(\bm 0)} |\zb|w_\del(\zb)d\zb = 2d.
\end{equation}
 Conditions \eqref{eq:kernelassumption_2} and \eqref{eq:kernelassumption_3} imply that for any $R>0$
    \begin{equation}\label{eq:kernel_2ndMoment2zero} 
    \lim_{\del\to 0} \int_{B_R(\bm 0)} |\zb|^2 w_\del(\zb)d\zb = 0.
\end{equation}
Indeed, given any $\ep\in (0,R)$, 
\[\begin{split}
    \int_{B_R(\bm 0)}|\zb|^2w_\del(\zb)d\zb&=\int_{B_\ep(\bm 0)} |\zb|^2w_\del(\zb)d\zb+\int_{\ep<|\zb|<R}|\zb|^2w_\del(\zb)d\zb\\
    &\le \ep\int_{B_\ep(\bm 0)} |\zb|w_\del(\zb)d\zb+R^2\int_{B_\ep(\bm 0)^c}w_\del(\zb)d\zb.
\end{split}\]
Letting $\del\to 0$ and using \eqref{eq:kernelassumption_2} and \eqref{eq:kernelassumption_3} we obtain that
\[
\limsup_{\del\to 0} \int_{B_R(\bm 0)}|\zb|^2w_\del(\zb)d\zb\le 2d\ep.
\]
Since $\ep>0$ can be arbitrarily small, \eqref{eq:kernel_2ndMoment2zero} holds. 

A consequence of this is that these conditions on the sequence of kernels are sufficient for the convergence of corresponding nonlocal operators to their local counterparts as $\del$ tends to zero as stated below. The proof is presented in  \Cref{subsec:Appendix1}. 
\begin{proposition}\label{prop:locallim_ptws_lp}
    Let $u\in C^1_c(\R^d)$ and $\vb\in C^1_c(\R^d;\R^d)$. Let \(\{w_\del:\del\in (0,1)\}\) satisfy \cref{eq:kernelassumption,eq:kernelassumption_2,eq:kernelassumption_3}.
    Denote \(\mathcal{G}_\del^{\nub}:=\mathcal{G}_{w_\del}^{\nub}\) and \(\mathcal{D}_\del^{\nub}:=\mathcal{D}_{w_\del}^{\nub}\).
    Then for any $\xb\in\R^d$, 
    \begin{equation}
        \mathcal{G}_\del^{\nub} u(\xb)\to \nabla u(\xb),\quad \del\to 0,
    \end{equation}  
    \begin{equation}
        \mathcal{D}_\del^{\nub} \vb(\xb)\to \mathrm{div}\,\vb(\xb),\quad \del\to 0.
    \end{equation}
    Moreover, for any $p\in [1,\infty]$, the above two convergences are strong in $L^p(\R^d;\R^d)$ and $L^p(\R^d)$ respectively.
\end{proposition}
This localization result leads to the natural question about the convergence property of bounded sequence $\{u_\delta\in \mathcal{S}^{\nub}_{w_\delta}(\mathbb{R}^d)\}$; that is, if $\sup_{\delta}|u_\delta|_{\mathcal{S}^{\nub}_{w_\delta}} <\infty,$ is the sequence compact in $L^{2}_{\mathrm{loc}}$? And if yes, what can we say about limit points $u$? As the localization result Proposition \ref{prop:locallim_ptws_lp}   suggest, does $u$ belong in $H^{1}_{\mathrm{loc}}(\mathbb{R}^d)$?  We conjecture that these questions can be answered in the affirmative for the general set of kernels satisfying \eqref{eq:kernelassumption},\eqref{eq:kernelassumption_2}, and \eqref{eq:kernelassumption_3}. Although we are unable to answer these questions in full generality, we will address the question in this subsection for a selected set of sequence of kernels. The types of sequence of  kernels we will be focused on are listed in the following assumption. 

\begin{assumption}\label{ass:kernel4cptthm}
    We assume that \(\{w_\del:\del\in (0,1)\}\) is given by one of the three cases:
    \begin{enumerate}[label=(\Roman*)]
        \item \label{item:first} \(w_\del(\zb):=\del^{-d-1}w(\zb/\del)\) for \(\del\in (0,1)\) where \(w\) satisfies \eqref{eq:kernelassumption} and \[\int_{\R^d} |\zb|w(\zb)d\zb = 2d.\]
        \item \label{item:second} \(w_\del(\zb):=2d\del|\zb|^{\del-d-1}\) for \(\del\in (0,1)\). 
        \item \label{item:third} \(\{w_\del:\del\in (0,1)\}\) is a set of kernels satisfying \cref{eq:kernelassumption,eq:kernelassumption_2,eq:kernelassumption_3}, and there exist constants $R_1>0$, $R_2>0$ and $\bar\del_0\in (0,1)$ and a function $h:(0,\bar\del_0)\to \R_+$ with $\lim_{\del\to 0}h(\delta)=0$ such that 
    \begin{equation}\label{ker_cond_tail_int}
        \lim_{\del\to 0} \int_{|\zb|>R_1h(\del)} w_\del(\zb)d\zb = +\infty
    \end{equation}
    and
    \begin{equation}\label{ker_cond_2nd_mom}
        \sup_{\del\in (0,\bar\del_0)} \frac{1}{h(\del)}\int_{B_{R_2}(\bm 0)} |\zb|^2w_\del(\zb)d\zb:=M_0<+\infty.
    \end{equation}
    In addition, when $d=1$ assume $w_\del(\zb)$ is nonincreasing  for each $\del\in (0,1)$. 
    \end{enumerate}
\end{assumption}

\begin{remark}\label{rem:ThirdKernel}
    A typical example satisfying \Cref{ass:kernel4cptthm}\ref{item:third} is given by 
    \begin{equation}
        w_\del(\zb)=\frac{1}{|\log\del|}\frac{1}{|\zb|(|\zb|+\del)^d}
    \end{equation}
    with $R_1=R_2=\bar\del_0=1$ and $h(\del)=\del$. One may show that the first two types of kernels in \Cref{ass:kernel4cptthm} satisfy \cref{eq:kernelassumption,eq:kernelassumption_2,eq:kernelassumption_3} but not \Cref{ass:kernel4cptthm}\ref{item:third}. Specifically, \ref{item:first} does not satisfy \eqref{ker_cond_2nd_mom} and \ref{item:second} does not satisfy \eqref{ker_cond_tail_int}.
\end{remark}

The main theorem of this subsection is the following.

\begin{theorem}\label{thm:cpt_full}
    Let \(\{w_{\del_n}:\del_n\in(0,1)\}_{n=1}^\infty\) be a sequence of kernels that belongs to either \ref{item:first}, \ref{item:second} or \ref{item:third} of  \Cref{ass:kernel4cptthm}. Suppose $\{u_n\}_{n=1}^\infty\subset L^2(\R^d)$ is a bounded sequence and \[\sup_n |u_n|_{\cS^{\nub}_{w_{\del_n}}(\R^d)}=\sup_n \left\|\mathfrak{G}^{\nub}_{w_{\del_n}} u_n\right\|_{L^2(\R^d;\R^d)}=B<\infty,\]then $\{u_n\}_{n=1}^\infty$ has a compact closure in $L^2(\Om)$ for bounded domain $\Om\subset \R^d$. Moreover, if in addition $\Om$ has a Lipschitz boundary and $u_n\in\cS^{\nub}_{w_{\del_n}}(\Om)$ for every $n$, then any limit point $u\in H_0^1(\Om)$ with \[\|\nabla u\|_{L^2(\Om;\R^d)}\le B.\]
\end{theorem}

The theorem will be shown for the three classes  of kernels in \cref{ass:kernel4cptthm} separately, each requiring subtly different techniques. From now on, we let $\{\del_n\}_{n=1}^\infty$ be a sequence of positive numbers such that $\del_n\in(0,1)$ and $\del_n\to 0$ as $n\to\infty$.

\begin{proof}[Proof of \Cref{thm:cpt_full} for the class of kernels satisfying  \Cref{ass:kernel4cptthm}\ref{item:first} ]
    Taking the Fourier transform and by applying   \Cref{prop:SequalH} we have, \[\sup_n \left\|\bm\lambda_{{w_{\del_n}}}^{\nub}\hat{u}_n\right\|_{L^2(\R^d;\C^d)}=\sup_n \left\|\mathfrak{G}^{\nub}_{w_{\del_n}} u_n\right\|_{L^2(\R^d;\R^d)}=B<\infty.\]
   We prove the theorem in two step. We first show the sequence is compact in the $L^{2}_{\mathrm{loc}}$ toplogy, and then demonstrate that any limit point is in $H^{1}_{0}(\Omega)$. 
   
   \textbf{Step I}. To show that $\{u_n\}_{n=1}^\infty$ has a compact closure in $L^2(\Om)$ for bounded domain $\Om\subset \R^d$, we will apply the compactness criterion given in \Cref{lem:RKFcpt}. For that we 
    choose $P(\xb):=\chi_{\nub}(\xb)\chi_{B_1(\bm 0)}(\xb)\xb/|\xb|$ and demonstrate that  \begin{equation}\label{limlimsup}
        \lim_{\tau\to 0^+}\limsup_{n\to\infty}\|\Pb_\tau*u_n-\pb u_n\|_{L^2(\R^d;\R^d)}=0.
    \end{equation} 
   As before, we consider \(\etab_\tau\) given by \eqref{defeta}.
   
    \textbf{Claim}. There exists $C=C(w,d)$ such that for $0<\del\le \tau<1$, \begin{equation}\label{I-etatau}
        |\bm\eta_\tau(\xib)|\le C\tau|\bm\lambda^{\nub}_{w_\del}(\xib)|,\quad\forall\xib\in\R^d.
    \end{equation}
    Once this claim is proved, \eqref{limlimsup} holds. Indeed, since for any $\tau\in (0,1)$, there exists $n_\tau>0$ such that for any $n>n_\tau$, $\del_n<\tau$ we have $|\bm\eta_\tau(\xib)|\le C\tau|\bm\lambda^{\nub}_{n}(\xib)|$ for all $\xib\in\R^d$. Thus for $n>n_\tau$, 
    \[\begin{split}
        \|\Pb_\tau*u_n-\pb u_n\|^2_{L^2(\R^d;\R^d)}&=\int_{\R^d}|\bm\eta_\tau(\xib)\hat{u}_n(\xib)|^2d\xib\\
        &\le C\tau^2\int_{\R^d}\left|\bm\lambda_{{w_{\del_n}}}^{\nub}(\xib)\hat{u}_n(\xib)\right|^2d\xib\le CB^2\tau^2.
    \end{split}\]
  What remains is to prove the above claim.  By a change of variable, one may check that $\bm\lambda^{\nub}_{w_\del}(\xib)=\del^{-1}\bm\lambda^{\nub}_{w}(\del\xib)$ for all $\xib\in\R^d$. Therefore, for $0<\del\le \tau<1$, by \Cref{lem:labbdd_original}, there exist constants \(N_1=N_1(w,d)\) and \(C_1=C_1(w,d)\) such that \[|\bm\lambda^{\nub}_{w_\del}(\xib)|\ge C_1|\xib|,\quad\forall |\xib|<\frac{N_1}{\del} \] and \[|\bm\lambda^{\nub}_{w_\del}(\xib)|\ge \frac{C_1}{\del}\int_{\R^d}\min(1,w(\xb))d\xb\ge \frac{C_2(w,d)}{\tau},\quad\forall |\xib|\ge \frac{N_1}{\del} .\]Comparing these estimates with \eqref{etaest1} proves the \eqref{I-etatau}, completing the proof of the claim. 

    \textbf{Step II}. Suppose $u_n\in\cS^{\nub}_{w_{\del_n}}(\Om)$ for every $n$ and $u$ is a limit point of $\{u_n\}_{n=1}^\infty$ in $L^2(\Om)$, that is, $u_n\to u$ in $L^2(\Om)$ up to a subsequence. Then we show that  $u\in H_0^1(\Om)$ and $\|\nabla u\|_{L^2(\Om;\R^d)}\le B$. Let $\tilde{u}$ be the zero extension of $u$ outside $\Om$, i.e, \[\tilde{u}(\xb)=\begin{cases}
        u(\xb),& \xb\in\Om,\\
         0,& \xb\in\Om^c.
    \end{cases}\]Then $u_n\to \tilde{u}$ in $L^2(\R^d)$. By definition of distributional gradient, for any $\phib\in C^\infty_c(\R^d;\R^d)$, 
    \begin{equation*}
        \begin{split}
            \left|\int_{\R^d} u_n(\xb)\mathcal{D}^{-\nub}_{w_{\del_n}}\phib(\xb)d\xb\right|&=\left| -\int_{\R^d}\mathfrak{G}^{\nub}_{w_{\del_n}} u_n(\xb)\cdot \phib(\xb)d\xb \right|\\
            &\le \|\mathfrak{G}^{\nub}_{w_{\del_n}} u_n\|_{L^2(\R^d;\R^d)}\|\phib\|_{L^2(\R^d;\R^d)}\\
            &\le B\|\phib\|_{L^2(\R^d;\R^d)}.
        \end{split}
    \end{equation*}
    Passing to the limit as $n\to\infty$, one obtains 
    \begin{equation}\label{uinH1pre}
        \left|\int_{\R^d}\tilde{u}(\xb)\mathrm{div}\,\phib(\xb)d\xb\right|\le B\|\phib\|_{L^2(\R^d;\R^d)},\quad\forall \phib\in C^\infty_c(\R^d;\R^d),
    \end{equation}
    where we used \Cref{prop:locallim_ptws_lp} to assert $\mathcal{D}^{-\nub}_{w_{\del_n}}\phib\to\mathrm{div}\,\phib$ in $L^2(\R^d)$. Therefore, $\tilde{u}\in H^1(\R^d)$ with 
    \[\|\nabla \tilde{u}\|_{L^2(\R^d;\R^d)}\le B.\] 
    Then
    by \cite[Theorem~3.7]{wloka1987partial}, $u\in H_0^1(\Om)$ and \(\|\nabla u\|_{L^2(\Om;\R^d)}\le B\), where we used the fact $\Om$ is a bounded Lipschitz domain.
\end{proof}

\begin{proof}[Proof of \Cref{thm:cpt_full} for the class of kernels satisfying  \Cref{ass:kernel4cptthm}\ref{item:second} ]
Again we will use the Riesz-Kolmogorov-Fr\'echet criterion \Cref{lem:RKFcpt}.  The key ingredient is a good lower bound on the norm of Fourier symbol \(\lab^{\nub}_{w_\del}\).

We {\bf claim} that if \(w_\del(\zb)=2d\del|\zb|^{\del-d-1}\), then there exist constants \(C_1(d)>0\), \(C_2(d)>0\) and \(\del_0\in (0,1/2)\) such that for any \(\del\in (0,\del_0)\),
    \begin{equation}\label{frac-lower-upper-bound}
        C_1(d)|\xib|^{1-\del}\le |\lab^{\nub}_{w_\del}(\xib)| \le C_2(d) |\xib|^{1-\del},\quad\forall\xib\in\R^d.
    \end{equation}
    We will prove \eqref{frac-lower-upper-bound} shortly but a consequence of it is that there exists $C=C(d)$ such that for $\tau\in(0,1)$ and \(\del\in (0,\del_0)\), \begin{equation}\label{eta-estimate-typeII}
        |\bm\eta_\tau(\xib)|\le C\tau^{\frac{1}{2}}|\bm\lambda^{\nub}_{w_\del}(\xib)|,\quad\forall\xib\in\R^d,
    \end{equation}
    from which the condition in \Cref{lem:RKFcpt} is satisfied. 

To prove \eqref{eta-estimate-typeII}, we estimate it in two regimes. Let \(\del_0\) be the constant used in \eqref{frac-lower-upper-bound}. Then for any $\delta\in (0, \delta_0)$ and \(|\xib|\le 1/\tau\), using \eqref{etaest1} and \eqref{frac-lower-upper-bound} we know that there exists a constant \(C(d)>0\) such that
    \[\frac{|\etab_\tau(\xib)|}{|\lab^{\nub}_{w_\del}(\xib)|} \le C(d)\frac{\tau|\xib|}{|\xib|^{1-\del}}\le C(d)\tau|\xib|^\del\le C(d)\tau^{1-\del}\le C(d)\tau^{\frac{1}{2}}.\]
    Similarly, for \(|\xib|>1/\tau\), again using using \eqref{etaest1} and \eqref{frac-lower-upper-bound}, we have 
    \[\frac{|\etab_\tau(\xib)|}{|\lab^{\nub}_{w_\del}(\xib)|} \le C(d)\frac{1}{|\xib|^{1-\del}}\le C(d)\tau^{1-\del}\le C(d)\tau^{\frac{1}{2}}.\]

Let us now prove \eqref{frac-lower-upper-bound}. By \eqref{deffouriersymboldiv} and the form of \(w_\del\), for \(\del\in (0,1/2)\) and \(\xib\in\R^d\) we have
\[\begin{split}
    |\lab^{\nub}_{w_\del}(\xib)| &\le 2d\int_{\R^d} \frac{\del}{|\zb|^{d+1-\del}} \left|e^{2\pi i\zb\cdot\xib}-1\right|d\zb\\
    &= 2d|\xib|^{1-\del}\int_{\R^d} \frac{\del}{|\zb|^{d+1-\del}} \left|e^{2\pi i\zb\cdot\xib/|\xib|}-1\right|d\zb\\
    &\le 2d\del |\xib|^{1-\del} \left(\int_{B_1(\bm 0)}\frac{2\pi|\zb|}{|\zb|^{d+1-\del}}d\zb + \int_{B_1(\bm 0)^c}\frac{2}{|\zb|^{d+1-\del}} d\zb \right)\\
    &=2d\del |\xib|^{1-\del}\om_{d-1} \left(\frac{2\pi}{\del}+\frac{2}{1-\del}\right)\\
    &\le 8d\pi \om_{d-1}|\xib|^{1-\del}.
\end{split}\]    
For the lower bound, we recall from \cite[Lemma~5.1]{han2023nonlocal} that when \(d\ge 2\) the norm of imaginary part of \(\lab^{\nub}_{w_\del}(\xib)\) is given by
\[\Im (\lab^{\nub}_{w_\del})(\xib)
=2d\om_{d-2} \int_{0}^{\frac{\pi}{2}}\cos(\theta)\sin^{d-2}(\theta) \int_{0}^{\infty} \frac{\del}{r^{2-\del}}\sin(2\pi|\xib|r\cos(\theta)) drd\theta.\]
By a change of variable we have
\[|\lab^{\nub}_{w_\del}(\xib)|\ge 2d|\xib|^{1-\del} \om_{d-2} \int_{0}^{\frac{\pi}{2}}\cos^{2-\del}(\theta)\sin^{d-2}(\theta) \int_{0}^{\infty} \frac{\del}{r^{2-\del}}\sin(2\pi r) drd\theta.\]
By \Cref{prop:symblim} in \Cref{subsec:Appendix3} we know as \(\del\to 0\),
\[\int_{0}^{\frac{\pi}{2}}\cos^{2-\del}(\theta)\sin^{d-2}(\theta) \int_{0}^{\infty} \frac{\del}{r^{2-\del}}\sin(2\pi r) drd\theta\to 2\pi \int_{0}^{\frac{\pi}{2}}\cos^{2}(\theta)\sin^{d-2}(\theta)d\theta>0.\]
Then there exist constants \(C_1(d)>0\) and \(\del_0\in (0,1/2)\) such that for any \(\del\in (0,\del_0)\),
\[|\lab^{\nub}_{w_\del}(\xib)|\ge C_1(d)|\xib|^{1-\del},\quad\forall\xib\in\R^d.\]
For \(d=1\) the proof is similar as one can show \(|\lambda^{\nub}_{w_\del}(\xi)|\ge 2|\xi|^{1-\del} \int_{0}^{\infty} \frac{\del}{r^{2-\del}}\sin(2\pi r) dr\) for any \(\xi\in\R\). Thus the inequality \eqref{frac-lower-upper-bound} is established.

\end{proof}

{ To prove \Cref{thm:cpt_full} when the set of kernels satisfies \Cref{ass:kernel4cptthm}\ref{item:third}, we first establish lower bound estimates of the parameterized symbols $\lab^{\nub}_{w_\del^c} (\xib)$ that is uniform in $ \delta$. Here $w_\del^c(\zb)$ stands for the compactly supported kernel after truncation $w_\del^c(\zb):=w_\del(\zb)\chi_{B_{R_2}(\bm 0)}(\zb)$. We achieve this by separately estimating  the real and imaginary parts of  $\lab^{\nub}_{w_\del^c} (\xib)$. We recall that we can write 
\[
\lab^{\nub}_{w_\del^c} (\xib) = \Re(\lab^{\nub}_{w_\del^c})(\xib) + i\Im(\lab^{\nub}_{w_\del^c})(\xib)
\]
where 
\[
\Re(\lab^{\nub}_{w_\del^c})(\xib) = \int_{B_{R_2}(\bm 0)}\chi_{\nub}(\zb) {\zb\over|\zb|}w_{\delta}(\zb)(\cos(2\pi\xib\cdot\zb) - 1)d\zb \quad \text{and}
\]
\[
\Im(\lab^{\nub}_{w_\del^c})(\xib) = \int_{B_{R_2}(\bm 0)}\chi_{\nub}(\zb) {\zb\over|\zb|}w_{\delta}(\zb)\sin(2\pi\xib\cdot\zb)d\zb.
\]
We estimate $\left|\Im(\lab^{\nub}_{w_\del^c})(\xib)\right|$ and $\left|\Re(\lab^{\nub}_{w_\del^c})(\xib)\right|$ for small $|\xib|$ and large $|\xib|$ respectively.

\begin{lemma}\label{lem:lowerbd_im_re}
    Let \(\{w_\del:\del\in(0,1)\}\) be a family of kernels satisfying \Cref{ass:kernel4cptthm}. Then there exist constants $\bar\del_1=\bar\del_1(h,R_1,R_2)\in (0,\bar\del_0)$, $C_0=C_0(M_0,d)$, $C_1=C_1(d)$ and $C_2=C_2(R_1,M_0,d)$ such that for any $\del\in (0,\bar\del_1)$,
    \begin{equation}\label{eq:Imest}
        \left|\Im(\lab^{\nub}_{w_\del^c})(\xib)\right|\ge C_1|\xib|,\quad\forall |\xib|\le \frac{C_0}{h(\del)},
    \end{equation}
    and 
    \begin{equation}\label{eq:Reest}
        \left|\Re(\lab^{\nub}_{w_\del^c})(\xib)\right|\ge C_2\int_{|\zb|\ge R_1 h(\del)}w_\del^c(\zb)d\zb>0,\quad\forall |\xib|> \frac{C_0}{h(\del)}.
    \end{equation}
\end{lemma}
\begin{proof}
    Without loss of generality, we assume $d\ge 2$. The case $d=1$ can be proved similarly. In \cite[Lemma~5.1]{han2023nonlocal} it is shown that
    \[\left|\Im (\lab^{\nub}_{w_\del^c})(\xib)\right|=\left|\om_{d-2}\int_{0}^{\frac{\pi}{2}}\cos(\theta)\sin^{d-2}(\theta)\int_0^{R_2}r^{d-1}\overline{w_\del}(r)\sin(2\pi|\xib|r\cos(\theta))drd\theta\right|.\]
    Using the inequality $\sin x\ge x-x^2/2$ for $x\ge 0$, one obtains that 
    \[\begin{split}
       \left|\Im (\lab^{\nub}_{w_\del^c})(\xib)\right|&\ge  \om_{d-2}\int_{0}^{\frac{\pi}{2}}\cos(\theta)\sin^{d-2}(\theta)\cdot\\
       &\quad\int_0^{R_2}r^{d-1}\overline{w_\del}(r)\left(2\pi|\xib|r\cos(\theta)-2\pi^2|\xib|^2r^2\cos(\theta)^2\right)drd\theta\\
       &=|\xib|\left(c_1(d)\int_{B_{R_2}(\bm 0)}|\zb|w_\del(\zb)d\zb-c_2(d)|\xib|\int_{B_{R_2}(\bm 0)}|\zb|^2w_\del(\zb)d\zb\right)
    \end{split}\]
    where $c_1(d)$ and $c_2(d)$ are positive constants. Since 
    \[
    \lim_{\del\to 0}h(\del)=0\quad\text{and}\quad\lim_{\del\to 0}\int_{B_{R_2}(\bm 0)}|\zb|w_\del(\zb)d\zb=2d,
    \]
    there exists $\bar\del_1\in (0,\bar\del_0)$ depending on $R_1$ and $R_2$ such that \[R_1h(\del)<R_2\quad\text{and}\quad\int_{B_{R_2}(\bm 0)}|\zb|w_\del(\zb)d\zb>d\] for any $\del\in (0,\bar\del_1)$. Choose $C_0=C_0(M_0,d)$ such that $C_0M_0c_2(d)=c_1(d)d/2$, then for any $\del\in (0,\bar\del_1)$ and $|\xib|\le C_0/h(\del)$, by the calculation above and \eqref{ker_cond_2nd_mom},
    \[\Im (\lab^{\nub}_{w_\del^c})(\xib)\ge |\xib|\left(c_1(d)d-c_2(d)\frac{C_0}{h(\del)}\int_{B_{R_2}(\bm 0)}|\zb|^2w_\del(\zb)d\zb\right)\ge \frac{1}{2}c_1(d)d|\xib|.\]
    Then \eqref{eq:Imest} is shown where $C_1=c_1(d)d/2$.
    By \Cref{lem:lablowerbdd_xilarge}, taking $N=C_0/h(\del)$ and $\ep=R_1h(\del)$ there exist a constant $C_2=C_2(R_1,C_0,d)=C_2(R_1,M_0,d)>0$ such that for $|\xib|> C_0/h(\del)$,
    \[\left|\Re(\lab^{\nub}_{w_\del^c})(\xib)\right|\ge C_2\int_{|\zb|\ge \frac{C_0R_1}{|\xib|}}w_\del^c(\zb)d\zb\ge \int_{|\zb|\ge  R_1 h(\del)}w_\del^c(\zb)d\zb.\]
    Note that \eqref{ker_cond_tail_int} implies that $B_{R_1h(\del)}(\bm 0)\subset\supp\,w^{c}_\del$ for $\del$ small enough. Thus without loss of generality one may assume that $B_{R_1h(\del)}(\bm 0)\subset\supp\,w_\del$ for $\del\in (0,\bar\del_0)$. 
    The proof is then complete.
\end{proof}}

\begin{remark}\label{rem:equiv_norm_trunc_w_delta}
Following the exact proof of \Cref{lem:equiv_norm_trunc_w}, one may show that 
for \(\{w_\del:\del\in(0,1)\}\) be a family of kernels satisfying \eqref{eq:kernelassumption}, the function spaces $\cS^{\nub}_{w_\del}(\R^d)$ and $\cS^{\nub}_{w^c_\del}(\R^d)$ are equivalent with equivalent norms. 
    For a fixed constant $R>0$, if we denote  $w_\del^c(\zb):=w_\del(\zb)\chi_{B_R(\bm 0)}(\zb)$ as the compactly supported kernel obtained from a truncation of $w_\del$, then there exist constants $\bar\del_2\in (0,1)$ depending on $R$ such that for any $\del\in (0,\bar\del_2)$
    \begin{equation}
        \frac{1}{\sqrt{2}}\|u\|_{\cS^{\nub}_{w_\del}(\R^d)}\le \|u\|_{\cS^{\nub}_{w_\del^c}(\R^d)}\le \sqrt{2}\|u\|_{\cS^{\nub}_{w_\del}(\R^d)},\quad\forall u\in \cS^{\nub}_{w_\del}(\R^d).
    \end{equation}
\end{remark}
We are now ready to prove \Cref{thm:cpt_full} for the last class of kernels.

\begin{proof}[Proof of \Cref{thm:cpt_full} for the class of kernels satisfying \Cref{ass:kernel4cptthm}\ref{item:third}]
    Denote $w_{\del_n}^c(\zb):=w_{\del_n}(\zb)\chi_{B_{R_2}(\bm 0)}(\zb)$ where $R_2$ is given by \eqref{ker_cond_2nd_mom}. Similar to the proof of \Cref{thm:cpt_seq_nonint_kernel} and in view of \Cref{rem:equiv_norm_trunc_w_delta}, one can show that there exists a constant $\tilde{B}$ depending only on $B$ and $\sup_n\|u_n\|_{L^2(\R^2)}$ such that
    \[
    \sup_n \int_{\R^d}|\lab^{\nub}_{w_{\del_n}^c}(\xib)\hat{u}_n(\xib)|^2d\xib\le \tilde{B}^2.
    \]
    As in the proof of \Cref{thm:cpt_seq_nonint_kernel}, it suffices to show 
    \begin{equation}\label{eq:lim_w_del_n}
        \lim_{\tau\to 0}\limsup_{n\to \infty}\sup_{\xib\neq 0}\frac{|\etab_{\tau}(\xib)|}{|\lab^{\nub}_{w_{\del_n}^c}(\xib)|}=0.
    \end{equation}
    Choose $\bar\del_1$ as in \Cref{lem:lowerbd_im_re}. Since $\del_n\to 0$ as $n\to\infty$, there exists $N>0$ such that for any $n> N$, $\del_n\in (0,\bar\del_1)$. For $n>N$, we consider two cases \(|\xib|\le C_0/h(\del_n)\) and \(|\xib|> C_0/h(\del_n)\).

    For \(|\xib|\le C_0/h(\del_n)\), using \eqref{eq:Imest} in \Cref{lem:lowerbd_im_re} and \eqref{etaest1} we know that there exists a constant \(C(d)>0\) such that
    \[\frac{|\etab_\tau(\xib)|}{|\lab^{\nub}_{w_{\del_n}^c}(\xib)|} \le C(d)\frac{\tau|\xib|}{|\xib|}\le C(d)\tau.\]
   Similarly,  for \(|\xib|>C_0/h(\del_n)\), \eqref{eq:Reest} in \Cref{lem:lowerbd_im_re} and \eqref{etaest1} yield
    \[\frac{|\etab_\tau(\xib)|}{|\lab^{\nub}_{w_{\del_n}^c}(\xib)|} \le \frac{C(d)}{C_2(R_1,M_0,d)\int_{|\zb|>R_1h(\del_n)}w_{\del_n}^c(\zb)d\zb}\to 0,\quad n\to\infty,\]
    where we used \eqref{ker_cond_tail_int} and \eqref{eq:kernelassumption_2} to obtain 
    \[    \int_{|\zb|>R_1h(\del_n)}w_{\del_n}^c(\zb)d\zb=\int_{|\zb|>R_1h(\del_n)}w_{\del_n}(\zb)d\zb-\int_{|\zb|>R_2}w_{\del_n}(\zb)d\zb\to\infty,\quad n\to\infty.
    \]
    Therefore, \eqref{eq:lim_w_del_n} holds and the rest of the proof follows is similar to the proof given for the other class of kernels.
\end{proof}

We conclude this section by presenting an application of the compactness result to prove a uniform Poincar\'e inequality applicable the parameterized function spaces $\cS^{\nub}_{w_\del}(\Om)$. The proof of this inequality follows the exact procedures as the proof of 
\Cref{thm:unifPoincare_nonint_kernel} and thus omitted.
\begin{theorem}\label{thm:unifPoincare}
    Let \(\{w_\del:\del\in (0,1)\}\) be a family of kernels given by \Cref{ass:kernel4cptthm}. Assume $\Om$ is a bounded  Lipschitz domain. Then there exist $\del_0>0$ and  $C(\del_0)>0$ such that for any $\del\in (0,\del_0)$, \begin{equation}
        \|u\|_{L^2(\Om)}\le C(\del_0)\|\mathfrak{G}^{\nub}_{w_\del} u\|_{L^2(\R^d;\R^d)},\quad\forall u\in \cS^{\nub}_{w_\del}(\Om).
    \end{equation}
\end{theorem}

\section{Application I: Convergence of parameterized problems and asymptotically compatible schemes}\label{sec:conv_and_AC}
\subsection{Parameterized problems}
In this section, we apply the two compactness results proved in the previous section to study the convergence properties of parameterized problems and the asymptotic compatibility of some numerical schemes for solving these problems. 
Let \(\Om\) be a bounded Lipschitz domain in \(\R^d\) and the function \(A\in L^\infty(\R^d;\R^{d\times d})\) be a symmetric matrix such that there exists a constant \(\mu>0\) such that 
\[\xib^TA(\xb)\xib\ge \mu |\xib|^2,\quad\forall\xb\in\R^d,\ \xib\in \R^d.\]
We would like to investigate the parameterized nonlocal equations
\begin{equation}\label{eq:conv_diff_non_int}
    \left\{\begin{split}
        -\mathfrak{D}^{-\nub}_{w_n}(A\mathfrak{G}^{\nub}_{w_n} u)&=f\ \text{in}\ \Om,\\
        u&=0\ \text{in}\ \R^d\backslash\Om, 
    \end{split}\right.
\end{equation}
where the sequence of kernels $\{w_n\}$ will have some specific properties. From the ellipticity of the coefficient $A$, we see that the energy space for the problem is $\cS^{\nub}_{w_n}(\Om)$. 
The weak form of \eqref{eq:conv_diff_non_int} can be expressed as follows.  Given \(f\in L^2(\Om)\), find \(u\in \cS^{\nub}_{w_n}(\Om)\) such that 
\begin{equation}\label{eq:conv_diff_non_int-weak}
    B_{w_n}(u,v)=(f,v)_{L^2(\Om)},\quad\forall v\in \cS^{\nub}_{w_n}(\Om)
\end{equation}
where the bilinear form 
\(B_{w_n}:\cS^{\nub}_{w_n}(\Om)\times \cS^{\nub}_{w_n}(\Om)\to\R\) is given by
\begin{equation} \label{eq:bilinear_form_w}
    B_{w_{n}}(u,v):=\int_{\R^d}A(\xb)\mathfrak{G}^{\nub}_{w_{n}} u(\xb)\cdot\mathfrak{G}^{\nub}_{w_{n}} v(\xb)d\xb. 
\end{equation}

We analyze these parameterized nonlocal equation for two classes of sequence of kernels. 

{\bf Class I.}   For a kernel \(w(\xb)\) satisfying \Cref{assu:nonint_kernel}, we consider a sequence of radial kernels $\{w_{n}\}$ where 
\[
0\le w_n(\zb)\nearrow w(\zb) \quad \text{as}\quad n\to\infty\quad \text{for a.e.}\ \zb\in\R^d.  
\]
These are exactly the kernels considered in \Cref{thm:cpt_seq_nonint_kernel}.
We also notice that using Lax-Milgram theorem, corresponding to $f\in L^{2}(\Omega)$ and for each $n$, there is a unique solution $u_n\in \cS^{\nub}_{w_n}(\Om)$ to  \eqref{eq:conv_diff_non_int-weak}, see \cite{han2023nonlocal}, with the stability estimate 
\[
|u_n|_{\cS^{\nub}_{w_n}(\Om)} \leq C \|f\|_{L^{2}(\Omega)}
\]
for some $C >0$ (independent of $n$) where we have used the uniform Poincar\'e inequality  \Cref{thm:unifPoincare_nonint_kernel}.  The result we are going to state shortly asserts that $u_n \to u$ strongly in $L^2(\Omega)$ where $u\in \cS^{\nub}_{w}(\Om)$ solves 
\begin{equation}\label{nonint-equation}
 B_{w}(u,v)=(f,v)_{L^2(\Om)},\quad\forall v\in \cS^{\nub}_{w}(\Om).
\end{equation}
Again a unique solution $u$ to \eqref{nonint-equation} exists and, generally, it is expected to be more regular than $u_n$ since $w$ is nonintegrable and $\cS^{\nub}_{w}(\Om)$ is compactly contained in $L^{2}(\Omega)$.
We will also establish connection with discrete approximations of \eqref{eq:conv_diff_non_int-weak} with the solution $u$ of \eqref{nonint-equation}. To that end,  consider the Galerkin approximation of   \eqref{eq:conv_diff_non_int-weak} where we look for  \(u_{n,h}\in V_{n,h}\) such that 
\begin{equation}\label{eq:nonconforming_DG}
    B_{w_n}(u_{n,h},v)= (f,v)_{L^2(\Om)},\quad\forall v\in V_{n,h}.
\end{equation}
where the finite element spaces \(\{V_{n,h}\}_h\subset \cS^{\nub}_{w_n}(\Om)\) are the conforming spaces of piecewise polynomials, i.e, 
\[V_{n,h}:=\{v\in \cS^{\nub}_{w_n}(\Om):v|_T\in P(T),\ \forall T\in \cT_h\},\]
where \(P(T)=\cP_k(T)\) is the space of polynomials on \(T\in \cT_h\) with degree less or equal than a given nonnegative integer \(k\), and \(\cT_h\) is a quasi-uniform triangulation on \(\Om\).

The theorem we state below says that the discrete solutions $u_{n, h} \to u$ in $L^{2}(\Omega)$ and as such the discretized problem \eqref{eq:nonconforming_DG} can be viewed as an approximation to the nonlocal problem \eqref{nonint-equation}. Moreover, since $\mathcal{S}^{\nub}_{w}(\Omega)$ is compactly contained in $L^{2}(\Omega),$ the finite element space $V_{n,h}$ which is conforming in the rough space $\cS^{\nub}_{w_n}(\Om)$, may not be conforming in the regular space $\mathcal{S}^{\nub}_{w}(\Omega)$ for a sufficiently singular kernel $w$. We then conclude that the Galerkin approximation scheme \eqref{eq:nonconforming_DG} can be considered as a (possibly discontinuous) Galerkin (DG) scheme for \eqref{nonint-equation}. 
\begin{theorem}\label{thm:convergence_non_int}
    Assume the degree of polynomials \(k\ge 1\). Let \(\{u_n\}_n\), \(u\) and \(\{u_{n,h}\}_{n,h}\) be the solutions to problems \eqref{eq:conv_diff_non_int-weak}, \eqref{nonint-equation} and \eqref{eq:nonconforming_DG} respectively. Then there exists a constant $C>0$ such that for each \(n\in \N\),
    \begin{equation*}
        \|u_n-u_{n,h}\|_{\cS^{\nub}_{w_n}(\Om)}\le C\inf_{v_{n,h}\in V_{n,h}}\|u_n-v_{n,h}\|_{\cS^{\nub}_{w_n}(\Om)}\to 0,\quad h\to 0,
    \end{equation*}
    and, moreover,
    \begin{equation*}
        \|u_n-u\|_{L^2(\Om)}\to 0,\quad n\to \infty,
    \end{equation*}
    and 
    \begin{equation*}
        \|u_{n,h}-u\|_{L^2(\Om)}\to 0,\quad n\to \infty,\ h\to 0.
    \end{equation*}
\end{theorem}

{\bf Class II.} We next consider a family of kernels \(\{w_{\del_n}\}_{n}\) given in  \Cref{ass:kernel4cptthm} where $\del_n\to 0$ as $n\to \infty$. 

As before, for each $n$, we may apply Lax-Milgram theorem to show that a unique solution $u_{\del_n} \in\cS^{\nub}_{w_{\del_n}}(\Om)$ exists corresponding to $f\in L^{2}(\Omega)$ to  \eqref{eq:conv_diff_non_int-weak} for kernels of this class. There will also be a positive constant $C>0$, independent of $n$, where for each $n$,
\[
|u_{\del_n}|_{\cS^{\nub}_{w_{\del_n}}(\Om)} \leq C \|f\|_{L^{2}(\Omega)}.  
\]
For the above estimate we have used the uniform Poincar\'e inequality \Cref{thm:unifPoincare}. We may now use the compactness result \Cref{thm:cpt_full} to deduce the existence of $u_0\in H^{1}_{0}(\Om)$ such that (up to a subsequence) $u_{\del_n}\to u_0$ in $L^{2}(\Omega)$ as $n\to \infty$. As we will state in the next theorem, it turns out that  \(u_0\in H^1_0(\Om)\) solves 
\begin{equation}\label{eq:loc_conv_diff}
    B_0(u_0,v)=(f,v)_{L^2(\Om)},\quad\forall v\in H^1_0(\Om),
\end{equation}
where \(B_0(u,v):=\int_{\Om} A(\xb)\nabla u(\xb)\cdot \nabla v(\xb)d\xb\) for any \(u,v\in H^1_0(\Om)\). Moreover, considering a  conforming finite element spaces \(\{V_{{\del_n},h}\}_h\subset \cS^{\nub}_{w_{\del_n}}(\Om)\) which are the spaces of piecewise polynomials of degree less or equal than \(k\), i.e, 
\[V_{{\del_n},h}:=\{v\in \cS^{\nub}_{w_{\del_n}}(\Om):v|_K\in \cP_k(K),\ \forall K\in \cT_h\},\]
and the Galerkin approximation  \(u_{{\del_n},h}\in V_{\del,h}\) such that 
\begin{equation}\label{eq:Galerkin_conv_diff_delta}
    B_{\del_{n}}(u_{{\del_n},h},v)= (f,v)_{L^2(\Om)},\quad\forall v\in V_{{\del_n},h},
\end{equation}
the theorem asserts that $u_{n, h} \to u_0$ as $n\to \infty$ and $h\to 0$. The following theorem states these results precisely. 

\begin{theorem}\label{thm:convergence_delta}
Assume that $\{w_{\del_n}\}_{n}$ is given in  \Cref{ass:kernel4cptthm} where $\del_n\to 0$ as $n\to \infty$. 
    Assume also that the degree of polynomials \(k\ge 1\). Let \(\{u_{\del_n}\}_n\), \(u_0\) and \(\{u_{{\del_n},h}\}_{n,h}\) be the solutions to problems \eqref{eq:conv_diff_non_int-weak},  \eqref{eq:loc_conv_diff} and \eqref{eq:Galerkin_conv_diff_delta} respectively. Then there exists a constant $C>0$ such that for each \(n\in \N_+\), 
    \begin{equation*}
        \|u_{\del_n}-u_{\del_n,h}\|_{\cS^{\nub}_{w_{\del_n}}(\Om)}\le C\inf_{v_{n,h}\in V_{\del_n,h}}\|u_{\del_n}-v_{n,h}\|_{\cS^{\nub}_{w_{\del_n}}(\Om)}\to 0,\quad h\to 0,
    \end{equation*}
    and, moreover,
    \begin{equation*}
        \|u_{\del_n}-u_0\|_{L^2(\Om)}\to 0,\quad n\to \infty,
    \end{equation*}
    and 
    \begin{equation*}
        \|u_{\del_n,h}-u_0\|_{L^2(\Om)}\to 0,\quad n\to \infty,\ h\to 0.
    \end{equation*}
\end{theorem}
We defer the proofs of \Cref{thm:convergence_non_int} and \Cref{thm:convergence_delta} to the next subsection but note that these two theorems are particular instances of a more abstract result obtained in \cite{tian2014asymptotically, tian2015nonconforming} on the convergence of asymptotically compatible (AC) schemes for parameterized linear problems. A slightly modified version of the abstract formulation is presented next. 
\subsection{Asymptotically compatible schemes for linear equations with bounded and measurable coefficients}

Let \(\{\cS_{\sig}:\sig\in \Sig\}\) be a family of Hilbert spaces over \(\R\), where the index set \(\Sig\) is a subset of \(\R\) such that $0\in\Sig$ and \(\sup\Sig=\infty\). Here we identify the dual space of $\cS_0$ with itself, $\cS_0^*=\cS_0$. (In applications we will take $\cS_0$ as the space of $L^2$ functions.) Let \(\cS_{\infty}\) and \(\cH\) be two Hilbert spaces over \(\R\) which are continuously embedded into \(\cS_0\). Associated with the spaces \(\{\cS_\sig:\sig\in\Sig\cup \{\infty\}\}\), there are bilinear forms \(\{B_\sig:\cS_\sig\times \cS_\sig\to\R\}\) respectively. 

Given $f\in\cS_0$, we consider the parameterized variational problems: for each \(\sig\in \Sig\), find \(u_\sig\in \cS_\sig\) such that 
\begin{equation}\label{eq:nonlocal}
    B_\sig(u_\sig,v)=(f,v)_{\cS_0},\quad\forall v\in \cS_\sig,
\end{equation}
and the limiting problem: find \(u_\infty\in \cS_\infty\) such that
\begin{equation}\label{eq:local}
    B_\infty(u_\infty,v)=(f,v)_{\cS_0},\quad\forall v\in \cS_\infty.
\end{equation}
Let \(\{V_{\sig,h}\subset \cS_{\sig}:\sig\in\Sig, h\in (0,h_0)\}\) be a family of finite-dimensional Hilbert spaces and consider  the Galerkin approximation of \eqref{eq:nonlocal}: for each \(\sig\in \Sig\) and \(h\in (0,h_0)\), find \(u_{\sig,h}\in V_{\sig,h}\) such that
\begin{equation}\label{eq:Galerkin}
    B_\sig(u_{\sig,h},v)=(f,v)_{\cS_0},\quad\forall v\in V_{\sig,h}.
\end{equation}
The main question we would like to address is whether $u_{\sigma, h} \to u_\infty$ in $\cS_0$ as $(\sigma, h) \to(\infty, 0)$. As has been discussed extensively in \cite{tian2014asymptotically, tian2015nonconforming}, in general the answer is negative. However, under addition assumptions, convergence is shown to hold.

\begin{assumption}\label{ass:generalAC}
    Assume that there exist four positive constants \(M_1\), \(M_2\), \(C_1\) and \(C_2\) independent of \(\sig\) such that the following conditions hold:
\begin{enumerate}[label=(\Alph*)]
    \item Uniform embeddings and asymptotically compact embeddings:
    \begin{enumerate}[label=(\Alph{enumi}\roman*)]
        \item For any \(\sig\in \Sig\), $\cH\subset \cS_\sig$, \(M_1\|u\|_{\cS_0}\le \|u\|_{\cS_\sig}\) for any \(u\in\cS_\sig\),   and \(\|u\|_{\cS_\sig}\le M_2\|u\|_{\cH}\) for any \(u\in\cH\).
        \item For any set \(\{u_\sig:\sig\in \Sig\}\), if \(\sup\{\|u_\sig\|_{\cS_\sig}:\sig\in \Sig\}\) is finite, then \(\{u_\sig:\sig\in \Sig\}\) is relatively compact in \(\cS_0\) with each limit point in \(\cS_\infty\).
    \end{enumerate}
    \item Boundedness and coercivity of the bilinear forms \(\{B_\sig:\sig\in\Sig\cup\{\infty\}\}\):
    \begin{enumerate}[label=(\Alph{enumi}\roman*)]
        \item \(|B_\sig(u,v)|\le C_2\|u\|_{\cS_\sig}\|v\|_{\cS_\sig}\), \(\forall u,v\in\cS_\sig\), \(\forall\sig\in\Sig\cup\{\infty\}\).
        \item \(B_\sig(u,u)\ge C_1\|u\|_{\cS_\sig}^2\), \(\forall u\in \cS_\sig\), \(\forall\sig\in \Sig\cup\{\infty\}\).
    \end{enumerate}
    \item Existence of a subset \(\cS_*\) of \(\cH\) such that
    \begin{enumerate}[label=(\Alph{enumi}\roman*)]
        \item \(\cS_*\subset \cap_{\sig\in\Sig\cup\{\infty\}}\cS_\sig\) and \(\cS_*\) is dense in \(\cS_\sig\) for each \(\sig\in\Sig\cup\{\infty\}\) with respect to their norms,
        \item for any sequence \(\{u_{\sig_k}\in \cS_{\sig_k}\}_{k\in\N_+}^{\Sig\ni \sig_k\to\infty}\) and \(u\in \cS_\infty\) satisfying 
        \[\sup_{k\in\N_+}\|u_{\sig_k}\|_{\cS_{\sig_k}}<\infty\]
        and \(u_{\sig_k}\to u\) in \(\cS_0\) as \(k\to\infty\) one has 
        \[B_{\sig_k}(u_{\sig_k},\phi)\to B_\infty(u,\phi),\quad k\to\infty\] 
        for any \(\phi\in S_*\).
    \end{enumerate}
    \item Approximation of nonlocal space and asymptotic density in limiting space:
    \begin{enumerate}[label=(\Alph{enumi}\roman*)]
        \item For given \(\sig\in \Sig\) and \(v\in \cS_\sig\), \(\inf_{v_{\sig,h}\in V_{\sig,h}}\|v-v_{\sig,h}\|_{V_{\sig,h}}\to 0\) as \(h\to 0\).
        \item For any \(v\in \cS_*\), there exists \(\{v_k\in V_{\sig_k,h_k}\}^{\sig_k\to\infty}_{h_k\to 0}\) such that  \(\|v-v_k\|_{\cH}\to 0\) as \(k\to\infty\).
    \end{enumerate}
\end{enumerate}
\end{assumption}

It should be noted that assumptions (A) and (C) in the above relax the conditions specified in the original paper \cite{tian2014asymptotically}. Specifically, assumption (A) eliminates the necessity for the continuous embedding of 
$\mathcal{S}_\infty$ into  $\mathcal{S}_\sigma$ and assumption (C) substitutes the requirement for strong convergence of operators on a dense subset with the requirement for the convergence of bilinear forms in a weak sense. We refer the readers to a recent work \cite{du2024asymptotically} that discusses extensions of the original AC framework to study nonlinear problems.

\begin{theorem}[Main Convergence Theorem \cite{tian2014asymptotically}]\label{thm:main_convergence}
    Suppose \Cref{ass:generalAC} holds and let \(\{u_\sig\}\), \(u_\infty\) and \(\{u_{\sig,h}\}\) be the solutions to problems \eqref{eq:nonlocal}, \eqref{eq:local} and \eqref{eq:Galerkin} respectively. Then given properties (A)-(D), we have that for each \(\sig\in \Sig\),
    \begin{equation}
        \|u_\sig-u_{\sig,h}\|_{\cS_\sig}\le \frac{C_2}{C_1}\inf_{v_{\sig,h}\in V_{\sig,h}}\|u_\sig-v_{\sig,h}\|_{\cS_\sig}\to 0,\quad h\to 0,
    \end{equation}
    and, moreover,
    \begin{equation}\label{eq:loc_lim}
        \|u_\sig-u_\infty\|_{\cS_0}\to 0,\quad \sig\to \infty,
    \end{equation}
    and 
    \begin{equation}\label{eq:AC_lim}
        \|u_{\sig,h}-u_\infty\|_{\cS_0}\to 0,\quad\sig\to \infty,\ h\to 0.
    \end{equation}
\end{theorem}
\begin{proof}
    We sketch the proof briefly. The first inequality is the best approximation and the convergence follows from the standard conforming Galerkin approximation theory. 
    To show \eqref{eq:loc_lim}, first note that 
    \[\sup\{\|u_\sig\|_{\cS_\sig}:\sig\in\Sig\}\le \frac{\|f\|_{\cS_0}}{C_1}\]
    by coercivity (Bii) and embedding of \(\cS_\sig\) into \(\cS_0\). Then by asymptotically compact embedding property (Aii) we know there exists \(\tilde{u}\in \cS_0\) such that \(u_\sig\to \tilde{u}\) as \(\sig\to\infty\) up to a subsequence. Then it suffices to show that \(\tilde{u}\) is the unique solution to \eqref{eq:local} to conclude \eqref{eq:loc_lim}. By density of \(\cS_*\) in \(\cS\), it suffices to show \eqref{eq:local} holds (with \(\tilde{u}\) in place of \(u_\infty\)) for any \(v\in\cS_*\). Fix some \(v\in\cS_*\), since \(\cS_*\subset \cS_\sig\) for each \(\sig\in \Sig\), one knows that 
    \[B_\sig(u_\sig,v)=(f,v)_{\cS_0}.\]
    Letting \(\sig\to\infty\) and using (Cii) yield \(B_\infty(\tilde{u},v)=(f,v)_{\cS_0}\) as desired.

    Finally, to show \eqref{eq:AC_lim}, first use the bound \[\sup\{\|u_{\sig,h}\|_{\cS_\sig}:\sig\in\Sig,\ h\in (0,h_0)\}\le \frac{\|f\|_{\cS_0}}{C_1}\]
    and asymptotically compact embedding property to show there exists \(\tilde{u}\in \cS_0\) such that \(u_{\sig,h}\to \tilde{u}\) as \(\sig\to\infty, h\to 0\) up to a subsequence. 
    Then again by density it suffices to show that \(B_\infty(\tilde{u},v)=(f,v)_{\cS_0}\) for any \(v\in \cS_*\). Fix \(v\in\cS_*\), by asymptotic density, there exists \(\{v_k\in V_{\sig_k,h_k}\}^{\sig_k\to\infty}_{h_k\to 0}\) such that \(\|v-v_k\|_{\cH}\to 0\) as \(k\to\infty\). Note that 
    \[B_{\sig_k}(u_{\sig_k,h_k},v_k)=(f,v_k)_{\cS_0}.\]
    Then the final conclusion follows from the computation below:
    \[\begin{split}
        &|B_\infty(\tilde{u},v)-(f,v)_{\cS_0}|\\
        \le &|B_\infty(\tilde{u},v)-B_{\sig_k}(u_{\sig_k,h_k},v)|+|B_{\sig_k}(u_{\sig_k,h_k},v-v_k)| + |B_{\sig_k}(u_{\sig_k,h_k},v_k)-(f,v)_{\cS_0}|\\
        \le &|B_\infty(\tilde{u},v)-B_{\sig_k}(u_{\sig_k,h_k},v)|+C_2\|u_{\sig_k,h_k}\|_{\cS_{\sig_k}}\|v-v_k\|_{\cS_{\sig_k}}+|(f,v-v_k)_{\cS_0}|\\
        \le &B_\infty(\tilde{u},v)-B_{\sig_k}(u_{\sig_k,h_k},v)|+C_2\frac{\|f\|_{\cS_0}}{C_1}M_2\|v-v_k\|_{\cH}+C_{\cH}\|f\|_{\cS_0}\|v-v_k\|_{\cH}\\
        \to &0,\quad k\to\infty,
    \end{split}\]
    where we used (Cii), (Ai) and the embedding inequality \(\|v\|_{\cS_0}\le C_{\cH}\|v\|_{\cH}\) for any \(v\in\cH\).
\end{proof}

With this abstract result at hand, we are now ready to prove \Cref{thm:convergence_non_int} and \Cref{thm:convergence_delta}.

\begin{proof}[Proof of \Cref{thm:convergence_non_int}]
    It suffices to verify assumptions (A)-(D) in \Cref{thm:main_convergence}. Using the notations in \Cref{thm:main_convergence}, we let \(\Sig=\N_+\), \(\cS_\sig=\cS^{\nub}_{w_n}(\Om)\), \(\cS_\infty=\cS^{\nub}_w(\Om)\), \(\cS_*=C^\infty_c(\Om)\) and \(\cH=H_{0}^{1}(\Omega)\), where functions in $H_{0}^{1}(\Omega)$
    are understood to be in $H^1(\R^d)$ and vanish outside of $\Omega$ as the boundary of $\Omega$ is regular enough \cite[Theorem~3.7]{wloka1987partial}.  Then \(M_1=1\), \(C_2=\|A\|_{L^\infty(\R^d;\R^{d\times d})}\) and \Cref{thm:cpt_seq_nonint_kernel} verifies (Aii). 
    
    The second inequality in (Ai) amounts to showing that there exists $M_2>0$ such that 
     \[\|u\|_{\cS^{\nub}_{w_n}(\Om)}\le M_2\|u\|_{H^1(\R^d)},\quad\forall u\in H_{0}^{1}(\Omega).\]
   But this follows from  \Cref{prop:SequalH} where the norm of \(\cS^{\nub}_{w_n}(\Om)\) can be expressed as 
    \[\|u\|_{\cS^{\nub}_{w_n}(\Om)}^2=\|u\|_{\cS^{\nub}_{w_n}(\R^d)}^2=\int_{\R^d}\left(1+|\lab^{\nub}_{w_n}(\xib)|^2\right)|\hat{u}(\xib)|^2d\xib,\quad\forall u\in \cS^{\nub}_{w_n}(\Om).\]
    Now using \eqref{Fsym0thorderest}, one obtains that
    \begin{equation}
        |\bm\lambda_{w_n}^{\nub}(\xib)|\le 2\sqrt{2}\pi M_{w_n}^1|\xib|+\sqrt{2}M_{w_n}^2\le 2\sqrt{2}\pi M_{w}^1|\xib|+\sqrt{2}M_{w}^2,\quad\xib\in\R^d.
    \end{equation}
   Thus the second inequality in (Ai) holds true for some \(M_2\) independent of \(n\) but dependent of \(w\),  using the fact that 
    \[\|u\|_{H^1(\R^d)}^2=\int_{\R^d}\left(1+|\xib|^2\right)|\hat{u}(\xib)|^2d\xib.\]
    
    For (Bii), one may choose \(C_1\) independent of \(N\) but dependent of some \(N_0\) determined by \Cref{thm:unifPoincare_nonint_kernel}. Note that (Ci) follows immediately from \cite[Theorem~3.2]{han2023nonlocal}. 
    
    To show (Cii), suppose \(v_n\in \cS^{\nub}_{w_n}(\Om)\),  \(\sup_{n\in\N_+}\|v_n\|_{\cS^{\nub}_{w_n}(\Om)}<\infty\), and \(v_n\to v\) in \(L^2(\Om)\) for \(v\in \cS^{\nub}_w(\Om)\). We need to prove that for any \(\phi\in C^\infty_c(\Om)\),
    \begin{equation}\label{eq:lim_bilinear_form-41} (A\mathfrak{G}^{\nub}_{w_n}v_n,\mathcal{G}^{\nub}_{w_n}\phi)_{L^2(\R^d;\R^d)}\to (A\mathfrak{G}^{\nub}_w v,\mathcal{G}^{\nub}_w \phi)_{L^2(\R^d;\R^d)},\quad n\to\infty.
    \end{equation}
    To that end, first, using (2.8) and (2.9) from \cite[Lemma~2.1]{han2023nonlocal}, one may show that \(\cG^{\nub}_{w_n}\phi\to \cG^{\nub}_w\phi\) in \(L^2(\R^d;\R^d)\) as \(n\to\infty\), and if \(\psib\in C^\infty_c(\Om;\R^d)\) then \(\cD^{-\nub}_{w_n}\psib\to\cD^{-\nub}_w\psib\) in \(L^2(\R^d)\) as \(n\to\infty\). Second, we show that \(\mathfrak{G}^{\nub}_{w_n}v_n\rightharpoonup \mathfrak{G}^{\nub}_w v\) in \(L^2(\R^d;\R^d)\). If this holds, then since \(A\in L^\infty(\R^d;\R^{d\times d})\), it follows that \(A\mathfrak{G}^{\nub}_{w_n}v_n\rightharpoonup A\mathfrak{G}^{\nub}_w v\) in \(L^2(\R^d;\R^d)\), hence, combing the first step, one can conclude that \eqref{eq:lim_bilinear_form-41} holds, and consequently, (Cii) holds. To show \(\mathfrak{G}^{\nub}_{w_n}v_n\rightharpoonup \mathfrak{G}^{\nub}_w v\) in \(L^2(\R^d;\R^d)\), fix some \(\wb\in L^2(\R^d;\R^d)\) and \(\ep>0\), by density, there exists \(\psib\in C^\infty_c(\R^d;\R^d)\) such that \(\|\wb-\psib\|_{L^2(\R^d;\R^d)}<\ep\). Then 
    \[\begin{split}
        &\left|(\mathfrak{G}^{\nub}_{w_n}v_n,\wb)_{L^2(\R^d;\R^d)}-(\mathfrak{G}^{\nub}_{w}v,\wb)_{L^2(\R^d;\R^d)}\right|\\
        \le &\left|(\mathfrak{G}^{\nub}_{w_n}v_n,\wb-\psib)_{L^2(\R^d;\R^d)}\right|+\left|(\mathfrak{G}^{\nub}_{w_n}v_n,\psib)_{L^2(\R^d;\R^d)}-(\mathfrak{G}^{\nub}_{w}v,\psib)_{L^2(\R^d;\R^d)}\right|\\
        &\quad +\left|(\mathfrak{G}^{\nub}_{w}v,\psib-\wb)_{L^2(\R^d;\R^d)}\right|\\
        \le &\ep\sup_{n\in\N_+}\|v_n\|_{\cS^{\nub}_{w_n}(\Om)}+|(v_n,\cD^{-\nub}_{w_n}\psib)_{L^2(\R^d)}-(v,\cD^{-\nub}_w \psib)_{L^2(\R^d)}|+\ep\|v\|_{\cS^{\nub}_w(\Om)}
    \end{split}\]
    Since \(\cD^{-\nub}_{w_n}\psib\to\cD^{-\nub}_w\psib\) in \(L^2(\R^d)\), together with \(v_n\to v\) in \(L^2(\R^d)\), one can conclude that \((v_n,\cD^{-\nub}_{w_n}\psib)_{L^2(\R^d)} \to(v,\cD^{-\nub}_w \psib)_{L^2(\R^d)}\) as \(n\to\infty\) and thus \(\mathfrak{G}^{\nub}_{w_n}v_n\rightharpoonup \mathfrak{G}^{\nub}_w v\) in \(L^2(\R^d;\R^d)\) and (Cii) holds.

    Finally, since \(\cH=H_{0}^{1}(\Omega)\)
    and \(\{V_{n,h}\}_{n,h}\) are spaces of piecewise polynomials of degree \(k\ge 1\), one may show that (D) is satisfied as in \cite[Theorem~3.8]{tian2014asymptotically}.
\end{proof}

\begin{proof}[Proof of \Cref{thm:convergence_delta}]
    The proof is quite similar to that of \Cref{thm:convergence_non_int}. We only point out a few key differences. Using the notations in \Cref{thm:main_convergence}, we let \(\Sig=(\del_0^{-1},\infty)\), \(\cS_\sig=\cS^{\nub}_{w_{1/\sig}}(\Om)\), \(\cS_\infty=H^1_0(\Om)\), \(\cS_*=C^\infty_c(\Om)\) and \(\cH=H_{0}^{1}(\Omega)\).  
    Then \Cref{thm:cpt_full} verifies (Aii). 
    
    For the second inequality in (Ai), one may argue as before for 
    \(\{w_\del\}\) in \Cref{ass:kernel4cptthm} by proving
    \[|\lab^{\nub}_{w_\del}(\xib)|\le D_1|\xib|+D_2,\quad\forall\xib\in\R^d\]
    where \(D_1,D_2\ge 0\) are constants independent of \(\del\). In fact the inequality is even true for kernels satisfying just \Cref{eq:kernelassumption,eq:kernelassumption_2,eq:kernelassumption_3} and follows from \eqref{Fsym0thorderest}, where 
    \[
    |\lab^{\nub}_{w_\del}(\xib)|\le 2\sqrt{2}\pi M^1_{w_{\delta}}|\xib|+\sqrt{2}M_{w_\delta}^2,\quad\forall\xib\in\R^d
    \] and that 
    \[
    M^1_{w_{\delta}} \to 2d \quad \text{and}\quad M_{w_\delta}^2\to 0\quad \text{as $\delta\to 0$.}
    \]
    
    Assumption (Bii) can be verified by \Cref{thm:unifPoincare}. 

    To show (Cii), suppose \(v_n\in \cS^{\nub}_{w_{\del_n}}(\Om)\),  \(\sup_{n\in\N_+}\|v_n\|_{\cS^{\nub}_{w_{\del_n}}(\Om)}<\infty\), and  \(v_n\to v\) in \(L^2(\Om)\) for \(v\in H^1_0(\Om)\). We need to prove that for any \(\phi\in C^\infty_c(\Om)\) as \(n\to\infty\)  (\(\del_n\to 0)\) 
    \begin{equation}\label{eq:lim_bilinear_form}
(A\mathfrak{G}^{\nub}_{w_{\del_n}}v_n,\mathcal{G}^{\nub}_{w_{\del_n}}\phi)_{L^2(\R^d;\R^d)}\to (A\nabla v,\nabla \phi)_{L^2(\R^d;\R^d)},\quad n\to\infty.
    \end{equation}
    This holds true by the same argument for \eqref{eq:lim_bilinear_form-41} if one notices that, according to \Cref{prop:locallim_ptws_lp}, \(\mathcal{G}^{\nub}_{w_{\del_n}}\phi\to \nabla \phi\) in \(L^2(\R^d;\R^d)\), and if \(\psib\in C^\infty_c(\Om;\R^d)\) then \(\cD^{-\nub}_{w_{\del_n}}\psib\to\divv\psib\) in \(L^2(\R^d)\) as \(n\to\infty\). 
    
    The rest of the proof is similar to that of \Cref{thm:convergence_non_int} and thus omitted.
\end{proof}

\section{Application II.  Optimal Control Problem of a Nonlocal Equation}\label{sec:optimal_control}
\subsection{Parameterized continuous and discrete optimal control problems}
In this section, we apply the compactness results proved in \Cref{sec:cpt} to analyze the well-posedness, the numerical approximation, and the limiting behavior of solutions of a parameterized optimal control problem of linear nonlocal equations.  Optimal control problems with constraints that involve nonlocal equations have been a focus of recent research interest. While this work is inspired by work \cite{mengesha2023optimal} which deals with an optimal control problem of linearized peridynamics, we should mention that such types of problems where the state equation is either a fractional or nonlocal equation have been investigated in literature. To cite a few, the paper \cite{d2014optimal} rigorously analyzed optimal control problems when the state equation is of the form 
\begin{equation}\label{NL-old-form}
-\mathcal{L}_{\delta}u=g,\quad \text{where\quad  $\mathcal{L}_{\delta}u(x) = 2\int_{\Omega}(u(\yb)-u(\xb))\gamma_\delta(\xb,\yb) d\yb$ }
\end{equation}
subject to some volumetric boundary condition outside of $\Omega$. In the above, $\gamma_\delta(\xb, \yb)$ serves as a kernel. The paper also presents a finite element approximation and numerical simulations which illustrate the theoretical results. The papers \cite{munoz2022local,MR4118336} have also studied optimal control problems with the state equation being the parameterized nonlocal equations of the form \eqref{NL-old-form} and demonstrated rigorously the convergence of optimal pairs to  an optimal pair of an optimal control problem with a local (PDE) based state equation. Mathematical analysis and numerical approximation of optimal control problems with the fractional equation as the state equation have also be been investigated in  
\cite{antil2017note, antil2020optimal, antil2022optimal, burkovska2021optimization, d2019priori,otarola2019maximum}.  
To put our work in a clear perspective, we should mention that the parameterized state equation constraints we will be dealing with is different from the classical fractional equations and the nonlocal equations of the form \eqref{NL-old-form}. 

To properly describe the problem we study,  
we assume that \(\{w_\del\}_{\del>0}\) is a family of kernels given in \Cref{ass:kernel4cptthm}. We fix \(A\in L^\infty(\R^d;\R^{d\times d})\) which is uniformly elliptic and symmetric. That is, for any $\xb\in \Om$, $A(\xb) = A(\xb)^{T}$ and  for some  \(\mu>0\) we have  
\[\xib^TA(\xb)\xib\ge \mu |\xib|^2,\quad\forall\xb\in\R^d,\ \xib\in \R^d.\]
Given  $\Om$, a bounded polygonal domain in $\R^d$, and $\nub\in \mathbb{S}^{d-1}$, consider the parameterized nonlocal equations
\begin{equation}\label{eq:conv_diff_non_int-optimal}
    \left\{\begin{split}
        -\mathfrak{D}^{-\nub}_{w_\delta}(A\mathfrak{G}^{\nub}_{w_\delta} u)&=f\ \text{in}\ \Om,\\
        u&=0\ \text{in}\ \R^d\backslash\Om, 
    \end{split}\right.
\end{equation}
where the right-hand side data $f$ comes from an admissible class that satisfies a box condition. Given the functions $\alpha, \beta \in C(\overline{\Omega})$ such that $\alpha(\xb) <\beta(\xb)$ for all $\xb\in \Omega$,
we introduce the admissible class of right-hand side 
\[
Z_{\mathrm{ad}} = \{g\in L^{2}(\Omega): \alpha(\xb) \leq g(\xb) \leq \beta(\xb),\,\,\text{a.e. } \xb\in\Omega\}.
\]
We remark that there exists a constant $B>0$ such that
\begin{equation}\label{eq:bound_Z_ad}
    \sup_{g\in Z_{\mathrm{ad}}}\|g\|_{L^2(\Om)}\le B,
\end{equation}
since for any $g\in Z_{\mathrm{ad}}$, $|g(\xb)|\leq \max\{|\alpha(\xb)|, |\beta(\xb)|\}$ a.e. $\xb\in\Om$ and $\Om$ is bounded. 
We note that corresponding to each $f\in L^2(\Om)$, a unique solution $u_{\delta}(f) \in \mathcal{S}_{w_\delta}^{\nub}(\Omega)$ exists that solve the weak form of \eqref{eq:conv_diff_non_int-optimal}
 \begin{equation*}
        B_\del(u_\del,v) =\langle  f,v\rangle,\quad \forall v\in \cS^{\nub}_{w_\del}(\Om), 
\end{equation*}
where $B_\del:\cS^{\nub}_{w_\del}(\Om)\times \cS^{\nub}_{w_\del}(\Om)\to \R$ is the bilinear form given in  \eqref{eq:bilinear_form_w} with $w_\del$ in place of $w_n$ and \(\langle \cdot,\cdot\rangle\) is the inner product on \(L^2(\Om)\).
The solution map $\mathfrak{S}_\delta:L^{2}(\Omega) \to \mathcal{S}_{w_\delta}^{\nub}(\Omega)$ given by $\mathfrak{S}_\del(f)= u_\delta(f)$ is a bounded linear map with the estimate 
\begin{equation}\label{eq:sol_op_energy_norm}
    |\mathfrak{S}_\del(f)|_{\mathcal{S}_{w_\delta}^{\nub}(\Omega)}=|u_\delta(f)|_{\mathcal{S}_{w_\delta}^{\nub}(\Omega)} \leq \frac{C(\del_0)}{\mu} \|f\|_{L^{2}(\Omega)},\quad \forall f\in L^2(\Om),\ \del\in (0,\del_0),
\end{equation}
where the constants \(\del_0>0\) and \(C(\del_0)>0\) are chosen as in \Cref{thm:unifPoincare}. Indeed, By taking the test function \(v=u_\del\) in the weak form of \eqref{eq:conv_diff_non_int-optimal} and using the ellipticity of $A$, we know 
    \[\mu |u_\del|_{\cS^{\nub}_{w_\del}(\Om)}^2\le B_\del(u_\del,u_\del)=\langle  f,u_\del\rangle\le \| f\|_{L^2(\Om)}\|u_\del\|_{L^2(\Om)}.\]
    Combing this with the inequality $\|u_\del\|_{L^2(\Om)}\le C(\del_0)|u_\del|_{\cS^{\nub}_{w_\del}(\Om)}$ from \Cref{thm:unifPoincare}, one obtains \eqref{eq:sol_op_energy_norm} and, in addition, the following $L^2$ estimate
    \begin{equation}\label{eq:sol_op_L2_norm}
        \|\mathfrak{S}_\del(f)\|_{L^2(\Om)}\le \frac{C(\del_0)^2}{\mu} \|f\|_{L^{2}(\Omega)},\quad \forall f\in L^2(\Om),\ \del\in (0,\del_0).
    \end{equation}
    Hence $\mathfrak{S}_\del:L^2(\Om)\to L^2(\Om)$ is a bounded linear operator and $\|\mathfrak{S}_\del\|_{L^2(\Om)\to L^2(\Om)}\le C(\del_0)^2/\mu$ for all $\del\in (0,\del_0)$.

As we vary $f$ in $Z_{\mathrm{ad}}$,  the goal of the optimal control problem is to find a pair $(u_\delta(f), f) \in \mathcal{S}_{w_\delta}^{\nub}(\Omega)\times Z_{\mathrm{ad}}$, called the \textit{optimal solution pair}, that minimizes a certain objective functional $I_\delta(u, g):\mathcal{S}_{w_\delta}^{\nub}(\Omega) \times Z_{\mathrm{ad}} \to \mathbb{R}$. In this case, the right-hand side $f$ will serve as a \textit{control variable} and $u$ will be the \textit{state variable} satisfying the state equation \eqref{eq:conv_diff_non_int-optimal}. Following the work in \cite{mengesha2023optimal}, we consider the objective functional $I_\delta = I|_{\mathcal{S}_{w_\delta}^{\nub}(\Omega) \times Z_{\mathrm{ad}}}$ where $I: L^{2}(\Omega)\times Z_{\mathrm{ad}}\to \mathbb{R}$ is given by 
\[I(u, g ):=\int_{\Om} F(\xb,u(\xb))d\xb +\frac{\lambda}{2}\int_{\Om} \Ga(\xb)| g (\xb)|^2 d\xb,\]
where \(\lambda> 0\), \(\Ga\in L^1(\Om)\) is positive and \(F:\Om\times \R\to \R\) satisfies the following properties:
\begin{enumerate}
    \item For all \(\xi\in\R\) the mapping \(\xb\mapsto F(\xb,\xi)\) is measurable.
    \item For a.e. \(\xb\in\Om\) the mapping \(\xi\mapsto F(\xb,\xi)\) is convex (and continuous).
    \item There exist a constant \(c_1>0\) and a function \(\ell\in L^1(\Om)\) for which \[|F(\xb,\xi)|\le c_1|\xi|^2+\ell(\xb)\] for a.e. \(\xb\in \Om\) and all \(\xi\in\R\).
    \item There exist two functions \(c\in L^1(\Om)\) and \(d\in L^2(\Om)\) such that \[F(\xb,\xi)\ge c(\xb)+d(\xb)  \xi\]
    for a.e. \(\xb\in \Om\) and all \(\xi\in\R\).
    \item \(F\) is continuously differentiable in the second argument. That is, for a.e. $\xb\in\Om$, \( F_\xi(\xb, \cdot) \in C^0(\R)\) where 
    \( F_\xi (\xb,\xi):=\frac{\partial F}{\partial \xi}(\xb,\xi)\) for \(\xb\in\Om\) and \(\xi\in\R\).
    Moreover, there exist a constant \(c_2>0\) and a function $e\in L^{2}(\Omega)$ such that 
    \begin{equation}\label{eq:F_xi_est}
        | F_\xi (\xb,\xi)|\le c_2|\xi| + e(\xb),\quad\text{a.e.}\ \xb\in \Om,\ \forall\xi\in \R.
    \end{equation}
\end{enumerate}
A typical objective functional is $
I(u, g) = \|u - u_{\mathrm{des}}\|_{L^{2}(\Om)}^{2} + \frac{\lambda}{2}\|g\|_{L^{2}}^{2}$ where we are looking for an optimal pair  $(u, g)$ where the control has the smallest $L^{2}(\Om)$-norm and the state $u$ is the closest to a given desired state $u_{\mathrm{des}} \in L^{2}(\Omega)$. 

There are now a number of outstanding questions we would like to address in this section. First, we would like to address whether the problem is well-posed, that is, whether there is a unique optimal pair $(\overline{u_{\delta}}, \overline{f_{\delta}})$ to the optimal control problem. Second, we would like to study the asymptotic behavior, with respect to appropriate topology, of the sequence of optimal pairs $(\overline{u_{\delta}}, \overline{f_{\delta}})$ as $\delta\to 0,$ and if there is a limit point $(u, f)$, we would like to determine whether it solves an optimal control problem. And finally, we would like to study the discretization of the parameterized optimal control problems and establish the existence of asymptotically compatible schemes. 

The following result states the well-posedness of the parameterized optimal control problem as well as the convergence of solutions. It is one of the main results of this section and its proof is deferred for later. 
\begin{theorem}[Continuous problems and convergence of solutions]\label{conts-summary-optimal}
Assume all the conditions we stated at the beginning of this section on $\Omega$, the sequence of kernels $\{w_\delta\}_{\delta>0}$,  the control admissible set $Z_{\mathrm{ad}}$,  and the objective functional $I$. Then for each $\delta>0,$ there exists a unique optimal pair 
\((\overline{u_\del},\overline{ g _\del})\in \cS^{\nub}_{w_\del}(\Om)\times Z_{\mathrm{ad}}\) such that
    \begin{equation}\label{prob:non_con}
        I(\overline{u_\del},\overline{ g _\del}) = \min I(u_\del, g)
    \end{equation}
 where the minimization is over pairs \((u_\del, g)\in \cS^{\nub}_{w_\del}(\Om)\times Z_{\mathrm{ad}}\) that satisfy
    \begin{equation}\label{nonloc_state_eq}
        B_\del(u_\del,v) =\langle  g,v\rangle,\quad \forall v\in \cS^{\nub}_{w_\del}(\Om).
    \end{equation}
    Moreover, there exists a  unique pair $(\overline{u}, \overline{g})\in H^{1}_{0}(\Omega)\times Z_{\mathrm{ad}}$ such that \(\overline{u_\del}\to \overline{u}\) in \(L^2(\Om)\) and \(\overline{ g _\del}\rightharpoonup  \overline{ g }\) in \(L^2(\Om)\) and  \((\overline{u},\overline{ g })\) solves 
    \begin{equation}
    I(\overline{u},\overline{ g }) = \min I(u, g )
\end{equation}
where the minimization is over pairs \((u, g )\in H^1_0(\Om)\times Z_{\mathrm{ad}}\) that satisfy
\begin{equation}\label{loc_state_eq}
    B_0(u,v) =\langle  g ,v\rangle,\quad \forall v\in H^1_0(\Om).
\end{equation}
Here \(B_0(u,v):=\int_{\Om} A(\xb)\nabla u(\xb)\cdot \nabla v(\xb)d\xb\).
\end{theorem}
To state the corresponding result for discretized problems, let us introduce some notations. Let \(\{\mathcal{T}_h\}_{h>0}\) be a quasi-uniform mesh of size \(h\) on \(\Om\). Let \(X_{h}\)  be the space of continuous piecewise linear functions subject to the mesh with zero nonlocal boundary data:  
\[X_h:= \{ w _h\in C^{0}(\overline{\Omega})|\  w _h|_T\in \mathcal{P}_1(T),\ \forall T\in \mathcal{T}_h,\  w _h=0\ \text{on}\ \R^d\backslash\Om\}.\]
We equip $X_h$ with \(H^1(\Om)\)-norm. For the nonlocal discrete problem, we use the space of piecewise linear functions that are in $\cS^{\nub}_{w_\del}(\Om)$: 
\[ X_{\del,h} =  \{ w _h\in \cS^{\nub}_{w_\del}(\Om) |\  w _h|_T\in \mathcal{P}_1(T),\ \forall T\in \mathcal{T}_h\} 
\] equipped with \(\cS^{\nub}_{w_\del}(\Om)\)-norm. Similarly, let \(Z_h\) denote the piecewise constant functions with respect to the mesh \(\{\mathcal{T}_h\}_{h>0}\), i.e.,
\[Z_h:=\{z_h\in L^\infty(\Om)|\ z_h|_T\in \mathcal{P}_0(T),\ \forall T\in \mathcal{T}_h\}.\]
where \(\mathcal{P}_m(T)\), as before, is the space of polynomials of degree \(m\).  Notice that since $\alpha(\xb) < \beta(\xb) $ for all $\xb\in \Omega$, there exists a constant $h_0>0$ such that $Z_h\cap Z_{\mathrm{ad}} \neq \emptyset$ for any $h\in (0,h_0)$. Hereafter, we will always assume $h\in (0,h_0)$ implicitly.

\begin{theorem}[Discrete problems and convergence of solutions] \label{discrete-summary-optimal}
Assume all the conditions used in \Cref{conts-summary-optimal}. Then for any $\delta>0$, there exists a unique optimal pair  
\((\overline{u_{\del,h}},\overline{ g _{\del,h}})\in X_{\del,h}\times (Z_{h}\cap Z_{\mathrm{ad}})\) such that 
    \begin{equation}\label{prob:non_disc}
    I(\overline{u_{\del,h}},\overline{ g _{\del,h}})=\min I(u_{\del,h}, g _{\del,h})\end{equation}
    where the minimization is over pairs \((u_{\del,h}, g _{h})\in X_{\del,h}\times (Z_{h}\cap Z_{\mathrm{ad}})\) that satisfy 
    \begin{equation}\label{disc_nonloc_state_eq}
        B_\del(u_{\del,h},v_{\del,h}) = \langle  g _{h},v_{\del,h}\rangle,\quad\forall v_{\del,h}\in X_{\del,h}.
    \end{equation}
 Moreover, there is a unique pair $(\overline{u_h}, \overline{g_h})\in X_h \times (Z_h\cap Z_{\mathrm{ad}})$ such that  
 \(\overline{u_{\del,h}}\to \overline{u_h}\) in \(L^2(\Om)\) , \(\overline{ g _{\del,h}}\rightharpoonup \overline{ g _h}\) in \(L^2(\Om)\), as $\delta \to 0$, and 
  in addition, \((\overline{u_h},\overline{ g _h})\) solves the local discrete optimal control 
 \begin{equation}\label{prob:loc_con}
 I(\overline{u_h},\overline{ g _h})=\min I(u_h, g _h)\end{equation}
    where the minimization is over pairs \((u_h, g _h)\in X_h\times (Z_{h}\cap Z_{\mathrm{ad}})\) that satisfy 
    \begin{equation}\label{disc_loc_state_eq}
        B_0(u_h,v_h)=\langle  g _h,v_h\rangle,\quad\forall v_h\in X_h.
    \end{equation}
\end{theorem}

Before we present the proofs of \Cref{conts-summary-optimal} and \Cref{discrete-summary-optimal}, we note that applying the nonlocal and local Poincar\'e inequality and Lax-Milgram theorem, we know the state equations \eqref{nonloc_state_eq}, \eqref{loc_state_eq}, \eqref{disc_nonloc_state_eq} and \eqref{disc_loc_state_eq} are uniquely  solvable in their corresponding energy spaces. Like the continuous case, for equation \eqref{disc_nonloc_state_eq} we also introduce the discrete solution operator $\fS_{\del,h}:L^2(\Om)\to X_{\del,h}\subset \cS^{\nub}_{w_\del}(\Om)\subset L^2(\Om)$. One may check the same uniform estimates \eqref{eq:sol_op_energy_norm} and \eqref{eq:sol_op_L2_norm} hold for $\fS_{\del,h}$ for all $\del\in (0,\del_0)$ and $h\in (0,h_0)$. Since the techniques to prove the well-posedness of the four minimization problems as well as  the convergence results are essentially the same for the two theorems, we present the proof of \Cref{conts-summary-optimal} and omit the proof of \Cref{discrete-summary-optimal}. 
\begin{proof}[Proof of \Cref{conts-summary-optimal}] We will prove the theorem in several steps.

{\bf Well-posedness:} We show existence of a minimizer to the objective functional subject to the nonlocal constraint using the direct method of calculus of variations.  
Using the solution operator $\mathfrak{S}_\del$, we will work on the reduced cost functional  $j(g)=I(\mathfrak{S}_{\del}(g),g)$ given by 
\begin{equation}\label{eq:def_reduced_objective_functional_j}
    j( g ):=\mathfrak{F}(\mathfrak{S}_{\del}(g) )+\frac{\lambda}{2}\int_{\Om}\Ga(\xb)| g (\xb)|^2 d\xb,\end{equation}
where \(\mathfrak{F}:L^2(\Om)\to \R\) is defined by 
    \[\mathfrak{F}(v):=\int_{\Om} F(\xb,v(\xb))d\xb.\]
    We apply the direct method of calculus of variations to the problem of finding a minimizer to $\inf_{g\in Z_{\mathrm{ad}}} j(g)$.
    
     We first notice that \(Z_{\mathrm{ad}}\) is a closed, convex and bounded subset of Hilbert space \(L^2(\Om)\), by \cite[Theorem~2.11]{troltzsch2010optimal} \(Z_{\mathrm{ad}}\) is weakly sequentially compact. In addition, $j$ is bounded from below on $Z_{\mathrm{ad}}$. Indeed, it suffices to show that 
    $\mathfrak{F}\circ \mathfrak{S}_\del: Z_{\mathrm{ad}}\to\R$ is bounded from below, since the second term in $j$ is nonnegative. To that end, using item (3) of the the assumption on $F$, \eqref{eq:bound_Z_ad} and \eqref{eq:sol_op_L2_norm}, we have that for any $g\in Z_{\mathrm{ad}},$
    \[
    \begin{split}
    \mathfrak{F}(\mathfrak{S}_{\delta}(g)) &\geq -c_1\|\mathfrak{S}_\delta(g)\|^{2}_{L^{2}(\Omega)} -\|\ell\|_{L^1(\Omega)} \\
    &\geq -c_1\frac{C(\del_0)^4}{\mu^2}\|g\|_{L^{2}(\Omega)}^{2}-\|\ell\|_{L^1(\Omega)} \geq -c_1B^2\frac{C(\del_0)^4}{\mu^2}-\|\ell\|_{L^1(\Omega)}.
    \end{split}
    \]
    
    We henceforth denote $j_0=\inf_{g\in Z_{\mathrm{ad}}} j(g)$. It remains to show that there exists $\tilde{g}\in Z_{\mathrm{ad}}$ such that $j_0=j(\tilde{g})$. To this end, first we find a sequence $\{g_n\}_{n=1}^\infty\subset Z_{\mathrm{ad}}$ such that $\lim_{n\to\infty} j(g_n)=j_0$. Again by \cite[Theorem~2.11]{troltzsch2010optimal} there exists a subsequence $\{g_{n_k}\}_{k=1}^\infty\subset Z_{\mathrm{ad}}$ converging to some $\tilde{g}\in Z_{\mathrm{ad}}$ weakly in $L^2(\Om)$. Since $|\sqrt{\Gamma(\xb)}g_{n_k}(\xb)|\le |\sqrt{\Ga(\xb)}|\min\{|a(\xb)|,|b(\xb)|\}$, $\{\sqrt{\Gamma(\xb)}g_{n_k}(\xb)\}_{k=1}^\infty$ is uniformly bounded in $L^2(\Om)$. Then one can further obtain a subsequence, still denoted by $\{g_{n_k}\}_{k=1}^\infty$, converging weakly in $L^2(\Om)$. By a density argument, it is not hard to show this weak limit is $\sqrt{\Ga(\xb)}\tilde{g}(\xb)$.  Also, since \(F(\xb,\cdot)\) is convex in \(\R\) a.e. \(\xb\in\Om\) and $F$ is bounded from below by an affine map according to items (2) and (4) of the assumption on $F$,  
     \cite[Theorem~6.54]{fonseca2006modern} 
    implies that \(\mathfrak{F}\) is sequentially weakly lower semicontinuous. Moreover, since $\mathfrak{S}_\del$ is continuous on $L^{2}(\Omega)$, it is weakly continuous on $L^{2}(\Omega)$ and therefore, $\mathfrak{F}\circ \mathfrak{S}_\del:L^2(\Om)\to \R$ is also sequentially weakly lower semicontinuous on $L^{2}(\Omega)$. 
    Combining the above, we have that 
    \[\begin{split}
        j(\tilde{g})&=\mathfrak{F}(\mathfrak{S}_\del \tilde{g} )+\frac{\lambda}{2}\int_{\Om}\Ga(\xb)| \tilde{g} (\xb)|^2 d\xb\\
        &\le \liminf_{k\to\infty} \mathfrak{F}(\mathfrak{S}_\del g_{n_k}) + \liminf_{k\to\infty} \frac{\lambda}{2}\int_{\Om}\Ga(\xb)| g_{n_k} (\xb)|^2 d\xb\\
        &\le \liminf_{k\to\infty} \left(\mathfrak{F}(\mathfrak{S}_\del g_{n_k}) + \frac{\lambda}{2}\int_{\Om}\Ga(\xb)| g_{n_k} (\xb)|^2 d\xb\right)=\lim_{k\to \infty} j(g_{n_k}) =j_0.
    \end{split}\]
    Then the existence of minimizer is proved. The uniqueness of the optimal pair follows from the fact that $j$ is strictly convex because $\mathfrak{F}$ is convex and $\lambda > 0$.

{\bf Compactness:} Having shown that, for every horizon \(\del> 0\), the optimal control problem has a unique solution \((\overline{u_\del},\overline{ g _\del})\), we now study the behavior of the pair as \(\del\to 0\). 
Since $\overline{u_\del}=\fS_\del \overline{g_\del}$, by \eqref{eq:sol_op_energy_norm} and \eqref{eq:sol_op_L2_norm} one obtains 
\[
|\overline{u_\del}|_{\cS^{\nub}_{w_\del}(\Om)}\le \frac{C(\del_0)}{\mu} \| \overline{ g _\del}\|_{L^2(\Om)}\quad\text{and}\quad \|\overline{u_\del}\|_{L^2(\Om)}\le \frac{C(\del_0)^2}{\mu} \| \overline{ g _\del}\|_{L^2(\Om)}.
\]
    Therefore, using \eqref{eq:bound_Z_ad} we have
    \[\sup_{\del\in (0,\del_0)}\|\overline{u_\del}\|_{\cS^{\nub}_{w_\del}(\Om)}\le \sqrt{1+C(\del_0)^2}\frac{C(\del_0)}{\mu}B.\]
    By \Cref{thm:cpt_full}
    one concludes that, up to a subsequence, \(\overline{u_\del}\) converges strongly in \(L^2(\Om)\) to some \(u\in H^1_0(\Om)\).  
    Also, since \(Z_{\mathrm{ad}}\) is weakly sequentially compact, then there exists \( g \in Z_{\mathrm{ad}}\) such that \( \overline{ g _\del}\rightharpoonup  g \) in \(L^2(\Om)\) up to a subsequence. Without loss of generality, we assume \(\overline{u_\del}\to u\) in \(L^2(\Om)\) and \(\overline{ g _\del}\rightharpoonup  g\) in \(L^2(\Om)\).

{\bf Convergence:} Next we show that $(u, g)$ solves 
\[B_0(u,v)=\langle  g ,v\rangle,\quad\forall v\in H^1_0(\Om).\]
By density, it suffices to show this for all $v\in C^\infty_c(\Om)$. Fix $v\in C^\infty_c(\Om)$. On the one hand, due to the weak convergence \( \overline{ g _\del}\rightharpoonup  g \) in \(L^2(\Om)\) it holds that $\langle \overline{g_\delta}, v\rangle \to \langle g, v\rangle$ as $\del\to 0$. On the other hand, arguing as in the proof of (Cii) for \Cref{thm:convergence_delta} one knows
    \begin{equation}\label{eq:lim_bilinear_form-51}
        (A\mathfrak{G}^{\nub}_{w_\del}\overline{u_\delta},\mathcal{G}^{\nub}_{w_\del}v)_{L^2(\R^d;\R^d)}\to (A\nabla u,\nabla v)_{L^2(\R^d;\R^d)},\quad \del\to 0.
    \end{equation}
    Therefore, letting $\del\to 0$ in
    \[B_\del(\overline{u_\del},v)=\langle \overline{g_\del},v\rangle\]
    finishes the proof of convergence.

{\bf Limit point is the optimal solution: } Finally, we show that the limit point pair $(u, g)$ is an optimal pair in the sense that 
\[
I(u, g) \leq I(v, f)
\]
where $(v, f)\in H^{1}_{0}(\Omega)\times Z_{\mathrm{ad}}$ solves $ B_0(v,\phi) =\langle  f ,\phi\rangle,\  \forall \phi\in H^1_0(\Om)$. Given such a pair $(v, f)\in H^{1}_{0}(\Omega)\times Z_{\mathrm{ad}}$, we denote $v_\del:=\fS_\del f\in \mathcal{S}^{\nub}_{w_{\delta}}(\Om)$ for $\del\in (0,\del_0)$.
Since $\|v_\del\|_{\mathcal{S}^{\nub}_{w_{\delta}}(\Om)} \leq C\|f\|_{L^{2}(\Omega)}$, by compactness, \Cref{thm:cpt_full}, as $\del\to 0$, up to a subsequence, $v_{\delta}$ strongly converges in $L^{2}(\Omega)$ to some function belonging to $H^{1}_{0}(\Omega)$, and applying a similar argument as in the previous step, we can in fact show that   
\(v_\del\to v\) in \(L^2(\Om)\). By \cite[Theorem 4.9]{brezis2011functional} we may take a further subsequence,  such that for some $h(\xb) \in L^2$, $|v_\delta(\xb)| \leq h(\xb)$ for all $\delta$, and $v_\delta\to v$ almost everywhere.    We then have by item (3) of the assumption on \(F\) that 
\[
|F(\xb,v_\del(\xb))|\leq C |h(\xb)|^2 + \ell(\xb)\textrm{ and  }   F(\xb,v_\del(\xb)) \to F(\xb,v(\xb)) \quad \textrm{ a.e. } \xb\in \Omega.
\]
By Lebegue dominated convergence thereom, we have \[\lim_{\del\to 0} \int_{\Om}F(\xb,v_\del(\xb))d\xb=\int_{\Om} F(\xb,v(\xb))d\xb.\]
    Thus we have  \[\lim_{\del\to 0} I(v_\del, f )=I(v, f ).\] 
    Recall that \(\overline{ g _\del}\rightharpoonup  g\) in \(L^2(\Om)\). Using a density argument as before, one can show that $\sqrt{\Ga}\overline{g_\del}\rightharpoonup \sqrt{\Ga}g$ in $L^2(\Om)$ as $\del\to 0$. Then it follows that 
    \[\int_{\Om} \Ga(\xb)|g(\xb)|^2d\xb\le \liminf_{\del\to 0}\int_{\Om} \Ga(\xb)|\overline{ g _\del}(\xb)|^2d\xb\] and
    \[\begin{split}
        I(u,g)&=\int_{\Om} F(\xb,u(\xb))d\xb +\frac{\lambda}{2}\int_{\Om} \Ga(\xb)|g(\xb)|^2d\xb\\
        &\le \liminf_{\del\to 0} \left(\int_{\Om} F(\xb,\overline{u_\del}(\xb))d\xb +\frac{\lambda}{2}\int_{\Om} \Ga(\xb)|\overline{ g _\del}(\xb)|^2d\xb\right)\\
        &=\liminf_{\del\to 0} I(\overline{u_\del},\overline{ g _\del}),
    \end{split}\]
    as \(\overline{u_\del}\to u\) in \(L^2(\Om)\) and \(\overline{ g _\del}\rightharpoonup  g\) in \(L^2(\Om)\). Since \((\overline{u_\del},\overline{ g _\del})\) is a minimizer of $I$ subject to the nonlocal state equation, we have that \(I(\overline{u_\del},\overline{ g _\del})\le I(v_\del, f )\). Combining all the above inequalities yields 
    \[I(u,g)\le \liminf_{\del\to 0} I(\overline{u_\del},\overline{ g _\del})\le \lim_{\del\to 0}I(v_\del, f )=I(v, f ).\]
    Having shown now that $(u, g)$ is an optimal solution solution we will use the bar notation and write $(\overline{u}, \overline{g})$.
    This finishes the proof.
   \end{proof}

\subsection{First Order Optimality and Asymptotic Compatibility}

As we have seen in the proof of \Cref{conts-summary-optimal}, one can rewrite optimal control problem  \eqref{prob:non_con}-\eqref{nonloc_state_eq} as a variational problem involving the reduced objective functional \(j( g )\) defined in \eqref{eq:def_reduced_objective_functional_j}: find  \(\overline{ g _\del}\in Z_{\mathrm{ad}}\) such that
\begin{equation}\label{eq:non_con_reduced}
    j(\overline{ g _\del})=\min_{ g \in Z_{\mathrm{ad}}} j( g ).
\end{equation}
It is clear that 
if \(\overline{ g _\del}\) is the unique solution to \eqref{eq:non_con_reduced}, then \((\overline{u_\del},\overline{ g _\del})\) is the unique solution to \eqref{prob:non_con}-\eqref{nonloc_state_eq} where \(\overline{u_\del}:=\mathfrak{S}_\del\overline{ g _\del}\), and conversely, if \((\overline{u_\del},\overline{ g _\del})\) is the unique solution to \eqref{prob:non_con}-\eqref{nonloc_state_eq} then \(\overline{ g _\del}\) is the unique solution to \eqref{eq:non_con_reduced}.

By \cite[Lemma~2.21]{troltzsch2010optimal}, the first order optimality condition of the minimization problem \eqref{eq:non_con_reduced} is 
\[\langle j'(\overline{ g _\del}), q -\overline{ g _\del}\rangle\ge 0,\quad\forall  q \in Z_{\mathrm{ad}}.\] 
By calculating \(j'(\overline{ g _\del})\) the optimality condition can be written as 
\begin{equation}
    \langle \mathfrak{S}_\del^*  F_\xi (\cdot,\mathfrak{S}_\del\overline{ g _\del}(\cdot))+\lambda\Ga \overline{ g _\del}, q -\overline{ g _\del}\rangle\ge 0,\quad\forall  q \in Z_{\mathrm{ad}}.
\end{equation}
where \(\mathfrak{S}_\del^*:L^2(\Om)\to L^2(\Om)\) is the adjoint of \(\mathfrak{S}_\del\) in the \(L^2\)-sense. It is immediate that  \(\mathfrak{S}_\del\) is self-adjoint. Also, the map \( F_\xi (\cdot,\mathfrak{S}_\del\overline{ g _\del}(\cdot)):=(\xb\mapsto F_\xi (\xb,\mathfrak{S}_\del\overline{ g _\del}(\xb)))\in L^2(\Om)\) so that \(\mathfrak{S}_\del^*\) can act on it. Indeed, by assumption \(| F_\xi (\xb,\xi)|\le C|\xi|  + e(\xb)\), one has \[| F_\xi (\xb,\mathfrak{S}_\del\overline{ g _\del}(\xb))|\le C|\mathfrak{S}_\del\overline{ g _\del}(\xb)| + e(\xb)\in L^2(\Om).\]
By introducing a new notation \(\overline{ p _\del}\), the above equation can be rewritten as the system
\begin{equation}\label{eq:opt_cond}
    \left\{\begin{aligned}
        \langle \overline{ p _\del} + \lambda \Ga\overline{ g _\del}, q  - \overline{ g _\del} \rangle &\ge 0,\quad\forall  q \in Z_{\mathrm{ad}},\\
        \overline{ p _\del} &= \mathfrak{S}_\del^*  F_\xi (\cdot,\overline{u_\del}(\cdot)),\\
        \overline{u_\del} &= \mathfrak{S}_\del\overline{ g _\del}.
    \end{aligned}\right.
\end{equation}
Note that if \(\Ga=1\), then the first inequality amounts to the following identity
\begin{equation}
    \overline{ g _\del}=-\frac{1}{\lambda}\Pi_{Z_{\mathrm{ad}}}\overline{ p _\del},
\end{equation}
where \(\Pi_E:L^2(\Om)\to E\) denotes the \(L^2\)-projection onto the bounded, convex and closed set \(E\subset L^2(\Om)\), that is, for any \( f \in L^2(\Om)\), \(\Pi_E f \) is the unique solution to the minimization problem 
\[\min_{ g \in E}\| g - f \|_{L^2(\Om)}.\]
Since \(\mathfrak{S}_\del\) is self-adjoint, one obtains that \(\overline{ p _\del}=\mathfrak{S}_\del  F_\xi (\cdot,\overline{u_\del}(\cdot))\).

Note that the objective functional is strictly convex, the first order necessary condition is also sufficient. Therefore, we have the following proposition.
\begin{proposition}[Optimality Conditions]
    For every \(\del> 0\), the pair \((\overline{u_\del},\overline{ g _\del})\in \cS^{\nub}_{w_\del}(\Om)\times Z_{\mathrm{ad}}\) is a solution to \eqref{prob:non_con}-\eqref{nonloc_state_eq} if and only if \eqref{eq:opt_cond} holds.
\end{proposition}

The optimality conditions for the nonlocal discrete problem read:
\begin{equation}\label{eq:disc_non_opt_cond}
    \left\{
        \begin{aligned}
            \langle \overline{ p _{\del,h}} + \lambda\Ga \overline{ g _{\del,h}}, q _h-\overline{ g _{\del,h}}\rangle &\ge 0,\quad\forall q _h\in Z_{\mathrm{ad}}\cap Z_h,\\
            \overline{ p _{\del,h}} &= \mathfrak{S}_{\del,h}  F_\xi (\cdot,\overline{u_{\del,h}}(\cdot)),\\
            \overline{u_{\del,h}} &= \mathfrak{S}_{\del,h}\overline{ g _{\del,h}},
        \end{aligned}
    \right.
\end{equation}
where \(\mathfrak{S}_{\del,h}:L^2(\Om)\to L^2(\Om)\) is the discrete solution operator, and \(\mathfrak{S}_{\del,h}^*:L^2(\Om)\to L^2(\Om)\) is its \(L^2\)-adjoint operator. Note that in the second equation we used the fact that \(\mathfrak{S}_{\del,h}\) is self-adjoint in the $L^2$ sense. 

Similarly, the continuous local problem is well-posed if and only if the corresponding optimality conditions hold:
\begin{equation}\label{eq:con_loc_opt_cond}
    \left\{\begin{aligned}
        \langle \overline{ p } + \lambda \Ga\overline{ g }, q  - \overline{ g } \rangle &\ge 0,\quad\forall  q \in Z_{\mathrm{ad}},\\
        \overline{ p } &= \mathfrak{S}_0  F_{\xi} (\cdot,\overline{u}(\cdot)),\\
        \overline{u} &= \mathfrak{S}_0 \overline{ g },
    \end{aligned}\right.
\end{equation}
where \(\mathfrak{S}_0:L^2(\Om)\to L^2(\Om)\) is the solution operator defined as \(\mathfrak{S}_0 g :=u\) where \(u\in H^1_0(\Om)\subset L^2(\Om)\) is the unique solution to \eqref{loc_state_eq}. Here again we used \(\mathfrak{S}_0\) is self-adjoint in the second equation.

Finally we state and prove the results for asymptotic compatibility. 
We first recall \cite[Definition~7.1]{mengesha2023optimal} which is the definition of asymptotic compatibility of a scheme to the optimal control problem.
\begin{definition}[Asymptotic compatibility]\label{def:AC}
    We say that the family of solutions \(\{(\overline{u_{\del,h}},\overline{ g _{\del,h}})\}_{h>0,\del>0}\) to \eqref{prob:non_disc} is \textbf{asymptotically compatible} in \(\del,h>0\) if for any sequences \(\{\del_k\}_{k=1}^\infty\), \(\{h_k\}_{k=1}^\infty\) with \(\del_k,h_k\to 0\), we have that \(\overline{u_{\del_k,h_k}}\to \overline{u}\) strongly in \(L^2(\Om)\) and \(\overline{ g _{\del_k,h_k}}\rightharpoonup \overline{ g }\) weakly in \(L^2(\Om)\). Here \((\overline{u},\overline{ g })\in H^1_0(\Om)\times Z_{\mathrm{ad}}\) denotes the optimal solution for \eqref{prob:loc_con}.
\end{definition}

Recall that in the proof of \Cref{thm:convergence_delta} we have checked the conditions (A)-(D) hold to prove asymptotic compatibility for parameterized problems. 
It turns out that for the optimal control \eqref{prob:non_disc} we can establish similar asymptotic compatibility as well in the sense of \Cref{def:AC}. Such type of result is proved in \cite[Theorem~7.3]{mengesha2023optimal}. We demonstrate below that the parameterized discrete optimal control problem, under the constraint of the nonlocal state equation we consider here, has similar compatibility behavior.   The following is the main result of the subsection. Its proof is similar to that of \cite[Theorem~7.3]{mengesha2023optimal} with appropriate modification to fit our setting. 
\begin{theorem}[Asymptotic compatibility]\label{thm:AC}
    The solution to \eqref{prob:non_disc} is asymptotically compatible in \(\del,h>0\), in the sense of \Cref{def:AC}.
\end{theorem}
\begin{proof}
    We denote \(\{(\overline{u_k},\overline{ g _k},\overline{ p _k})\}_{k=1}^\infty:=\{(\overline{u_{\del_k,h_k}},\overline{ g _{\del_k,h_k}},\overline{ p _{\del_k,h_k}})\}_{k=1}^\infty\), which is the sequence of triples solving \eqref{eq:disc_non_opt_cond}. 
    We consider an arbitrary, non-relabeled subsequence of the triples \(\{(\overline{u_k},\overline{ g _k},\overline{ p _k})\}_{k=1}^\infty\), and show that it has a further subsequence which always converges to the same limit point. Moreover, this limit solves \eqref{eq:con_loc_opt_cond} and, since this uniquely characterizes the solution to \eqref{prob:loc_con}, asymptotic compatibility will follow.

    Since \(\{\overline{ g_k}\}_{k=1}^\infty\subset Z_{\mathrm{ad}}\), again by \cite[Theorem~2.11]{troltzsch2010optimal} there exists a subsequence (non-relabeled) and a function \( g _*\in Z_{\mathrm{ad}}\) such that \(\overline{ g _{k}}\rightharpoonup g _*\) in \(L^2(\Om)\). Meanwhile, using \eqref{eq:bound_Z_ad}, \eqref{eq:sol_op_energy_norm} and \eqref{eq:sol_op_L2_norm} with $\fS_{\del_k,h_k}$ in place of $\fS_{\del}$ 
    we know there exists a constant \(C_1:=\sqrt{1+C(\del_0)^2}C(\del_0)B/\mu\) such that \(\|\overline{u_k}\|_{\cS^{\nub}_{w_{\del_k}}(\Om)}\le C_1\) for all $k$ large enough. By \Cref{thm:cpt_full}, upon taking a further non-relabeled subsequence, there exists a limit point \(u_*\in H^1_0(\Om)\) such that \(\overline{u_k}\to u_*\) in \(L^2(\Om)\). Since \(\{(\overline{u_k},\overline{ g _k})\}_{k=1}^\infty\) are the pair of functions solving \eqref{prob:non_disc}, we have that
    \[B_{\del_k}(\overline{u_k},v_k)=\langle \overline{ g _k},v_k\rangle,\quad\forall v_k\in X_{\del_k,h_k}.\]
    Arguing as in the last paragraph of the proof of \Cref{thm:main_convergence}, with a slight modification that $(f,v_k)_{L^2}$ and $(f,v)_{L^2}$ are replaced by $\langle \overline{ g _k},v_k\rangle$ and $\langle  g _*,v\rangle$ respectively, one can show that
    \begin{equation}\label{eq:loc_opt_3}
        B_0(u_*,v)=\langle  g _*,v\rangle,\quad\forall v\in H^1_0(\Om).
    \end{equation}
    Recall that \(\overline{ p _k}\) satisfies the second equation in \eqref{eq:disc_non_opt_cond}, i.e., 
    \[B_{\del_k}(\overline{ p _k},v_k)=\langle  F_\xi (\cdot,\overline{u_k}(\cdot)),v_k\rangle,\quad\forall v_k\in X_{\del_k,h_k}.\]
    Since \[\| F_\xi (\cdot,\overline{u_k}(\cdot))\|_{L^2(\Om)}^2\le 2c_2^2\|\overline{u_k}\|_{L^2(\Om)}^2 + 2\|e\|_{L^{2}(\Omega)}^{2}\le 2c_2^2C_1^2 + 2\|e\|_{L^{2}(\Omega)}^{2}=:C_2,\]
    repeating the argument above for the uniform bound on $\|\overline{u_k}\|_{\cS^{\nub}_{w_{\del_k}}(\Om)}$,
    we have \(\|\overline{ p _k}\|_{\cS^{\nub}_{w_{\del_k}}(\Om)}\le C_1\| F_\xi (\cdot,\overline{u_k}(\cdot))\|_{L^2(\Om)}/B\le C_1\sqrt{C_2}/B\) for all $k$ large enough. 
    Moreover, up to a subsequence, using \cite[Theorem~4.9]{brezis2011functional}, item (5) of the assumption on $F$, and the Lebesgue dominated convergence theorem, we have that as \(k\to\infty\),
    \[ F_\xi (\cdot,\overline{u_k}(\cdot))\to  F_\xi (\cdot,u_*(\cdot))\ \text{in}\ L^2(\Om).\]
    Repeating the analysis for \(\{\overline{u_k}\}_{k=1}^\infty\), we identify \( p _*\in H^1_0(\Om)\) such that up to a subsequence \(\overline{ p _k}\to  p _*\) in \(L^2(\Om)\) and 
    \begin{equation}\label{eq:loc_opt_2}
        B_0( p _*,v)=\langle  F_\xi (\cdot,u_*(\cdot)),v\rangle,\quad\forall v\in H^1_0(\Om).
    \end{equation}
    Finally we show that 
    \begin{equation}\label{eq:loc_opt_1}
        \langle  p _* + \lambda \Ga g _*, q  -  g _* \rangle \ge 0,\quad\forall  q \in Z_{\mathrm{ad}}.
    \end{equation}
    Since \(\alpha,\beta\in C(\overline{\Om})\), it is not hard to show the following approximation result: for any \( q \in Z_{\mathrm{ad}}\), there exists a sequence \(\{ q _{k}\in Z_{h_k}\cap Z_{\mathrm{ad}}\}_{h_k>0}\) such that \( q _k\to  q \) in \(L^2(\Om)\) as \(k\to\infty\); see \Cref{lem:one_sided_approx} for its proof. 
    From the first equation of \eqref{eq:disc_non_opt_cond} we know
    \[\langle \overline{ p _{k}} + \lambda\Ga \overline{ g _{k}}, q _k-\overline{ g _{k}}\rangle \ge 0.\]
    That is, 
    \[\langle \overline{ p _k}, q _k-\overline{ g _k}\rangle +\langle \lambda \Ga\overline{ g _k}, q _k\rangle \ge \lambda\int_{\Om}\Ga(\xb)|\overline{ g _k}(\xb)|^2d\xb.\]
    By the weak lower semicontinuity, one has 
    \[\liminf_{k\to\infty} \int_{\Om}\Ga(\xb)|\overline{ g _k}(\xb)|^2d\xb\ge \int_{\Om}\Ga(\xb)| g _*(\xb)|^2d\xb.\]
    On the other hand, letting \(k\to\infty\) one has
    \[\langle \overline{ p _k}, q _k-\overline{ g _k}\rangle +\langle \lambda \Ga\overline{ g _k}, q _k\rangle\to \langle  p _*, q - g _*\rangle +\langle \lambda \Ga g _*, q \rangle.\]
    Therefore, 
    \[\langle  p _*, q - g _*\rangle +\langle \lambda \Ga g _*, q \rangle\ge \lambda\int_{\Om}\Ga(\xb)| g _*(\xb)|^2d\xb,\]
    and \eqref{eq:loc_opt_1} holds. From \eqref{eq:loc_opt_1}, \eqref{eq:loc_opt_2} and \eqref{eq:loc_opt_3} we know that \((u_*, g _*, p _*)\) solves the equations \eqref{eq:con_loc_opt_cond}. Since this system has a unique triple of solution \((\overline{u},\overline{ g },\overline{ p })\), it follows that \((u_*, g _*, p _*)=(\overline{u},\overline{ g },\overline{ p })\). Therefore, asymptotic compatibility holds.
\end{proof}

\section{Conclusion}\label{sec:conclusion}
In this work we establish two Bourgain-Brezis-Mironescu 
\cite{bourgain2001another} type compactness theorems regarding the nonlocal Sobolev spaces associated with half-space gradient operators \cite{han2023nonlocal}. An equivalent Fourier characterization of the nonlocal function spaces is shown, which highlights the key role of the Fourier symbol \(\lab^{\nub}_w(\xib)\) in the studies of nonlocal function spaces \(\cS^{\nub}_w(\Om)\). Using an improved lower bound estimate on \(|\lab^{\nub}_w(\xib)|\), we are able to show the locally compact embedding of \(\cS^{\nub}_w(\R^d)\) into \(L^2(\R^d)\) for a nonintegrable kernel \(w\). For the sequence of kernels $\{w_n \}$ such that $0\leq w_n \nearrow w$ where $w$ is nonintegrable, we prove the sequential compactness result in \Cref{thm:cpt_seq_nonint_kernel}. Another compactness result holds for a different sequence of kernels satisfying \Cref{ass:kernel4cptthm}, while the result is shown in \Cref{thm:cpt_full}. Based on these compactness theorems we show the uniform Poincar\'e inequalities generalizing the results in \cite{han2023nonlocal}, and study the convergence and asymptotic compatibility of nonlocal diffusion problems and their approximations, respectively. Finally we apply the theoretical results to optimal control problems with nonlocal equations as constraints.

A key tool to show \(L^p\)-compactness is the Riesz-Kolmogorov-Fr\'echet compactness criterion. Here, however, we can only show compactness for \(p=2\) as in such case one can use Plancherel's theorem to convert the \(L^2\) norm of a function into that of its Fourier transform. It is intriguing to explore whether the compactness result holds for general \(1<p<\infty\). For \(p=2\), one crucial ingredient we frequently use is the lower bound estimate of \(|\lab^{\nub}_w(\xib)|\). For large \(|\xib|\) the lower bound has been significantly improved compared to \cite[Lemma~5.1]{han2023nonlocal}. 
For small \(|\xib|\), we apply the technique of truncation at infinity to make $w$ compactly supported so that \Cref{lem:labbdd_original}(2)(a) is applicable to derive
\(|\lab^{\nub}_w(\xib)|\ge C|\xib|\) for \(|\xib|\) small. 

Note that \Cref{ass:kernel4cptthm} provides sufficient conditions to derive the crucial lower bound estimates for the Fourier symbols associated with a sequence of kernels in \Cref{lem:lowerbd_im_re}. However, it is worth exploring whether these assumptions can be relaxed.
An open question remains: can the compactness result in \Cref{thm:cpt_full} be established under the weakest assumptions, specifically \cref{eq:kernelassumption,eq:kernelassumption_2,eq:kernelassumption_3}? To illustrate, consider a kernel defined by
\[
w_\del(\zb)=\frac{2d}{\om_{d-1}}\frac{1}{|\log(\del)|}\frac{1}{|\zb|^{d+1}}\chi_{\{\del<|\zb|<1\}}(\zb). 
\]
This kernel appeared in \cite[Corollary~2.3]{ponce2004estimate} has the feature that it is truncated near origin. Our numerical results show that both the real part and imaginary part of the Fourier symbol are oscillating but eventually the 
real part stabilizes at a positive value while the imaginary part converges to zero. However, we could not prove the compactness results because we were unable to determine the critical value $|\xib|$ such that we use the lower bound of imaginary and real part separately across this critical value. We nevertheless conjecture that  \Cref{thm:cpt_full}  will remain valid under the assumptions \cref{eq:kernelassumption,eq:kernelassumption_2,eq:kernelassumption_3} and hope to explore this in future work.

\appendix
\section{}

\subsection{Proofs of \Cref{lem:lablowerbdd_xilarge} and \Cref{prop:locallim_ptws_lp}}
\label{subsec:Appendix1}
\begin{proof}[Proof of \Cref{lem:lablowerbdd_xilarge}]
    Recall from the proof of \cite[Lemma~5.1]{han2023nonlocal} that 
    \[|\Re(\lab^{\nub}_{w})(\xib)|\ge \int_{z_1>0}\frac{z_1}{|\zb|}w(\zb)(1-\cos(2\pi R^T_{\nub}\xib\cdot\zb))d\zb,\]
    where \(R_{\nub}\) is an orthogonal matrix such that \(\nub=R_{\nub}\eb_1\). Denote \(\mub:=R^T_{\nub}\xib/|R^T_{\nub}\xib|\). Then
    \[\begin{split}
        |\Re(\lab^{\nub}_{w})(\xib)|&\ge \int_{z_1>0}\frac{z_1}{|\zb|}w(\zb)(1-\cos(2\pi |\xib|\mub\cdot\zb))d\zb\\
        &=:I(\xib).
    \end{split}\]
    Our goal is to find constant \(c_{\lambda}=c_{\lambda}(N\ep,d)>0\) such that \begin{equation}\label{estI}
        I(\xib)\ge c_{\lambda}\int_{|\zb|>\frac{N\ep}{|\xib|}}w(\zb)d\zb,\quad\forall |\xib|>N,\,\xib\in\R^d.
    \end{equation}

    We distinguish two cases \(d=1\) and \(d\ge 2\). 

    \textbf{Case I: \(d=1\)}. Then \(\mu=e_1\) or \(\mu=-e_1\). In either case one has 
    \[I(\xi)= \int_{0}^{\infty} w(z)\left(1-\cos(2\pi|\xi|z)\right)dz.\]
    Since the right-hand side is even in \(\xi\), it suffices to show that for \(\xi>N\), 
    \begin{equation}\label{estI_1d}
        I(\xi)=\int_{0}^{\infty} w(z)\left(1-\cos(2\pi\xi z)\right)dz\ge c_\lambda \int_{\frac{N\ep}{\xi}}^{\infty} w(z)dz.
    \end{equation}
    Fix \(R>1/(4N)+\ep\). Without loss of generality, assume \(N<1/(4\ep)\). On the one hand, using the nonincreasing assumption on \(w\) and \cite[Lemma~2]{du2018stability} we have
    \[\begin{split}
        \int_{\frac{1}{4\xi}}^{R} w(z)\cos(2\pi\xi z)dz&=\frac{1}{2\pi\xi}\int_{\frac{\pi}{2}}^{2\pi \xi R} w\left(\frac{z}{2\pi\xi}\right)\cos(z)dz\\
        &=-\frac{1}{2\pi\xi}\int_{0}^{2\pi \xi R-\frac{\pi}{2}} w\left(\frac{1}{2\pi\xi}\left(z+\frac{\pi}{2}\right)\right)\sin(z)dz\le 0,
    \end{split}\]
    thus 
    \begin{equation}\label{1dest1}
        \begin{split}
            \int_{\frac{1}{4\xi}}^{R} w(z)dz&=\int_{\frac{1}{4\xi}}^{R} w(z)(1-\cos(2\pi\xi z))dz+\int_{\frac{1}{4\xi}}^{R} w(z)\cos(2\pi\xi z)dz\\
            &\le \int_{\frac{1}{4\xi}}^{R} w(z)(1-\cos(2\pi\xi z))dz.
        \end{split}
    \end{equation}
    On the other hand, again by nonincreasing property of \(w\) one has
    \begin{equation}\label{1dest2}
        \begin{split}
            \int_{\frac{N\ep}{\xi}}^{\frac{1}{4\xi}} w(z)dz&= \xi^{-1}\int_{N\ep}^{\frac{1}{4}} w(\xi^{-1}z)dz\le \xi^{-1}\frac{1-4N\ep}{2N\ep} \int_{\frac{N\ep}{2}}^{N\ep} w(\xi^{-1}z)dz\\
            &\le \xi^{-1}\frac{1-4N\ep}{2N\ep}\frac{1}{1-\cos\left(N\ep\right)} \int_{\frac{N\ep}{2}}^{N\ep} w(\xi^{-1}z)(1-\cos(2\pi z))dz\\
            &= \frac{1-4N\ep}{2N\ep}\frac{1}{1-\cos\left(N\ep\right)} \int_{\frac{N\ep}{2\xi}}^{\frac{N\ep}{\xi}} w(z)(1-\cos(2\pi \xi z))dz\\
            &\le \frac{1-4N\ep}{2N\ep}\frac{1}{1-\cos\left(N\ep\right)} \int_{0}^{\frac{1}{4\xi}} w(z)(1-\cos(2\pi\xi z))dz
        \end{split}
    \end{equation}
    Combining \eqref{1dest1} and \eqref{1dest2} yields 
    \[I(\xi)\ge\int_{0}^{R} w(z)\left(1-\cos(2\pi\xi z)\right)dz\ge c_\lambda \int_{\frac{N\ep}{\xi}}^{R} w(z)dz,\quad \xi>N,\]
    where 
    \[c_\lambda=\left(1+\frac{1-4N\ep}{2N\ep\left(1-\cos\left(N\ep\right)\right)}\right)^{-1}.\] 
    Letting \(R\to+\infty\) yields \eqref{estI_1d}.

    \textbf{Case II: \(d\ge 2\)}. Let \(\mub=(\mu_1,\dots,\mu_d)\). 
    Notice that \(I(\xib)\) maintains if \(\mub\) is replaced by \(-\mub\). Without loss of generality, we assume \(\mu_1\ge 0\). Denote \(\mub=(\mu_1,\mub')\) where \(\mub'\in\R^{d-1}\). 

    \textbf{Case II(i): \(\mu_1=0\)}. This corresponds to the case where \(\xib\perp\nub\). Fix \(R>\ep/2\). By doing a change of variable and using the Fubini-Tonelli theorem, one obtains
    \begin{equation}\label{estI_mu1eq0}
        \begin{aligned}
        I(\xib)&\ge\int_{|\zb'|<R} \left(1-\cos(2\pi|\xib|\mub'\cdot\zb')\right)\int_{0}^{\sqrt{R^2-|\zb'|^2}} \frac{z_1}{|\zb|}\overline{w}(|\zb|)dz_1d\zb'\\
        &=\int_{|\zb'|<R} \left(1-\cos(2\pi|\xib|\mub'\cdot\zb')\right)\int_{|\zb'|}^{R} \overline{w}(r)drd\zb'\\
        &=\int_{0}^{R} \overline{w}(r)\int_{|\zb'|<r}\left(1-\cos(2\pi|\xib|\mub'\cdot\zb')\right)d\zb'dr\\
        &\ge\int_{\frac{N\ep}{|\xib|}}^{R} \overline{w}(r)\int_{|\zb'|<r}\left(1-\cos(2\pi|\xib|\mub'\cdot\zb')\right)d\zb'dr\\
        &= \int_{\frac{N\ep}{|\xib|}}^{R} \overline{w}(r)\int_{B^{d-1}_r(\bm 0)}\left(1-\cos(2\pi|\xib|y_1)\right)d\yb dr\\
        &= \int_{\frac{N\ep}{|\xib|}}^{R} r^{d-1}\overline{w}(r)\int_{B^{d-1}_1(\bm 0)}\left(1-\cos(2\pi r|\xib|y_1)\right)d\yb dr,
        \end{aligned}
    \end{equation}
    where in the last two steps we used change-of-variable formula. 

    Denote \[J(r):=\int_{B^{d-1}_1(\bm 0)}\left(1-\cos(2\pi ry_1)\right)d\yb.\]
    It suffices to show \(J(r)\) has a positive lower bound \(c_J=c_J(N\ep,d)\) over \(\left[N\ep,\infty\right)\) to prove \eqref{estI}. Indeed, using \eqref{estI_mu1eq0} and this positive lower bound one obtains 
    \[I(\xib)\ge \int_{\frac{N\ep}{|\xib|}}^{R} r^{d-1}\overline{w}(r)J(r|\xib|)dr\ge c_J\int_{\frac{N\ep}{|\xib|}}^{R} r^{d-1}\overline{w}(r)dr=\frac{c_J}{\om_{d-1}}\int_{\frac{N\ep}{|\xib|}<|\zb|<R}w(\zb)d\zb\]
    for any \(|\xib|>N\) with \(\xib\in\R^d\). Letting \(R\to+\infty\) gives \eqref{estI}.
    Now we focus on the proof of \(J(r)\ge c_J\) for \(r\ge N\ep\).

    Note that for any \(r>0\), \(J(r)>0\) as the integrand is nonnegative. By a change of variable, one obtains
    \[\begin{split}
        J(r)&=\int_{0}^{1} (1-\cos(2\pi r y_1)) \int_{|\yb'|<\sqrt{1-y_1^2}} d\yb'dy_1\\
        &=V_{d-2}\int_{0}^{1} (1-\cos(2\pi r y_1))(1-y_1^2)^{\frac{d-2}{2}}dy_1\\
        &\ge V_{d-2}\left(\frac{\sqrt{3}}{2}\right)^{d-2}\int_{0}^{\frac{1}{2}} (1-\cos(2\pi r y))dy\\
        &\to V_{d-2}\left(\frac{\sqrt{3}}{2}\right)^{d-2}\frac{1}{2}>0,\quad r\to +\infty,
    \end{split}\]
    where in the last step we used Riemann-Lebesgue lemma. Since \(J(r)\) is continuous, there exists \(c_J=c_J(N\ep,d)>0\) such that \(J(r)\ge c_J\) for \(r\ge N\ep\). 
    This concludes the proof of Case II(i).

    \textbf{Case II(ii): \(\mu_1>0\)}. By \Cref{prop:orthobasis} in \Cref{subsec:Appendix2}, there exist \(\vb_1,\dots,\vb_{d-1}\in\S^{d-1}\) such that 
    \[\vb_k^{(1)}:=\vb_k\cdot \eb_1\ge 0,\quad 1\le k\le d-1,\]
    and 
    \(\vb_1,\dots,\vb_{d-1},\mub\) form an orthonormal basis for \(\R^d\). Moreover,  
    \[\vb_k^{(1)}=\frac{\mu_1^{k-1}|\mu_{k+1}|}{\sqrt{\mu_1^2+\cdots+\mu_k^2}\sqrt{\mu_1^2+\cdots+\mu_{k+1}^2}},\quad 1\le k\le d-1.\]
    Denote \(A_{\mub}:=(\vb_1,\dots,\vb_{d-1},\mub)\). Then \(A_{\mub}\) is a \(d\times d\) orthogonal matrix. Consider the following change of variable \(\zb=A_{\mub}\yb\) where \(\yb\in\R^d\). Then \[z_1=\mu_1 y_d+\sum_{k=1}^{d-1} \vb_k^{(1)}y_k>0\] 
    if \(y_1,\dots,y_d>0\). Thus, for \(R>1\),
    \begin{equation}\label{changeztoy}
        \begin{aligned}
            I(\xib)&\ge \int_{z_1>0,\,|\zb|<R}\frac{z_1}{|\zb|}w(\zb)(1-\cos(2\pi |\xib|\mub\cdot\zb))d\zb\\
            &\ge \int_{y_1>0,\dots,y_d>0,\,|\yb|<R} \left(\mu_1y_d + \sum_{k=1}^{d-1}\vb_k^{(1)}y_k\right)\frac{1-\cos(2\pi|\xib|y_d)}{|\yb|}w(\yb)d\yb.
        \end{aligned}
    \end{equation}

    \textbf{Case II(ii)(a):}
    If \(\mu_1\) is the largest among \(\{\vb_1^{(1)},\dots,\vb_{d-1}^{(1)},\mu_1\}\), then \(\mu_1\ge 1/\sqrt{d}\) and \eqref{changeztoy} yields
    \[\begin{split}
        I(\xib)&\ge \frac{1}{\sqrt{d}}\int_{y_1>0,\dots,y_d>0,\,|\yb|<R} y_d \frac{1-\cos(2\pi|\xib|y_d)}{|\yb|}w(\yb)d\yb\\
        &=\frac{1}{\sqrt{d}}\int_{0}^{R} y_d(1-\cos(2\pi|\xib|y_d))\int_{y_1>0,\dots,y_{d-1}>0,\,y_1^2+\dots+y_{d-1}^2<R^2-y_d^2}\frac{w(\yb)}{|\yb|}d\tilde{\yb}dy_d
    \end{split}\]
    where \(\tilde{\yb}=(y_1,\dots,y_{d-1})\). Using polar coordinates one may compute
    \[\begin{split}
        &\int_{y_1>0,\dots,y_{d-1}>0,\,y_1^2+\dots+y_{d-1}^2<R^2-y_d^2}\frac{w(\yb)}{|\yb|}d\tilde{\yb}\\
        = & \frac{1}{2^{d-1}} \int_{|\tilde{\yb}|<\sqrt{R^2-y_d^2}} \frac{1}{(y_d^2+|\tilde{\yb}|^2)^{\frac{1}{2}}}\overline{w}\left((y_d^2+|\tilde{\yb}|^2)^{\frac{1}{2}}\right)d\tilde{\yb}\\
        =& \frac{\om_{d-2}}{2^{d-1}}\int_{0}^{\sqrt{R^2-y_d^2}} \frac{r^{d-2}}{(y_d^2+r^2)^{\frac{1}{2}}}\overline{w}\left((y_d^2+r^2)^{\frac{1}{2}}\right)dr\\
        =& \frac{\om_{d-2}}{2^{d-1}}\int_{y_d}^{R} \overline{w}\left(r\right)\left(r^2-y_d^2\right)^{\frac{d-3}{2}}dr.
    \end{split}\]
    Therefore, 
    \begin{equation}\label{estI_mu1ge0_mu1large}
        \begin{aligned}
            I(\xib)&\ge \frac{1}{\sqrt{d}}\frac{\om_{d-2}}{2^{d-1}} \int_{0}^{R} y_d(1-\cos(2\pi|\xib|y_d))\int_{y_d}^{R} \overline{w}\left(r\right)\left(r^2-y_d^2\right)^{\frac{d-3}{2}}drdy_d\\
            &=\frac{1}{\sqrt{d}}\frac{\om_{d-2}}{2^{d-1}} \int_{0}^{R} \overline{w}(r)\int_{0}^{r} y_d(1-\cos(2\pi|\xib|y_d))(r^2-y_d^2)^{\frac{d-3}{2}} dy_d dr\\
            &=\frac{1}{\sqrt{d}}\frac{\om_{d-2}}{2^{d-1}} \int_{0}^{R} r^{d-1}\overline{w}(r)\int_{0}^{1} t(1-\cos(2\pi r|\xib|t))(1-t^2)^{\frac{d-3}{2}} dt dr\\
            &\ge \frac{1}{\sqrt{d}}\frac{\om_{d-2}}{2^{d-1}} \int_{\frac{N\ep}{|\xib|}}^{R} r^{d-1}\overline{w}(r)H(r|\xib|) dr
        \end{aligned}
    \end{equation}
    where \(H:(0,\infty)\to (0,\infty)\) is defined by
    \begin{equation}\label{defH}
        H(r):=\int_{0}^{1} t(1-\cos(2\pi rt))(1-t^2)^{\frac{d-3}{2}} dt.
    \end{equation}
    It suffices to show there exists \(c_H=c_H(N\ep,d)>0\) such that \(H(r)\ge c_H\) for \(r\ge N\ep\) to show \eqref{estI}. Indeed, from \eqref{estI_mu1ge0_mu1large} and the lower bound \(c_H\) one obtains 
    \[I(\xib)\ge c_H\frac{1}{\sqrt{d}}\frac{\om_{d-2}}{2^{d-1}} \int_{\frac{N\ep}{|\xib|}}^{R} r^{d-1}\overline{w}(r)dr=\frac{c_H}{\om_{d-1}}\frac{1}{\sqrt{d}}\frac{\om_{d-2}}{2^{d-1}}\int_{\frac{N\ep}{|\xib|}<|\zb|<R}w(\zb)d\zb\]
    for any \(|\xib|>N\) with \(\xib\in\R^d\). Letting \(R\to+\infty\) yields \eqref{estI}. Noticing that \(H(r)\) is continuous and positive for any \(r>0\), it suffices to show 
    \[\lim_{r\to +\infty} H(r)>0\]
    to conclude the existence of the positive lower bound \(c_H\). The proof uses again Riemann-Lebesgue lemma as follows.

    When \(d=2\), since the function \(t\mapsto t(1-t^2)^{-\frac{1}{2}}\) is increasing over \((0,1)\), one has
    \[\begin{split}
        H(r)&\ge \frac{1}{\sqrt{3}} \int_{\frac{1}{2}}^{1} (1-\cos(2\pi rt))dt\to \frac{1}{2\sqrt{3}},\quad r\to +\infty. 
    \end{split}\]

    When \(d\ge 3\), since the function \(t\mapsto (1-t^2)^{\frac{d-3}{2}}\) is nonincreasing over \((0,1)\), one has
    \[\begin{split}
        H(r) &\ge \left(\frac{\sqrt{3}}{2}\right)^{d-3}\int_{0}^{\frac{1}{2}} t(1-\cos(2\pi rt))dt\to \left(\frac{\sqrt{3}}{2}\right)^{d-3} \frac{1}{8},\quad r\to +\infty. 
    \end{split}\]
    Then the proof for Case II(ii)(a) is done.

    \textbf{Case II(ii)(b):} 
    If \(\vb_j^{(1)}\) is the largest among \(\{\vb_1^{(1)},\dots,\vb_{d-1}^{(1)},\mu_1\}\) where \(1\le j\le d-1\), then \(\vb_j^{(1)}\ge 1/\sqrt{d}\). Denoting \(\tilde{\yb}:=(y_1,\dots,\hat{y}_j,\dots,y_d)=(\tilde{y}_1,\dots,\tilde{y}_{d-1})\in\R^{d-1}\) (with \(\tilde{y}_{d-1}=y_d\)) we continue from \eqref{changeztoy} to compute 

    \[\begin{split}
        I(\xib)&\ge \frac{1}{\sqrt{d}} \int_{y_1>0,\dots,y_d>0,\,|\yb|<R} y_j \frac{1-\cos(2\pi|\xib|y_d)}{|\yb|}w(\yb)d\yb\\
        &=\frac{1}{\sqrt{d}} \int_{\tilde{y}_1>0,\dots,\tilde{y}_{d-1}>0,\,|\tilde{\yb}|<R} (1-\cos(2\pi|\xib|\tilde{y}_{d-1})) \int_{0}^{\sqrt{R^2-|\tilde{\yb}|^2}} \frac{y_j}{|\yb|}\overline{w}(|\yb|)dy_j d\tilde{\yb}\\
        &=\frac{1}{\sqrt{d}} \int_{\tilde{y}_1>0,\dots,\tilde{y}_{d-1}>0,\,|\tilde{\yb}|<R} (1-\cos(2\pi|\xib|\tilde{y}_{d-1})) \int_{|\tilde{\yb}|}^{R} \overline{w}(r)dr d\tilde{\yb}\\
        &=\frac{1}{\sqrt{d}} \frac{1}{2^{d-1}}\int_{B^{d-1}_{1}(\bm 0)} (1-\cos(2\pi |\xib|z_1))\int_{|\zb|}^R \overline{w}(r)drd\zb\\
        &=\frac{1}{\sqrt{d}} \frac{1}{2^{d-1}}\int_{0}^R\overline{w}(r)\int_{B^{d-1}_{r}(\bm 0)} (1-\cos(2\pi |\xib|z_1))d\zb dr\\
        &\ge \frac{1}{\sqrt{d}} \frac{1}{2^{d-1}}\int_{\frac{N\ep}{|\xib|}}^R r^{d-1}\overline{w}(r)\int_{B^{d-1}_{1}(\bm 0)} (1-\cos(2\pi r|\xib|z_1))d\zb dr.\\
    \end{split}\]
    Note that the above calculation is similar to \eqref{estI_mu1eq0} and the final expression is just a constant multiple of that of \eqref{estI_mu1eq0}. Therefore, following the proof in Case(II)(i) after \eqref{estI_mu1eq0} we complete the whole proof.
\end{proof}

\begin{proof}[Proof of \Cref{prop:locallim_ptws_lp}]
    We first prove the pointwise convergence result and then use generalized dominated convergence theorem to prove the \(L^p\) convergence. For simplicity we focus on \(\cG_\del^{\nub}\).

    Without loss of generality, we may assume that \(u\in C^2_c(\R^d)\) in the proof. The result for \(u\in C^1_c(\R^d)\) can then be obtained using a standard density argument so we omit the details. Since \(u\in C^2_c(\R^d)\), there exists \(R>0\) such that \(\supp u\subset B_{R}(\bm 0)=:B_{R}\). Recall that 
    \[\cG^{\nub}_\del u(\xb)=\int_{\R^d} \chi_{\nub}(\zb)\frac{\zb}{|\zb|}w_\del(\zb)(u(\xb+\zb)-u(\xb))d\zb.\] 
    We discuss two cases \(|\xb|\le 2R\) and \(|\xb|>2R\).

    \textbf{Case (I): }\(|\xb|\le 2R\). Then for \(|\zb|>3R\) one has \(|\xb+\zb|\ge|\zb|-|\xb|>R\) and thus \(u(\xb+\zb)=0\). Note that we have Taylor expansion 
    \[u(\xb+\zb)-u(\xb)= D u(\xb)\zb +\int_{0}^{1} \zb^T D^2u(\xb+t\zb)\zb dt.\]
    Hence,
    \[\begin{split}
        &\cG^{\nub}_\del u(\xb)\\
        =&\int_{B_{3R}} \chi_{\nub}(\zb)\frac{\zb}{|\zb|}w_\del(\zb)(u(\xb+\zb)-u(\xb))d\zb + \int_{B_{3R}^c} \chi_{\nub}(\zb)\frac{\zb}{|\zb|}w_\del(\zb)(-u(\xb))d\zb\\
        = & \left(\int_{B_{3R}} \chi_{\nub}(\zb)|\zb| w_\del(\zb)\frac{\zb}{|\zb|}\otimes \frac{\zb}{|\zb|}d\zb\right) \nabla u(\xb) + \int_{B_{3R}^c} \chi_{\nub}(\zb)\frac{\zb}{|\zb|}w_\del(\zb)(-u(\xb))d\zb\\
        &\quad\quad + \int_{B_{3R}} \chi_{\nub}(\zb)\frac{\zb}{|\zb|}w_\del(\zb)\left(\int_{0}^{1} \zb^T D^2u(\xb+t\zb)\zb dt\right)d\zb \\
        =&: I_1 + I_2 + I_3.
    \end{split}\]
    Since \(w_\del\) is radial, one can compute as in \cite[Remark~2.2]{han2023nonlocal} that
    \[\int_{B_{3R}} \chi_{\nub}(\zb)|\zb| w_\del(\zb)\frac{\zb}{|\zb|}\otimes \frac{\zb}{|\zb|}d\zb=\frac{1}{2d}\left(\int_{B_{3R}}|\zb|w_\del(\zb)d\zb\right) \I_d\to \I_d,\quad\del \to 0^+,\]
    where \(\I_d\) is the \(d\times d\) identity matrix and we used \eqref{eq:kernelassumption_3}. Hence \(I_1\to \nabla u(\xb)\) as \(\del\to 0\).

    For \(I_2\) and \(I_3\), we show that they converge to \(0\) as \(\del\to 0\). Indeed, as \(\del\to 0\), by \cref{eq:kernelassumption_2,eq:kernel_2ndMoment2zero},
    \[|I_2|\le \|u\|_{C^0(\R^d)}\int_{B_{3R}^c} w_\del(\zb)d\zb\to 0\]
    and 
    \[|I_3|\le \|u\|_{C^2(\R^d)}\int_{B_{3R}} |\zb|^2w_\del(\zb)d\zb\to 0.\]
    Combining all the estimates for \(I_1\), \(I_2\) and \(I_3\) yields 
    \[\cG^{\nub}_\del u(\xb)\to \nabla u(\xb),\quad \del\to 0,\ \forall |\xb|<2R.\]
    Moreover, there exist constants \(\del_0\in (0,1)\) and \(C_0=C_0(\|u\|_{C^2(\R^d)})\) such that \(|\cG^{\nub}_\del u(\xb)|\le C_0\) for \(|\xb|\le 2R\) and \(\del\in (0,\del_0)\).
    
    \textbf{Case (II): }\(|\xb|> 2R\). Then \(u(\xb)= 0\) and 
    \[\cG^{\nub}_\del u(\xb) = \int_{\R^d} \chi_{\nub}(\zb)\frac{\zb}{|\zb|}w_\del(\zb)u(\xb+\zb)d\zb=\int_{B_{R}(-\xb)} \chi_{\nub}(\zb)\frac{\zb}{|\zb|}w_\del(\zb)u(\xb+\zb)d\zb.\]
    Then for \(|\xb|>2R\),
    \[|\cG^{\nub}_\del u(\xb) |\le \|u\|_{C^0(\R^d)}\int_{B_R(-\xb)}w_\del(\zb)d\zb.\]
    Define \(g_\del:\R^d\to\R\) by
    \[g_\del(\xb)=\left\{\begin{aligned}
        &C_0,\quad |\xb|\le 2R,\\
        &\|u\|_{C^0(\R^d)} \int_{B_R(-\xb)}w_\del(\zb)d\zb,\quad |\xb|> 2R.
    \end{aligned}\right.\]
    Then \(|\cG^{\nub}_\del u(\xb) |\le g_\del(\xb)\) for any \(\xb\in\R^d\). 
    Note that for \(|\xb|>2R\), \(B_R(-\xb)\cap B_R(\bm 0)=\emptyset\), hence 
    \[|\cG^{\nub}_\del u(\xb) |\le g_\del(\xb)\le \|u\|_{C^0(\R^d)} \int_{B_R^c}w_\del(\zb)d\zb\to 0=\nabla u(\xb),\quad\del\to 0.\]
    Combining the pointwise limit in Case (I) we obtain 
    \[\cG^{\nub}_\del u(\xb)\to \nabla u(\xb),\quad \del\to 0,\ \forall \xb\in\R^d.\]
    It remains to show for \(p\in [1,\infty]\), \(\cG^{\nub}_\del u\to \nabla u\) in \(L^p(\R^d)\). For \(p=\infty\) this follows from the estimates above. Now suppose \(1\le p<\infty\). Notice that \(g_\delta(\xb)\to C_0\chi_{B_{2R}}(\xb)=:g(\xb)\) as \(\del\to 0\) for any \(\xb\in\R^d\), to apply the generalized dominated convergence theorem, we need to show that \(\|g_\del\|_{L^p(\R^d)}\to \|g\|_{L^p(\R^d)}\) as \(\del\to 0\). In fact, we will show \(\|g_\del-g\|_{L^p(\R^d)}\to 0\) as \(\del\to 0\). Note that \[\|g_\del-g\|_{L^p(\R^d)}^p = \|u\|_{C^0(\R^d)}^p \int_{B_{2R}^c} \left(\int_{B_R(-\xb)}w_\del(\zb)d\zb\right)^p d\xb.\] 
    Since \(B_R(-\xb)\subset B_R^c\) for \(|\xb|>2R\), one has
    \[\begin{split}
        &\int_{B_{2R}^c} \left(\int_{B_R(-\xb)}w_\del(\zb)d\zb\right)^p d\xb\\
        =& \int_{B_{2R}^c} \left(\int_{\R^d}\chi_{B_R(-\xb)}(\zb)w_\del(\zb)d\zb\right)^p d\xb \\
        = & \int_{B_{2R}^c} \left(\int_{B_R(-\xb)}w_\del(\zb)d\zb\right)^{p-1}\int_{B_R(-\xb)}w_\del(\zb)d\zb d\xb\\
        \le &\left(\int_{B_R^c}w_\del(\zb)d\zb\right)^{p-1} \int_{B_{2R}^c} \int_{B_R^c}\chi_{B_R(-\zb)}(\xb)w_\del(\zb)d\zb d\xb\\
        = &\left(\int_{B_R^c}w_\del(\zb)d\zb\right)^{p-1} \int_{B_R^c} \int_{B_{2R}^c\cap B_R(-\zb)} d\xb w_\del(\zb)d\zb\\
        \le & |B_R|\left(\int_{B_R^c}w_\del(\zb)d\zb\right)^{p} \to 0,\quad\del\to 0,
    \end{split}\]
    where we used Tonelli-Fubini theorem and \eqref{eq:kernelassumption_2}. This implies \(\|g_\del-g\|_{L^p(\R^d)}\to 0\) as \(\del\to 0\) and by generalized dominated convergence theorem \(\cG^{\nub}_\del u\to \nabla u\) in \(L^p(\R^d)\) for \(p\in [1,\infty)\).

\end{proof}

\subsection{An orthonormal basis in \(\R^d\)}
\label{subsec:Appendix2}
\begin{proposition}\label{prop:orthobasis}
    Let \(d\ge 2\) and \(\mub=(\mu_1,\dots,\mu_d)\in\S^{d-1}\) with \(\mu_1>0\). Then there exist \(\vb_1,\dots,\vb_{d-1}\in\S^{d-1}\) such that 
    \[\vb_k^{(1)}:=\vb_k\cdot \eb_1\ge 0,\quad 1\le k\le d-1,\]
    and 
    \(\vb_1,\dots,\vb_{d-1},\mub\) form an orthonormal basis for \(\R^d\). Moreover, one may choose 
    \begin{equation}\label{vkformula}
        \vb_k^{(1)}=\frac{\mu_1|\mu_{k+1}|}{\sqrt{\mu_1^2+\cdots+\mu_k^2}\sqrt{\mu_1^2+\cdots+\mu_{k+1}^2}},\quad 1\le k\le d-1.
    \end{equation}
\end{proposition}
\begin{proof}
    We prove it by induction on \(d\). 
    
    If \(d=2\), then one may choose 
    \[\vb_1=\left\{\begin{aligned}
        &(\mu_2,-\mu_1),\quad\text{if}\ \mu_2\ge 0,\\
        &(-\mu_2,\mu_1),\quad\text{if}\ \mu_2<0.
    \end{aligned}\right.\]
    Then \(\vb_1,\mu\) form an orthonormal basis for \(\R^2\) and \(\vb_1^{(1)}=|\mu_2|\) as desired.

    Suppose the conclusion holds for some \(d\ge 2\). Given \(\mub=(\mu_1,\dots,\mu_{d+1})\in\S^{d}\) with \(\mu_1>0\), one has \(|\mu_{d+1}|<1\) and 
    \[\tilde{\mub}:=\frac{1}{\sqrt{1-\mu_{d+1}^2}}(\mu_1,\dots,\mu_d)\in \S^{d-1}.\]
    By the inductive hypothesis, there exist \(\wb_1,\dots,\wb_{d-1}\in\S^{d-1}\) such that \(\wb_k^{(1)}\ge 0\) for \(1\le k\le d-1\) and \(\wb_1,\dots,\wb_{d-1},\tilde{\mub}\) form an orthonormal basis for \(\R^d\).
    Define \(\vb_k:=(\wb_k,0)\in \S^{d}\) for \(1\le k\le d-1\) and \[\vb_d:=\left(|\mu_{d+1}|\tilde{\mub},-\sgn(\mu_{d+1})\sqrt{1-\mu_{d+1}^2}\right)\in\S^d.\]
    Then it is clear that \(\vb_1,\dots,\vb_d,\mub\) form an orthonormal basis for \(\R^{d+1}\). Moreover, if \(\wb_1^{(1)},\dots,\wb_{d-1}^{(1)}\) are assumed to have form in \eqref{vkformula}, then it can be checked that \(\vb_1^{(1)},\dots,\vb_d^{(1)}\) satisfy \eqref{vkformula} as well. By induction, the proposition holds for any \(d\ge 2\).
\end{proof}

\subsection{A useful limit concerning the Fourier symbol with respect to Riesz fractional kernel}
\label{subsec:Appendix3}
\begin{proposition}\label{prop:symblim}
    The following equality holds:
    \begin{equation}
        \lim_{\del\to 0} \int_{0}^{\infty} \frac{\del}{z^{2-\del}}\left(e^{2\pi i z}-1\right)dz = 2\pi i,
    \end{equation}
    that is, 
    \begin{equation}
        \lim_{\del\to 0} \int_{0}^{\infty} \frac{\del}{z^{2-\del}}\left(\cos(2\pi z)-1\right)dz = 0,\quad \lim_{\del\to 0} \int_{0}^{\infty} \frac{\del}{z^{2-\del}}\sin(2\pi z)dz = 2\pi.
    \end{equation}
    In addition, these three equalities hold with the upper limit \(\infty\) replaced by \(1\).
\end{proposition}
\begin{proof}
    Note that the last assertion follows from the estimate
    \[\left|\int_{1}^{\infty} \frac{\del}{z^{2-\del}}\left(e^{2\pi i z}-1\right)dz\right|\le 2\int_{1}^{\infty} \frac{\del}{z^{2-\del}}dz\le 2\del\int_{1}^{\infty} \frac{1}{z^{3/2}}dz\]
    for \(\del\in (0,1/2)\), provided the following two equalities 
    \begin{equation}\label{coslim}
        \lim_{\del\to 0} \int_{0}^{1} \frac{\del}{z^{2-\del}}\left(\cos(2\pi z)-1\right)dz = 0
    \end{equation}
    and 
    \begin{equation}\label{sinlim}
        \lim_{\del\to 0} \int_{0}^{1} \frac{\del}{z^{2-\del}}\sin(2\pi z)dz = 2\pi
    \end{equation}
    hold. Using the inequality \(|1-\cos(x)|\le x^2/2\) for any \(x\in\R\), it is clear that \eqref{coslim} holds as 
    \[\left|\int_{0}^{1} \frac{\del}{z^{2-\del}}\left(\cos(2\pi z)-1\right)dz\right|\le 2\pi^2\del \int_{0}^{1} z^\del dz\le 2\pi^2\del\to 0,\quad\del\to 0.\]
    To show \eqref{sinlim}, first notice that \(\int_{0}^{1}\frac{\del}{z^{1-\del}}dz = 1\). Since \(\lim_{x\to 0^+} \sin(x)/x=1\), for any \(\ep>0\), there exists \(N_0=N_0(\ep)\in (0,1)\) such that for any \(z\in (0,N_0)\),
    \[\left|\frac{\sin(2\pi z)}{2\pi z} - 1\right|<\ep.\]
    Then 
    \[\left|\int_{0}^{N_0}\frac{\del}{z^{1-\del}}\left(\frac{\sin(2\pi z)}{2\pi z} - 1\right)dz\right|<\ep\int_{0}^{1} \frac{\del}{z^{1-\del}}dz=\ep.\]
    On the other hand, for \(\del\in (0,1)\) one has
    \[\left|\int_{N_0}^{1} \frac{\del}{z^{1-\del}}\left(\frac{\sin(2\pi z)}{2\pi z} - 1\right)dz\right|\le 2 \int_{N_0}^{1}\frac{\del}{z}dz=-2\del \ln(N_0)\to 0,\quad\del\to 0.\]
    Thus, there exists \(\del_0=\del_0(\ep)\) such that for \(\del\in (0,\del)\)
    \[\left|\int_{0}^{1} \frac{\del}{z^{2-\del}}\sin(2\pi z)dz-2\pi\right|=2\pi\left|\int_{0}^{1}\frac{\del}{z^{1-\del}}\left(\frac{\sin(2\pi z)}{2\pi z} - 1\right)dz\right|<2\ep.\]
    This proves \eqref{sinlim} and the whole proposition.
\end{proof}

\subsection{An approximation result}
\begin{lemma}\label{lem:one_sided_approx}
    Let $\Om$ be a bounded polygonal domain in $\R^d$ and \(\{\mathcal{T}_h\}_{h>0}\) be a quasi-uniform mesh of size \(h\) on $\Om$. Given \(\alpha,\beta\in C(\overline{\Om})\) 
    such that $\alpha(\xb)<\beta(\xb)$ for all $\xb\in\overline{\Om}$, 
    define \(Z_{\mathrm{ad}}:=\{f\in L^2(\Om):\ \alpha(\xb)\le f(\xb)\le \beta(\xb),\ a.e.\ x\in\Om\}\) and denote \(Z_h\) the piecewise constant functions with respect to the mesh \(\{\mathcal{T}_h\}_{h>0}\), i.e.,
    \[Z_h:=\{z_h\in L^\infty(\Om)|\ z_h|_T\in \mathcal{P}_0(T),\ \forall T\in \mathcal{T}_h\}.\]
    Then for any $q\in Z_{\mathrm{ad}}$, there exists $\{q_h\}_{h>0}\subset Z_{h}\cap Z_{\mathrm{ad}}$ such that $q_h\to q$ in $L^2(\Om)$.
\end{lemma}
\begin{proof}
    Without loss of generality, we assume that $\alpha(\xb)<\beta(\xb)$ for all $\xb\in\overline{\Om}$. Given $q\in L^2(\Om)$, define the piecewise constant function $p_h\in L^2(\Om)$ by
    \[
    p_h(\xb):=\frac{1}{|T|}\int_T q(\yb)d\yb,\quad\text{if}\ \xb\in \text{int}(T),\ T\in\cT_h.
    \]
    where $|T|$ denotes the Lebesgue measure of $T$. It is well-known that $p_h\to q$ in $L^2(\Om)$. However, $p_h$ is not necessarily in $Z_h$. Consider another piecewise constant $q_h\in L^2(\Om)$ given by
    \[
    q_h(\xb):=\min\left\{\max\left\{p_h(\xb),\sup_{\yb\in \overline{T}}\alpha(\yb)\right\},\inf_{\yb\in \overline{T}}\beta(\yb)\right\},\quad\text{if}\ \xb\in \text{int}(T),\ T\in\cT_h.
    \]
    It is clear that $q_h\in Z_{h}\cap Z_{\mathrm{ad}}$. To show $q_h\to q$ in $L^2(\Om)$, it suffices to show that $\|q_h-p_h\|_{L^2(\Om)}\to 0$ as $h\to 0$. To that end, consider the partition $\cT_h=\cT^1_h\cup \cT_h^2\cup \cT_h^3$ of $\cT_h$ with respect to $q$ defined by
    \[\begin{split}
        \cT_h^1&:=\left\{T\in\cT_h:\frac{1}{|T|}\int_T q(\yb)d\yb\le \sup_{\yb\in \overline{T}}\alpha(\yb) \right\},\\
        \cT_h^2&:=\left\{T\in\cT_h:\frac{1}{|T|}\int_T q(\yb)d\yb\ge \inf_{\yb\in \overline{T}}\beta(\yb) \right\},\\
        \cT_h^3&:=\left\{T\in\cT_h:\sup_{\yb\in \overline{T}}\alpha(\yb)<\frac{1}{|T|}\int_T q(\yb)d\yb< \inf_{\yb\in \overline{T}}\beta(\yb) \right\}.
    \end{split}
    \]
    Then for $T\in  \cT_h^3$, $p_h(\xb) = q_{h}(\xb)$ for all $x\in T$ and, therefore,   
    \[
    \begin{split}
        &\, \|p_h-q_h\|_{L^2(\Om)}^2\\
        =&\sum_{T\in \cT_h}\int_T |p_h(\xb)-q_h(\xb)|^2d\xb\\
        =&\sum_{T\in\cT_h^1}|T|\left(\frac{1}{|T|}\int_T q(\yb)d\yb- \sup_{\yb\in \overline{T}}\alpha(\yb)\right)^2+\sum_{T\in\cT_h^2}|T|\left(\frac{1}{|T|}\int_T q(\yb)d\yb- \inf_{\yb\in \overline{T}}\beta(\yb)\right)^2\\
        \le&\sum_{T\in\cT_h^1}|T|\left(\sup_{\yb\in \overline{T}}\alpha(\yb)-\inf_{\yb\in \overline{T}}\alpha(\yb)\right)^2+\sum_{T\in\cT_h^2}|T|\left(\sup_{\yb\in \overline{T}}\beta(\yb)- \inf_{\yb\in \overline{T}}\beta(\yb)\right)^2,
    \end{split}
    \]
    where in the last step we used $\alpha(\xb)\le q(\xb)\le \beta(\xb)$ for a.e. $\xb\in\Om$. Since $\alpha$ and $\beta$ are continuous on $\overline{\Om}$, they are uniformly continuous on $\overline{\Om}$. Then for any $\ep>0$, there exists $\tau>0$ such that for any $0<h<\tau$ and $T\in\cT_h$, \[
    \sup_{\yb\in \overline{T}}\alpha(\yb)-\inf_{\yb\in \overline{T}}\alpha(\yb)<\ep,\quad \sup_{\yb\in \overline{T}}\beta(\yb)- \inf_{\yb\in \overline{T}}\beta(\yb)<\ep.
    \]
    Therefore, \[
    \|p_h-q_h\|_{L^2(\Om)}^2\le \ep^2\sum_{T\in\cT_h^1}|T|+\sum_{T\in\cT_h^2}|T|\le \ep^2|\Om|.
    \]
    This implies $\|q_h-p_h\|_{L^2(\Om)}\to 0$ as $h\to 0$. Thus $q_h\to q$ in $L^2(\Om)$ as $h\to 0$.
\end{proof}

\bibliographystyle{abbrv} 

\bibliography{mylibbetterbibtex}

\end{document}